\documentclass{article}
\usepackage{t1enc}
\usepackage[latin1]{inputenc}
\usepackage{graphicx}
\usepackage{amsthm}
\usepackage{amsmath}
\usepackage{amssymb}
\usepackage{amsfonts}
\usepackage{bbm}
\usepackage{mathtools}
\usepackage{sseq}
\usepackage[pagebackref,linktoc=all, hidelinks = true]{hyperref}
\usepackage[curve]{xypic}
\usepackage[left=3.2cm,right=3.2cm,top=3cm,bottom=3cm]{geometry} 
\usepackage{pdfpages}
\usepackage{stmaryrd}
\usepackage{marvosym}
\usepackage{mathrsfs}
\usepackage{tikz}
\usetikzlibrary{cd}
\usepackage{wrapfig}
\usepackage{pigpen}
\usepackage[capitalise]{cleveref}
\usepackage{todonotes}
\usepackage[mathscr]{eucal}
\usepackage{comment}
\usepackage{enumerate}
\usepackage{xcolor}
\usetikzlibrary{patterns}
\usepackage[normalem]{ulem}

\newcommand{\hotimes}{\hat{\otimes}}

\newcommand{\xr}{\xrightarrow}

\newcommand{\mv}{\mathbf{v}}

\DeclareMathOperator{\inv}{inv}

\DeclareMathOperator{\coind}{coind}

\numberwithin{equation}{section}
\newtheorem{thm}[equation]{Theorem}

\newtheorem*{thm*}{Theorem}
\newtheorem*{mthm*}{Main Theorem}

\newtheorem{theorem}[equation]{Theorem}

\newtheorem{lemma}[equation]{Lemma}

\newtheorem{cor}[equation]{Corollary}

\newtheorem{prop}[equation]{Proposition}
\newtheorem{proposition}[equation]{Proposition}

\newtheorem*{conjecture*}{Conjecture}

\newtheorem*{question*}{Question}

\theoremstyle{definition}
\newtheorem{defi}[equation]{Definition}
\newtheorem{construction}[equation]{Construction}

\newtheorem{definition}[equation]{Definition}

 \newtheorem{example}[equation]{Example}
  \newtheorem*{example*}{Example}
 
 \newtheorem{remark}[equation]{Remark}

\newtheorem{rmk}[equation]{Remark}

\theoremstyle{remark}

\newcommand{\G}{\mathbb{G}}

\newcommand{\N}{\mathbb{N}}
\newcommand{\Q}{\mathbb{Q}}

\newcommand{\T}{\mathbf{T}}

\newcommand{\Z}{\mathbb{Z}}

\newcommand{\CC}{\mathcal{C}}
\newcommand{\DD}{\mathcal{D}}

\newcommand{\FF}{\mathcal{F}}

\newcommand{\UU}{\mathcal{U}}
\newcommand{\II}{\mathcal{I}}

\newcommand{\XX}{\mathcal{X}}

\newcommand{\OO}{\mathcal{O}}

\newcommand{\MM}{\mathcal{M}}

\newcommand{\VV}{\mathcal{V}}

\newcommand{\tauexp}[1]{#1}

\newcommand{\tensor}{\otimes}

\DeclareMathOperator{\type}{type}
\DeclareMathOperator{\height}{ht}
\DeclareMathOperator{\uni}{uni}
\DeclareMathOperator{\rank}{rank}

\DeclareMathOperator{\op}{op}

\DeclareMathOperator{\Sp}{Sp}

\DeclareMathOperator{\Sub}{Sub}
\DeclareMathOperator{\Fin}{Fin}

\DeclareMathOperator{\ev}{ev}

\DeclareMathOperator{\Aff}{Aff}
\DeclareMathOperator{\id}{id}

\DeclareMathOperator{\im}{im}

\DeclareMathOperator{\sm}{\wedge}
\DeclareMathOperator{\colim}{colim}

\DeclareMathOperator{\Hom}{Hom}
\DeclareMathOperator{\Spec}{Spec}

\DeclareMathOperator{\map}{map}

\DeclareMathOperator{\QCoh}{QCoh}
\DeclareMathOperator{\SI}{SI}

\DeclareMathOperator{\Mod}{Mod}

\DeclareMathOperator{\rk}{rk}
\DeclareMathOperator{\res}{res}

\DeclareMathOperator{\tr}{tr}

\DeclareMathOperator{\Set}{Set}

\DeclareMathOperator{\Fun}{Fun}

\DeclareMathOperator{\Map}{Map}

\DeclareMathOperator{\supp}{supp}

\newcommand{\markus}[1]{#1}
\newcommand{\lennart}[1]{#1}

\newcommand{\rev}[1]{#1}
\newcommand{\revm}[1]{#1}

\title{Invariant prime ideals in equivariant Lazard rings}
\author{Markus Hausmann and Lennart Meier}
\date{}

\begin{document}

\maketitle

\abstract{Let $A$ be an abelian compact Lie group. In this paper we compute the spectrum of invariant prime ideals of the $A$-equivariant Lazard ring, or equivalently the spectrum of points of the moduli stack of $A$-equivariant formal groups. We further show that this spectrum is homeomorphic to the Balmer spectrum of compact $A$-spectra, with the comparison map induced by equivariant complex bordism homology.}

\setcounter{tocdepth}{1}
\tableofcontents

\section{Introduction}
Let us pose the question: what algebraic input do we need to develop equivariant versions of chromatic homotopy theory? 

Chromatic homotopy theory studies stable homotopy theory through the lens of formal groups, building on Quillen's identification of the complex bordism ring $\pi_*MU$ with the Lazard ring \cite{Qui69, QuillenElementary}. Around the same time, tom Dieck introduced for every compact Lie group $A$ in \cite{tomDieckBordism} an equivariant analog of $MU$, the homotopical $A$-equivariant complex bordism $MU_A$. Letting $A$ be abelian, Cole, Greenlees and Kriz \cite{CGK} \rev{many years later} found the correct notion of an $A$-equivariant formal group law. Recently, the first author generalized work of Hanke--Wiemeler \cite{HankeWiemeler} and showed that $\pi_*^AMU_A$ is indeed the universal ring for $A$-equivariant formal group laws, thus establishing an equivariant analog of Quillen's theorem for the equivariant Lazard ring $L_A$. 

Many structural features of stable homotopy theory can be explained through the chromatic perspective. The central notion of chromatic homotopy theory is that of \emph{height}. Honda classified formal groups over a \rev{separably closed} field of characteristic $p$ in terms of the height $0\leq n\leq \infty$. Thus, the points of the moduli stack of formal groups $\MM_{FG}$ correspond to pairs $(p,n)$ with $n=0$ if and only if $p=0$. Hopkins and Smith \cite{H-S98} showed that the same classification pertains to thick subcategories of finite spectra: Given a finite spectrum $X$, its $MU$-homology $MU_*X$ defines a coherent sheaf over $\MM_{FG}$. Taking the support of $MU_*X$ in the Zariski spectrum $|\MM_{FG}|$ of points, we obtain a support theory on  
\revm{compact} spectra. The thick subcategory theorem states that this support theory is the universal one. In other words, the induced map $|\MM_{FG}|\to \Spec(Sp^c)$ to the Balmer spectrum of \revm{compact} 
spectra (cf.\ \cite{BalmerSpectrum, BalmerSpectraSpectraSpectra}) is a homeomorphism. 

We show the following equivariant \rev{generalization} (a more precise statement of which we give as \cref{thm:support}):
\begin{thm} \label{thm:main} Let $A$ be an abelian compact Lie group. Then the spectrum of points of the moduli stack $\MM^A_{FG}$ of $A$-equivariant formal groups is homeomorphic to the Balmer spectrum of \revm{compact} $A$-spectra, with the comparison map induced by a support theory based on complex bordism homology $(MU_A)_*$.
\end{thm}
This establishes $MU_A$ and the theory of equivariant formal groups as fundamental tools for building equivariant versions of chromatic homotopy theory.

For abelian groups as above, the Balmer spectrum of finite $A$-spectra has been computed completely in the papers \cite{BalmerSanders} (the case $A=C_p$), \cite{BarthelHausmannNaumannNikolausNoelStapleton} (the finite abelian case) and \cite{BarthelGreenleesHausmann} (the general abelian case). In a surprising turn of history, it \revm{was} the algebraic counterpart which had not been computed before.
As a set, both $|\MM^A_{FG}|$ and $\Spec(\Sp_A^c)$ decompose as the disjoint union of one copy of $|\MM_{FG}| \cong \Spec(\Sp^c)$ for every closed subgroup of $A$. Thus, the correct notion of height of an $A$-equivariant formal group $F$ over a field of characteristic $p$ consists of a pair: a height $n$ of a non-equivariant formal group and a closed subgroup $B\subset A$ such that $F$ is induced along the zig-zag $A\to A/B \leftarrow \{1\}$. 

The more subtle information lies in the topology of the spectrum, which encodes on the algebraic level how heights can deform and on the homotopical level the chromatic interdependence of the various geometric fixed points $\Phi^B X$ of a finite $A$-spectrum $X$. 

We will detail our results below in the language of invariant prime ideals. Crucially, we exhibit equivariant lifts $\mv_n$ of the classical $v_n$ and show that in \rev{many cases they provide} a sequence of generators of invariant prime ideals. The non-equivariant $v_n$ play an important role in many of the highlights of chromatic homotopy theory, like the greek-letter construction \cite{Rav86}, the construction of the Morava K-theories or the periodicity theorem \cite{H-S98}, and we hope that our equivariant lifts open the prospect to generalize these to the equivariant context. 

\subsection{Invariant prime ideals and statement of results}
As indicated above, the main theorem can also be stated in terms of invariant prime ideals of the equivariant Lazard ring $L_A$, as we now explain. Similarly to the non-equivariant case, $L_A$ is the ground ring of a flat Hopf algebroid $(L_A,S_A)$, classifying $A$-equivariant formal group laws and their strict isomorphisms. By \cite[Theorem 5.52]{Hau} and the discussion thereafter, this is isomorphic to the graded Hopf algebroid $(\pi_*^A MU_A, \pi_*^AMU_A \tensor MU_A)$ of cooperations. The associated stack is the moduli stack of $A$-equivariant formal groups. Hence, the category of graded $(L_A,S_A)$-comodules is equivalent to the category of quasi-coherent sheaves over $\MM^A_{FG}$. 

Recall that an ideal $I$ of $L_A$ is called \emph{invariant} (in the sense of Hopf algebroids) if it is a sub-comodule, i.e., if $\eta_L(I)S_A=\eta_R(I)S_A$ for the left and right unit $\eta_L,\eta_R\colon L_A\to S_A$. Every invariant \emph{prime} ideal $\mathfrak{p}$ gives rise to a point of the moduli stack of prime ideals via the quotient field of $L_A/\mathfrak{p}$. This defines a map from the set of invariant prime ideals $\Spec^{\inv}(L_A)$ to $|\MM^A_{FG}|$, which we show to be a homeomorphism in \cref{thm:ClassificationOfInvariantPrimes}.

For the non-equivariant Lazard ring, Morava and Landweber \cite[\rev{Theorem 2.7}]{LandweberIdeals} showed that the invariant prime ideals are precisely the ideals $I_{p,n}=(v_0,\hdots,v_{n-1})$ for a prime $p$ and $n\in \mathbb{N}\cup \{\infty\}$ (with $I_{p,0}$ being the $0$-ideal for all $p$).

To describe the invariant prime ideals in the equivariant case we recall that $L_A$ contains universal Euler classes $e_V$ for all characters $V\in A^*$, and that equivariant Lazard rings are contravariantly functorial in continuous group homomorphisms. In particular, all equivariant Lazard rings are algebras over the non-equivariant Lazard ring.

Then, given a non-equivariant invariant prime ideal $I_{p,n}$ and a closed subgroup $B$ of $A$ we obtain an invariant prime ideal $I^A_{B,p,n}\subseteq L_A$ as the kernel of the composite\rev{
\[ L_A\xr{\res^A_B} L_B \to \Phi^B L \to \Phi^B L\otimes_L L/I_{p,n}.\]}
Here, $\Phi^BL$ is defined as the localization of \rev{$L_B$} away from all the Euler classes of non-trivial characters for $B$. The ring $\Phi^BL$ is an algebraic version of geometric fixed points and indeed agrees with the coefficient ring of the $B$-geometric fixed points of $MU_B$.
\begin{thm}[\cref{thm:ClassificationOfInvariantPrimes}] For every abelian compact Lie group $A$ the map 
\begin{align*}  \Sub(A)\times \Spec^{\inv}(L) \to & \Spec^{\inv}(L_A) \\
            (B,I_{p,n}) \mapsto & I^A_{B,p,n}
\end{align*}
is a bijection.    
\end{thm}
Here, $\Sub(A)$ is the set of all closed subgroups of $A$. Hence, as for the Balmer spectrum, the invariant prime ideals of $L_A$ decompose as a set as one copy of $\Spec^{\inv}(L)$ for every closed subgroup of $A$. And similarly to the Balmer spectrum, the main work then lies in understanding the Zariski topology, in particular in determining the containments between invariant prime ideals associated to different subgroups.

We obtain the following:
\begin{thm}[\cref{thm:inclusions}] There is an inclusion $I^A_{B,p,n}\subseteq I^A_{B',q,n'}$ if and only if
\begin{enumerate}
    \item $B'$ is a subgroup of $B$.
    \item $p=q$ or $n=0$ (in which case $I^A_{B,p,0}=I^A_{B,q,0}$), the components $\pi_0(B/B')$ are a $p$-group and $n'\geq n + \rank_p(\pi_0(B/B'))$.
\end{enumerate}
\end{thm}

Comparing with \cite{BarthelGreenleesHausmann} we see that these correspond precisely to the inclusions in the Balmer spectrum, but with roles reversed: There is an inclusion $I^A_{B,p,n}\subseteq I^A_{B',q,n'}$ if and only if there is an inclusion $P^A_{B',q,n'}\subseteq P^A_{B,p,n}$. Here, $\markus{P^A_{B,p,n}} = \{ X\in \Sp_A^c\ |\ K(n)_*(\Phi^B X)=0\}$ are the thick subcategories of $\Sp_A^c$ with $K(n)$ being Morava K-theory at the prime $p$. 

To show that $I^A_{B,p,n}$ indeed includes into $I^A_{B', q, n'}$ when conditions $1$ and $2$ are satisfied, one can reduce to the case $A=\T$ the circle group where it is straightforward to describe explicit generators for the invariant prime ideals. The main step in ruling out further inclusions is the construction of equivariant refinements $\overline{\mv}_{n-1}\in L_{C_p^{n}}$ of the elements $v_{n-1}\in L$ which exhibit maximal height shifts (\cref{def:vn}, \cref{prop:xn}). 
Roughly speaking, $\overline{\mv}_{n-1}$ is of height $-1$ at the top group $C_p^n$ (i.e., it lies in the ideal \rev{$I^{C_p^n}_{C_p^{n},p,0}$}) while it is of height $n-1$ at the trivial group (i.e., it lies in the ideal \rev{$I^{C_p^n}_{\{1\},p,n}$} but not in \rev{$I^{C_p^n}_{\{1\},p,n-1}$}). This is the algebraic analog of the existence of finite $C_p^{n}$-spectra of underlying type $n$ whose $C_p^{n}$-geometric fixed points are rationally non-trivial as in  \rev{\cite[Section 4]{BarthelHausmannNaumannNikolausNoelStapleton} and \cite[Section 7]{KuhnLloydChromatic}}. More precisely, $\overline{\mv}_{n-1}$ is canonically defined only modulo a certain smaller ideal (analogously to $v_{n-1}$ only being defined uniquely up to the ideal $I_{n-1}$ \rev{and up to a unit}). More details are given in \cref{subsec:non-inclusions}.

We further show that - at least over elementary abelian $p$-groups - the elements $\overline{\mv}_i$ give rise to generators of the invariant prime ideals:
\begin{thm}[\cref{thm:generators1}] For all primes $p$ and $n\in \N$ the elements
\[ p_1^* \overline{\mv}_0, p_2^* \overline{\mv}_1,\hdots, p_{n-1}^* \overline{\mv}_{n-2},\overline{\mv}_{n-1} \]
generate the ideal $I^{C_p^n}_{C_p^n,p,0}$. Here, $p_i\colon C_p^n\to C_p^i$ denotes the projection to the first $i$ coordinates.
\end{thm}
Suitable restrictions of the $\overline{\mv}_n$ then form generators for the ideals $I^{C_p^n}_{C_p^n,p,m}$ at higher height~$m$, see Section \ref{sec:generators}. We emphasize that in contrast to the non-equivariant situation, the sequence of the $p_i^* \overline{\mv}_{i-1}$ is \emph{not} a regular sequence. In fact, since  $I^{C_p^n}_{C_p^n,p,0}$ consists precisely of the Euler-class-power torsion, it does not contain a non-zero divisor and hence cannot be generated by a regular sequence (unless $n=0$). The torsion in the ring $L_{C_p^n}$ is closely linked to the torsion in the group of characters $(C_p^n)^*$. Hence one might hope that $I^A_{B,n}$ is generated by a regular sequence whenever $A$ is a torus, and indeed that is the case in all the cases we understand (cf., Remark \ref{rem:generatorstori}).

Finally,  to describe the Zariski topology we need one additional ingredient. When $A$ is infinite, the set of closed subgroups $\Sub(A)$ \rev{carries} a non-trivial metric topology, turning it into a totally-disconnected compact Hausdorff space. Together with the inclusions between the invariant prime ideals, this topology determines the Zariski topology on $\Spec^{\inv}(L_A)$.

\begin{thm}[\cref{thm:zariski}] The Zariski topology on $\Spec^{\inv}(L_A)$ has as basis the closed subsets $C$ which are
\begin{enumerate}
    \item[(i)] closed under upward inclusions, i.e., if $I^A_{B',q,n'}\in C$ and $I^A_{B',q,n'}\subseteq I^A_{B,p,n}$, then $I^A_{B,p,n}\in C$, and which
    \item[(ii)] are locally constant on $\Sub(A)$ in the sense that every $B\in \Sub(A)$ has a neighborhood $U$ such that, for every $n$ and $p$, either $I^A_{B',p,n}\in C$ for all $B'\in U$ or $I^A_{B',p,n}\notin C$ for all $B'\in U$.
\end{enumerate}
\end{thm}

Comparing with \cite{BarthelGreenleesHausmann}, we see that this description precisely matches the computation of the topology on the Balmer spectrum, with $I^A_{B,p,n}$ replaced by $P^A_{B,p,n}$. Hence the assignment
\[ I^A_{B,p,n}\mapsto P^A_{B,p,n}\]
yields a homeomorphism from $\Spec^{\inv}(L_A)$ (and hence $|\MM_{FG}^A|$) to $\Spec(\Sp_A^c)$. In the last section we explain that this comparison map can be obtained less ad hoc via $MU_A$-homology:

\begin{thm}[\cref{sec:universalsupport}]\label{thm:support}
    Let $X$ be a compact $A$-spectrum and $I^A_{B,p,n}$ an invariant prime ideal. Then the localization $(MU_A)_{I^A_{B,p,n}}\wedge X$ is non-trivial if and only if the $B$-geometric fixed points $\Phi^B X$ are of type $\leq n$
    at $p$, i.e, if and only if $P^A_{B,p,n}$ is in the Balmer support of $X$.

This shows that
\[ X\mapsto \supp((\underline{MU}_A)_*X) \subseteq \Spec^{\inv}(L_A)\cong |\MM^A_{FG}| \]
defines a universal support theory on compact $A$-spectra and thus a homeomorphism $\Spec^{\inv}(L_A) \to \Spec(\Sp^c_A)$. Here, $(\underline{MU}_A)_* X$ is the Mackey functor recording $(MU_B)_*\res^A_B X$ for all closed subgroups $B$ of $A$, and $\supp((\underline{MU}_A)_*X)$ is defined as the set of invariant prime ideals at which the localization of this Mackey functor is non-trivial.
\end{thm} 

Knowing that $X\mapsto \supp((\underline{MU}_A)_*X)$ is a (not necessarily universal) support theory with a characterization in terms of geometric fixed points as above already yields a continuous bijection $\Spec^{\inv}(L_A) \to \Spec(\Sp^c_A)$.
This reproves one half of the main theorems of \cite{BarthelHausmannNaumannNikolausNoelStapleton} and \cite{BarthelGreenleesHausmann}. This half has been dubbed the chromatic Smith fixed point theorem in \cite{KuhnShort} and \cite{BalderramaKuhnChromatic}, where other proofs are given. To establish that our support theory is universal (and hence that $\Spec^{\inv}(L_A) \to \Spec(\Sp^c_A)$ is also open), we need additional topological input from \cite{BarthelHausmannNaumannNikolausNoelStapleton} and \cite{BarthelGreenleesHausmann}; see also \cite{KuhnLloydChromatic}.

We show in \cref{prop:universalsupportMFG} that the support theory from \cref{thm:support} can alternatively be built by viewing $(\underline{MU}_A)_* X$ as defining quasi-coherent sheaves on $\MM_{FG}^B$ for every closed $B\subseteq A$; the union of the supports of these sheaves agrees with $\supp((\underline{MU}_A)_*X)$ under the homeomorphism $\Spec^{\inv}(L_A)\cong |\MM^A_{FG}|$. To show this, we establish in \cref{prop:FACompatibility} how the adjunction between $\Sp^A$ and $\Sp^{A/B}$ defined by geometric fixed points and pullback corresponds on the algebraic level to pullback and pushforward along the open immersion $\MM_{FG}^{A/B} \subseteq \MM_{FG}^A$ from \cref{prop:OpenClosedSubstacks}.

\subsection{Acknowledgements} 
We thank Robert Burklund, Jeremy Hahn and Allen Yuan for their interest and enlightening discussions about the elements~$\overline{\mv}_n$ and transfers, \revm{and the anonymous referee for their comments}. We thank the Institut Mittag-Leffler for its hospitality during the research program \emph{Higher algebraic structures in algebra, topology and geometry} and the Hausdorff Research Institute for Mathematics for its hospitality during the trimester program \emph{Spectral Methods in Algebra, Geometry, and Topology} -- both have provided a welcoming and stimulating work environment. Finally, the first author was supported by the Knut and Alice Wallenberg Foundation, and the second author was supported by the NWO grant VI.Vidi.193.111a.

\section{Equivariant formal groups}
The aim of this section is to recall some basic definitions and properties about equivariant formal groups and equivariant formal group laws from \cite{CGK, GreenleesFGL, StricklandMulti, Hau}. To make our paper more self-contained, we \rev{also replicate} some of the proofs in our language, and we provide some small extensions of known results. Our treatment of equivariant formal groups is far from exhaustive and especially \cite{StricklandMulti} contains a wealth of results we do not touch upon.

\subsection{Basic definitions}\label{sec:basics}
In this subsection, we will recall the notions of an equivariant formal group and an equivariant formal group law over a commutative ring $k$. The definition of an equivariant formal group law is due to Cole, Greenlees and Kriz \cite{CGK}, and our definition of an equivariant formal group will be a variant of that of Strickland \cite{StricklandMulti}. 

For us, a \emph{formal $k$-algebra} is a complete linearly topologized commutative $k$-algebra with a countable system of open ideals generating the topology. \rev{We define a functor $F$ from formal $k$-algebras to the pro-category of commutative $k$-algebras, sending a formal $k$-algebra $R$ to the pro-system $(R/I)$ with $I\subseteq R$ running over all open ideals. The essential image is contained in the full subcategory $\mathrm{Pro}^{\mathrm{cont}, \mathrm{surj}}(\mathrm{CRing})$ on those pro-objects indexed on a countable directed set with surjective transition maps and actually equals this subcategory by the discussion after Definition 5.4 in \cite{Yasuda}. On this subcategory, taking the limit of the pro-system defines a right adjoint to $F$, which is also a one-sided inverse. Thus, $F$ is fully faithful and defines an equivalence onto $\mathrm{Pro}^{\mathrm{cont}, \mathrm{surj}}(\mathrm{CRing})$.} 

For us, the category of \emph{formal $k$-schemes} is the opposite of that of formal $k$-algebras. \rev{Using the above, i}t can be viewed as the category of ind-objects in affine $k$-schemes, indexed by a countable directed set with closed immersions as transition maps. We will sometimes use the notation $\Spec R$ or $\mathrm{Spf} R$ for the formal $k$-scheme associated to a formal $k$-algebra $R$, and $\OO_X$ for the formal $k$-algebra associated to a formal $k$-scheme $X$. 
The product on affine $k$-schemes induces one on formal $k$-schemes, and this corresponds to the completed tensor product on formal $k$-algebras.

We set $S = \Spec k$. Given a countable set $M$, we view $S \times M$ as a formal scheme, namely as the colimit over all $S\times N$ with $N\subseteq M$ finite. This corresponds to giving $k^M$ the product topology. If we just write $M$, we will apply this construction to $\Spec \Z$ instead of $S$.

For a compact Lie group $A$, we will denote by $A^* = \Hom(A, \T)$ its Pontryagin dual, which is always a discrete group. We use $\epsilon$ for the unit element in $A^*$.

\begin{definition}\label{def:EquivariantFormalGroup}
    Given a compact abelian Lie group $A$, an \emph{$A$-equivariant formal group over $k$} consists of a commutative group object $X$ in formal $k$-schemes together with a group homomorphism $\varphi\colon S \times A^* \to X$ of formal $k$-schemes satisfying the following two conditions: 
    \begin{enumerate}
    \item\label{item:EFG1} For the composite $\varphi_{\epsilon}\colon S \xrightarrow{\id_S \times \epsilon} S \times A^* \xrightarrow{\varphi} X$, the augmentation ideal $I_{\epsilon} = \ker(R \to k)$ of the induced map is fpqc-locally on $k$ a free $R$-module of rank $1$, where $R=\OO_X$. 
    \item \label{item:EFG2}The topology on $R$ is generated by products of the ideals $I_V = \ker(R \xrightarrow{\varphi_V^*} k)$ for $V\in A^*$ and $\varphi_V\colon S \xrightarrow{\id_S \times V} S \times A^* \xrightarrow{\varphi} X$.
    \end{enumerate}
    \end{definition}

\begin{remark}
    In the above definition, one can easily replace $\Spec k$ by an arbitrary quasi-compact scheme $S$, with formal $S$-schemes being a suitable subcategory of the ind-category of $\Aff_S$, the category of schemes affine over $S$. 

    Our definition differs in two aspects from that put forward in \cite[Definition 2.15]{StricklandMulti}. First, Strickland restricts to finite $A^*$. Second, Strickland asks $\ker(\OO_X \to k)$ 
    to be free of rank $1$ instead of locally free (cf.\ \cite[Proposition 2.10]{StricklandMulti}). We changed it so that our definition satisfies descent. Note that we could have asked equivalently that the augmentation ideal is Zariski locally on $k$ a free $R$-module of rank $1$ because line bundles satisfy fpqc-descent. 
\end{remark}

\begin{remark}\label{rem:GroupsCompletion}
    If we leave out the second condition in \cref{def:EquivariantFormalGroup}, we get a notion we call an \emph{$A$-equivariant group}. The category of $A$-equivariant formal groups embeds into that of $A$-equivariant groups and this inclusion has a right adjoint, called \emph{completion}. Concretely, this replaces $R$ in the notation in \cref{def:EquivariantFormalGroup} by the formal $k$-algebra $\lim_{V_1, \dots, V_n\in A^*}R/(I_{V_1}\cdots I_{V_n})$. We will only use $A$-equivariant groups to complete them to $A$-equivariant formal groups. 
\end{remark}

\rev{If we fix a trivialization of the augmentation ideal $I_{\epsilon}$ in \cref{def:EquivariantFormalGroup}, we can spell out the definition in more algebraic terms and obtain the notion of an equivariant formal group law.}

\begin{defi}
   An \emph{$A$-equivariant formal group law} over $k$ is a quadruple
    \[ (R,\Delta,\theta,y(\varepsilon)) \]
    of a formal $k$-algebra $R$, a continuous comultiplication $\Delta\colon R\to R\hotimes R$, a map of $k$-algebras $\theta\colon R\to k^{A^*}$ and an orientation $y(\epsilon)\in R$, such that
    \begin{enumerate}
    \item[(i)] the comultiplication is a map of $k$-algebras which is cocommutative, coassociative and counital for the augmentation $\theta_{\epsilon}\colon R\to k$,
    \item[(ii)] the map $\theta$ is compatible with the coproduct, and the topology on $R$ is generated by finite products of the kernels of the component functions $\theta_V\colon R\to k$ for $V\in A^*$, and
    \item[(iii)] the element $y(\epsilon)$ is regular and generates the kernel of $\theta_{\epsilon}$.
    \end{enumerate}
\end{defi}

\begin{remark}
    If we want to remember the base of an equivariant formal group law, we sometimes also write it as a quintuple $(k,R,\Delta,\theta,y(\varepsilon))$.
\end{remark}

\begin{lemma}\label{lem:EFGEFGL}
    An $A$-equivariant formal group $(\Spec R, \varphi)$ over $k$ together with an $R$-linear isomorphism $I_{\epsilon} \cong R$ of the augmentation ideal is equivalent datum to an $A$-equivariant formal group \emph{law} over $k$. 
    \end{lemma}
    \begin{proof}
    The maps $\Delta$ and $\theta$ are induced by the multiplication on $\Spec R$ and $\varphi$, respectively. The element $y(\varepsilon)$ corresponds to the trivialization of $I_{\epsilon} = \ker(R \to k)$.
\end{proof}

Given any equivariant formal group $G = (\varphi\colon A^* \to X)$ over $S = \Spec k$, we obtain for every $V\in A^*$ a morphism $\varphi_V\colon S \xrightarrow{\id_S \times V} S \times A^* \xrightarrow{\varphi} X$. If  $R=\OO_X$, this corresponds to the morphism $\theta_V\colon R \to k$. Moreover, $\varphi$ composed with left multiplication defines an $A^*$-action on $X$; 
for every $V\in A^*$ this gives a map $l_V \colon R \to R$. In terms of the data of an equivariant formal group law $F = (k,R,\Delta,\theta,y(\varepsilon))$, this can explicitly be written as 
\[ l_V\colon R\xr{\Delta} R\hotimes R \xr{\theta_V\hotimes \id_R}R. \]
Given $V\in A^*$, we set 
\[y(V)=l_V(y(\epsilon))\in R, \] 
which generates the kernel of $\theta_{V^ {-1}}$.
If $A$ is trivial, we have $R \cong k\llbracket y\rrbracket$. We want to describe an analog for general $A$. A \emph{complete flag} for $A$ is a sequence of characters $f=V_1,V_2,\hdots \in (A^*)^\N$ such that every character appears infinitely often. Given such a flag and $n\in \N$, we set \revm{$W_n=V_n\oplus V_{n-1}\oplus \dots \oplus V_1$ and}
\[ y(W_n)=y(V_n)y(V_{n-1})\cdots y(V_1). \]
Then every element $x$ of $R$ can be written uniquely as
\begin{equation}\label{eq:uniquelywritable} x=\sum_{n\in \N} a_n^f y(W_n) \end{equation}
for coefficients $a_n^f\in k$ \cite[Lemma 13.2]{CGK}. Hence, as a $k$-module, $R$ is isomorphic to a countable infinite product of copies of $k$.

\revm{We refer to \cite{CGK} and \cite{Hau} for more information about equivariant formal group laws.}
\subsection{Lazard rings}\label{sec:Lazard}
Our aim in this subsection is to recall the definition of the universal ring for equivariant formal group laws and to clarify its universal property. Let us begin by considering a very strict form of morphisms of equivariant formal group laws.
\begin{definition}\label{def:morphism} A \emph{morphism} between $A$-equivariant formal group laws $(k_1,R_1,\Delta_1,\theta_1,y(\varepsilon)_1)$ and $(k_2,R_2,\Delta_2,\theta_2,y(\varepsilon)_2)$ is a pair of maps $f\colon k_1 \to k_2$ and $g\colon R_1\to R_2$ which are compatible with both the comultiplications $\Delta$ and the augmentations $\theta$ and which send $y(\varepsilon)_1$ to $y(\varepsilon)_2$.
\end{definition}
 This leads to a category $A$-FGL of $A$-equivariant formal group laws. In \cite[Section 14]{CGK} it is shown that this category has an initial object $F^{\uni}$, the ground ring of which is called the \emph{$A$-equivariant Lazard ring} and denoted $L_A$. In fact, the category of $A$-equivariant formal group laws is equivalent to the category of commutative rings under $L_A$. To discuss this, note first that the forgetful functor
 \begin{align*}A\text{-FGL} \to \mathrm{CRing}, \qquad
 (k,R,\Delta,\theta,y(\varepsilon)) &\mapsto k
 \end{align*}
 into the category of commutative rings is cofibered in groupoids. Concretely this boils down to the following two observations:
 \begin{itemize}
     \item Every morphism of $A$-equivariant formal group laws whose first component $f\colon k_1 \to k_2$ is the identity map is an isomorphism. One observes indeed that the diagram
     \[
     \xymatrix{ R_1 \ar[rr]\ar[dr]_-{\cong} && R_2\ar[dl]^-{\cong}\\
     & \prod_{\N} k_1 = \prod_{\N} k_2 }
     \]
     obtained from \cref{eq:uniquelywritable} commutes. We call such an isomorphism living over the identity a \emph{very strict isomorphism} between $A$-equivariant formal group laws, in order to distinguish from other kinds of isomorphisms.
     \item Given a morphism $f\colon k_1 \to k_2$ and an $A$-equivariant formal group law $F$ over $k_1$, one can define a pushforward $f_*F$ over $k_2$ with the usual universal property. Its underlying $k_2$-algebra is given by a completion of $R \tensor_{k_1} k_2$, where $R$ is the underlying $k_1$-algebra of $F$ (cf. \cite[Section 2.E]{GreenleesFGL}, \cite[Section 2.3]{Hau}).
 \end{itemize}

 Note that the only automorphism of an $A$-equivariant formal group law over $k$ which is also a very strict isomorphism is the identity map. Hence, given two $A$-equivariant formal group laws over $k$, there either exists a unique very strict isomorphism between them or none at all. For this reason it is usually harmless to identify two very strictly isomorphic $A$-equivariant formal group laws, and we will often do so.

Now, given a map $f\colon L_A\to k$ there is an induced $A$-equivariant formal group law $f_*F^{\uni}$ over $k$ obtained by pushing forward the universal $A$-equivariant formal group law. Given an $A$-equivariant formal group law $F$ over $k$, we can apply this to the first component $f\colon L_A\to k$ of the unique map of $A$-equivariant formal group laws $F^{\uni}\to F$. The resulting morphism $f_*F^{\uni} \to F$ is necessarily a very strict isomorphism. So, as claimed above, we obtain:

\begin{cor} The functor
\begin{align*}  \mathrm{CAlg}_{L_A}  \to A\text{-FGL}, \qquad
            (f\colon L_A\to k)  \mapsto  f_*F^{\uni} \end{align*}
from commutative $L_A$-algebras \revm{to $A$-equivariant formal group laws} is an equivalence of categories. An inverse is given by sending an $A$-equivariant formal group law $F=(k,R,\Delta,\theta,y(\varepsilon))$ to the first component $f\colon L_A\to k$ of the unique morphism $F^{\uni}\to F$.
\end{cor}

\begin{remark}\label{rem:cofibered}
    The above proof is an instance of a general characterization of initial objects $X$ in categories cofibered in  groupoids $\CC \xrightarrow{F}\DD$, namely that pushforward defines an equivalence of $\DD_{F(X)/-}$ with $\CC$. 
\end{remark}

\begin{remark} Non-equivariantly the $k$-algebra $R$ is often fixed to be the power series ring $k\llbracket y(\varepsilon)\rrbracket $ rather than a ring only isomorphic to it. With this convention, the category of formal group laws is \emph{isomorphic} (not merely equivalent) to the category of commutative rings under $L$. In other words, every very strict isomorphism of formal group laws is the identity. Equivariantly one needs to be a little more careful: The statement `$A$-equivariant formal group laws are represented by the $A$-equivariant Lazard ring' is only true up to this notion of very strict isomorphism.
\end{remark} 

\subsection{Global functorality}\label{sec:GlobalFunctorality}
In this subsection, we will discuss both a covariant and a contravariant functoriality of the category of $A$-equivariant formal groups in $A$. 

\begin{defi}\label{def:corestriction}
    Let $\alpha\colon B \to A$ be a group homomorphism and let $G = (B^* \to X)$ be a $B$-equivariant formal group. We define the \emph{corestriction} $\alpha_*G$ to be the $A$-equivariant formal group which is the completion of the $A$-equivariant group $(A^* \xrightarrow{\alpha^*} B^* \to X)$.
\end{defi}

\begin{prop}\label{prop:alphastarfullyfaithful}For every injective group homomorphism $\alpha\colon B\to A$, the functor $\alpha_*$ from $B$-equivariant formal groups to $A$-equivariant formal groups is fully faithful. The essential image consists of those $A$-equivariant formal groups where the homomorphism from $A^*$ factors through $B^*$. 
\end{prop}
\begin{proof}
    For every $B$-equivariant formal group $G$, by construction $\alpha_*G$ is the same group object as $G$ in formal schemes with the structure morphism $A^* \to B^* \to G$ (since $A^* \to B^*$ is surjective, no completion is necessary). This implies fully faithfulness. 
\end{proof}

Upon choosing coordinates, \cref{def:corestriction} corresponds to the construction of \cite[Section 2.4]{Hau}: given a $B$-equivariant formal group law $F=(k,R,\Delta,\theta,y(\varepsilon))$ and a group homomorphism $\alpha\colon B\to A$, there is an induced $A$-equivariant formal group law $\alpha_*F$ over the same ring $k$, given by completing $R$ at products of the ideals $I_V = \ker(R\xrightarrow{\theta_V} k)$ for those $V\in B^*$ which are in the image of $\alpha^*\colon A^*\to B^*$. This defines a functor $\alpha_*$ from $B$-equivariant formal group laws to $A$-equivariant formal group laws, which induces a map $\alpha^*\colon L_A\to L_B$ on Lazard rings. Hence we obtain a functor
\[ \mathbf{L}\colon \text{(abelian compact Lie groups)}^{op}\to \text{commutative rings} \]
which we call the \emph{global Lazard ring}. As shown in \cite{Hau}, $L_A$ is isomorphic to $\pi_*^AMU_A$ and our map $\alpha^*$ corresponds to the restriction map on that level, explaining our terminology. 

\begin{remark} By Pontryagin duality, the opposite category of abelian compact Lie groups is equivalent to the category of finitely generated abelian groups. Therefore, everything in this paper could alternatively be phrased in terms of finitely generated abelian groups rather than abelian compact Lie groups, and the more algebraically minded reader might prefer to do so.
\end{remark}

In addition to this covariant functoriality, there is also a contravariant functoriality. 

\begin{defi}
    Let $\alpha\colon A \to C$ be a surjective group homomorphism and let $G = (C^* \to X)$ be a $C$-equivariant group. We define the \emph{coinduction} $\alpha^*G$ to be the $A$-equivariant group $(A^* \to X\times_{C^*}A^*)$, where the target denotes the quotient of $X \times  A^*$ by the antidiagonal $C^*$-action. 
\end{defi}

\begin{lemma}
    For $\alpha\colon A \to C$ a surjective group homomorphism and $G$ a $C$-equivariant formal group, $\alpha^*G$ is an $A$-equivariant formal group, i.e.\ needs no additional completion. 
\end{lemma}
\begin{proof}
If $G = (C^* \to \Spec R)$ is a $C$-equivariant formal group, then \[\Spec R\times_{C^*}A^* \cong \Spec \Map_{C^*}(A^*, R).\] 
\revm{Any choice of coset representatives defines an isomorphism of formal $k$-algebras $\Map_{C^*}(A^*, R)\cong \prod_{A^*/C^*}R$, with respect to the product topology on the latter. As products of complete rings are complete, the claim follows.}
\end{proof}

\begin{prop}\label{prop:alphaadjunction}
    Let $\alpha\colon A \to C$ be a surjective group homomorphism. As functors between $C$-equivariant formal groups and $A$-equivariant formal groups, $\alpha^*$ is the left adjoint of $\alpha_*$.
\end{prop} 
\begin{proof}
    For an $A$-equivariant group $G = (A^* \to X)$, define $\widetilde{\alpha}_*G$ as $C^* \to A^* \to X$. Then $\alpha^*$ and $\widetilde{\alpha}_*$ are adjoints between $C$-equivariant groups and $A$-equivariant groups in the sense of \cref{rem:GroupsCompletion}. Since completion is a right adjoint, the result follows from the previous lemma. 
\end{proof}

\subsection{Euler classes}
Given an $A$-equivariant formal group law $F=(R,\Delta,\theta,y(\varepsilon))$  over $k$, we can define Euler classes in $k$. Recall that for $V\in A^*$, we set $y(V)=l_V(y(\epsilon))\in R$. The corresponding Euler class is
\[ e_V=\theta_{\epsilon}(y(V))=\theta_V(y(\epsilon))\in k. \] 

In terms of the associated equivariant formal group $G = (\varphi\colon A^* \to X)$, we have 
\[S \times_{\varphi_{\epsilon},X, \varphi_V} S \cong \Spec (k \tensor_{\theta_{\epsilon}, R, \theta_V} k) \cong \Spec (k/e_V)\]
with $S = \Spec k$ and $\varphi_V$ being the composite $S \cong S\times \{V\} \subseteq S\times A^* \xrightarrow{\varphi} X$. This implies:
\begin{lemma}\label{lem:CoordinateFreeEuler} For a given $A$-equivariant formal group law with notation as above: 
    \begin{enumerate}
        \item The Euler class $e_V$ is invertible iff $S \times_{\varphi_{\epsilon},X, \varphi_V}S = \varnothing$.
        \item The Euler class $e_V$ is zero iff $\varphi_V = \varphi_{\epsilon}$. 
    \end{enumerate}
\end{lemma}

Thus, the vanishing or invertibility of Euler classes does not depend on chosen coordinates. This allows us to generalize these concepts to $A$-equivariant formal groups in the following way: 

\begin{defi}
    For an $A$-equivariant formal group $G = (\varphi\colon A^* \to X)$, we say that the \emph{Euler class $e_V$ is invertible} if $S \times_{\varphi_{\epsilon},X, \varphi_V}S = \varnothing$ and that \emph{$e_V$ is zero} if $\varphi_V = \varphi_{\epsilon}$. 
\end{defi}

Informally, $e_V$ is invertible if and only if the images of $S\times \{\epsilon\}$ and $S\times \{V\}$ in $X$ are disjoint. 

\begin{example}\label{exa:Gm}
 Let $\G_m = \Spec \Z[x^{\pm1}]$ be the multiplicative group over $S = \Spec \Z$. We choose the group homomorphism $\varphi\colon C_2 = (C_2)^* \to \G_m$ picking the units $\{\pm 1\}$ in $\Z$. This defines the structure of a $C_2$-equivariant group, and its completion is a $C_2$-equivariant formal group we call $\widehat{\G}_m^{C_2}$. 
 Let $V\in (C_2)^*$ be the unique non-trivial character. One computes
 \[\Spec \Z \times_{\varphi_{\epsilon}, \widehat{\G}_m^{C_2}, \varphi_V} \Spec \Z \cong  \Spec \Z \times_{\varphi_{\epsilon}, \G_m, \varphi_V} \Spec \Z\cong \Spec \Z/2.\]
 Thus, $e_V = \pm 2$, depending on the choice of coordinate. See also \cite[Section 7]{StricklandMulti} and \cite[Section 7]{GreenleesFGL} for more information on this and related examples. Note in particular that our example is the pushforward of the true $C_2$-equivariant multiplicative formal group (given by a completion of $\Spec \Z[(C_2\times \T)^*]$) along the map $\Z[C_2^*]\to \Z$ classifying $\pm 1$.
 \end{example}
 \begin{figure}
     \centering
     \begin{tikzpicture}
  \draw (-2,.4) parabola bend (0,0) (2,.4) node[right] {$\mathrm{Spf}(\mathbb{Z}\llbracket t\rrbracket)$};
    \draw (-2,-.4) parabola bend (0,0) (2,-.4) node[right] {$\mathrm{Spf}(\mathbb{Z}\llbracket t\rrbracket)$};
    
    \draw[fill] (0,0) circle (1.5pt);
    
    \draw (-2, -2)--(2, -2) node[right] {$\mathrm{Spec}(\mathbb{Z})$};
    
    \draw[fill] (0,-2) circle (1.5pt)node[below]{$(2)$};
\end{tikzpicture}
     \caption{A schematic picture of $\widehat{\G}_m^{C_2}$ from \cref{exa:Gm}}
     \label{fig:GmC2}
 \end{figure}
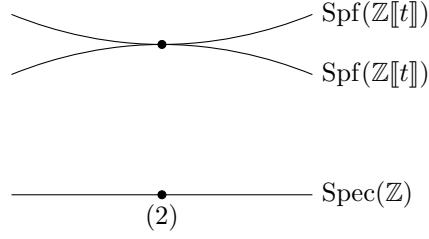

 We will use \rev{the following lemma several times}, taken from \cite[Corollary 2.8]{Hau} and the explanation thereafter: 
\begin{lemma}\label{lem:EulerClassesGenerateKernel}
    Let $B\subseteq A$ be a subgroup. Then the restriction map $L_A \to L_B$ is surjective, with kernel $I_A^B$ generated by the Euler classes $e_V$ where $V$ is running over a generating set of $\ker(A^* \to B^*)$. 
\end{lemma}
Similarly, the following holds on the level of equivariant formal groups.
\begin{proposition}\label{prop:EulerClassInvertibleOrZero}
    Let $G$ be an $A$-equivariant formal group. 
\begin{enumerate}
    \item Let $\alpha\colon A\to C$ be a surjective group homomorphism. Assume that for $V\notin \im(C^* \to A^*)$, the Euler class $e_V$ is invertible. Then $\alpha^* \alpha_*G \to G$ is an isomorphism. 
    \item Let $\alpha\colon B \to A$ be an injective group homomorphism. Assume that for $V\in \ker(A^* \to B^*)$, we have $e_V = 0$. Then $G$ is in the essential image of $\alpha_*$. 
\end{enumerate}
\end{proposition}
\begin{proof}
Let $G = (\varphi\colon A^* \to X)$. Fixing a coordinate $y(\epsilon)$ Zariski-locally and choosing a complete flag, $X$ is defined by the directed system $(\Spec R/(y(V_1)\cdots y(V_n)))_n$ and each of these terms has underlying space $\bigcup_{i=1}^n\im(\varphi_{V_i}) \subseteq \Spec R$. 

In the first item, \cref{lem:CoordinateFreeEuler} implies that $\im(\varphi_V)$ and $\im(\varphi_W)$ intersect each other in $\Spec \OO_X$ only if $[V] = [W] \in A^*/C^*$. Thus, the underlying space of $(\Spec R/(y(V_1)\cdots y(V_n)))_n$ decomposes into closed subspaces $\coprod_{\nu\in A^*/C^*} \bigcup_{V_i\in \nu} \im(\varphi_{V_i})$, of which only finitely many are non-empty. This induces decompositions of the schemes $\Spec R/(y(V_1)\cdots y(V_n)))$ and we obtain thus an $A^*$-equivariant isomorphism $G \cong \coprod_{A^*/C^*}\alpha_*G$ on the level of formal schemes. 

By construction, $\alpha^*\alpha_*G$ decomposes in the same way. On the unit copy, the map $\alpha^*\alpha_*G \to G$ is an isomorphism since $\alpha_*\alpha^*\alpha_* G \to \alpha_*G$ is one by \cref{prop:alphaadjunction}.
For the other copies, this follows by the $A^*$-equivariance of the map $\alpha^*\alpha_*G \to G$. 
    
    For the second item: by definition, the structure morphism $\varphi\colon A^* \to X$ of $G$ factors as $A^* \to B^* \xrightarrow{\varphi'} X$. The result follows from \cref{prop:alphastarfullyfaithful}. 
\end{proof}

The second part is also true in the setting of equivariant formal group \emph{laws}, as shown in \cite[Lemma 2.7]{Hau}. The analog of the first part becomes more complicated as a coordinate on $\tilde G$ does not determine a canonical coordinate on $\alpha^*\tilde G$ and hence $\alpha^*$ does not define a functor from $C$-equivariant formal group laws to $A$-equivariant formal group laws without additional choices; we will discuss this further in \cref{prop:GeometricFixedPoints}. 

\begin{cor}\label{cor:fullyfaithful}
    Let $\alpha\colon A\to B$ be a surjective group homomorphism. Then $\alpha^*$ is a fully faithful embedding from the category of $B$-equivariant formal groups to that of $A$-equivariant formal groups. The image consists of those $A$-equivariant formal groups such that $e_V$ is invertible for $V$ not in the image of $B^* \to A^*$. 
\end{cor}
\begin{proof}

If $G$ is a $B$-equivariant formal group, then $\alpha^*G$ has the property that $e_V$ is invertible for $V$ not in the image of $B^* \to A^*$ by construction. By the preceding proposition, invertibility of these Euler classes characterizes the image of $\alpha^*$. Moreover, $\alpha^*$ is fully faithful since $G \to \alpha_*\alpha^*G$ is an isomorphism by construction.
\end{proof}

The following proposition \rev{essentially provides} a classification of equivariant formal groups over fields. The same statement already appears in \cite[Corollary 8.3]{StricklandMulti}.
\begin{proposition}\label{prop:ClassificationOverFields}
    Let $k$ be a field and $G$ be an $A$-equivariant formal group over $k$. Denote by $A/B$ the Pontryagin dual of the subgroup $\{V \in A^*: e_V=0\}\subseteq A^*$.  Denote further by $A \xrightarrow{q} A/B \xleftarrow{i} \{1\}$ the obvious morphisms. Then $G \cong q^*i_*p_*G$, where $p_*G$ is the \emph{non-equivariant} formal group defined by $p\colon A \to \{1\}$. 
\end{proposition}
\begin{proof}
    Assume $e_V = 0$, i.e.\ $\varphi_{\epsilon} = \varphi_V$. By the $A^*$-action, this implies $\varphi_W = \varphi_{VW}$ for every $W\in A^*$. Setting $W = V^{-1}$, this implies $e_{V^{-1}} = 0$. Moreover, if $\varphi_{\epsilon} = \varphi_W$, this implies $\varphi_{\epsilon} = \varphi_{VW}$. Thus, $\{V \in A^*: e_V=0\}\subseteq A^*$ is indeed a subgroup. 

    Since $k$ is a field, $e_V = 0$ iff $e_V$ is not invertible. By \cref{prop:EulerClassInvertibleOrZero}, we thus have $G \cong q^*i_*\Gamma$ for some non-equivariant formal group $\Gamma$ over $k$. One computes $\Gamma \cong p_*q^*i_*\Gamma \cong p_*G$.
\end{proof}

For any $A$, let $\Phi^AL = L_A[e_V^{-1}]$ be the localization of $L_A$ away from all Euler classes $e_V$ for $V\neq \epsilon$. Our results above let us compute $\Phi^AL$ quite explicitely (cf.\ \cite[Corollary 6.4]{GreenleesFGL} and \cite[Proposition 2.11]{Hau}). 
\begin{prop}\label{prop:GeometricFixedPoints}
    There is a \rev{ring} isomorphism of the form
    \[ \Phi^AL\cong L[(b_0^V)^{\pm 1}, b_i^V\ |\ i>0,\, V\in A^*-\{\epsilon\}], \]    
    \rev{where the $b_i^V$ are independent polynomial variables.}
\end{prop}
\begin{proof}
    The ring $\Phi^AL$ classifies $A$-equivariant formal group laws $F = (R,\Delta,\theta,y(\varepsilon))$ such that $e_V$ is invertible for all $V\neq \epsilon$. By \cref{prop:EulerClassInvertibleOrZero}, $R \cong \map(A^*, \widehat{R})$, where $\widehat{R}$ is the completion of $R$ at the augmentation ideal. The structure of $F$ determines the structure of a non-equivariant formal group law $p_*F$ on $\widehat{R}$, where $p\colon A\to \{1\}$ is the projection; in particular, we obtain an isomorphism $\widehat{R} \cong k\llbracket y\rrbracket$, where $y$ is the image of $y(\epsilon)$. Vice versa, $\Delta$ and $\theta$ are determined by $p_*F$. In particular, $\theta_\epsilon$ is the composite $\map(A^*, \widehat{R}) \xrightarrow{\ev_{\epsilon}} \widehat{R} \to k$. Thus we see that $F$ is determined by $p_*F$, plus a choice of $y(\epsilon)$ mapping to $y$ under $\ev_{\epsilon}$. Such $y(\epsilon)$ are exactly those elements $(y^V)\in \map(A^*, k\llbracket y\rrbracket)$ such that $y^{\epsilon} = y$ and $y^V = b_0^V + b_1^Vy +\cdots$  with $b_0^V \in k^{\times}$. This gives the result.
\end{proof}
\begin{remark} \label{rem:naturalitybi} The elements $b_i^V \in \Phi^A L$ in the previous proposition already come from elements $\gamma_i^V\in L_A$, which are uniquely \rev{determined} 
by the property that
\[ e_{\epsilon \otimes \tau} = \gamma_0^V + \gamma_1^V e_{V^{-1}\otimes \tau} + \gamma_2^V (e_{V^{-1}\otimes \tau})^2 + \hdots + \gamma_n^V (e_{V^{-1}\otimes \tau})^n \in L_{A\times \T}/(e_{V^{-1}\otimes \tau})^{n+1}, \]
for all $n\in \mathbb{N}$. Here, $\tau\in \T^*$ denotes the tautological character for the circle group $\T$\revm{, and we view $L_{A\times \T}$ as an algebra over $L_A$ via restriction along the projection $A\times \T\to A$}. In particular, $\gamma_0^V$ equals the Euler class $e_V$. The elements $\gamma_i^V$ are natural in the sense that $\alpha^*\gamma_i^V=\gamma_i^{\alpha^*V}$ for every group homomorphism $\alpha\colon B\to A$. We refer to \cite[Section 2.7]{Hau} for more details on this construction.
\end{remark}

\subsection{The relationship between Lazard rings at different groups and their completions}
\label{sec:completion}

While not needed for our \emph{classification} of invariant prime ideals, it will be necessary for our study of \emph{containments} between invariant prime ideals to have a deeper look upon how Lazard rings at different groups relate. These properties are all based on the identification of the global Lazard ring with equivariant complex bordism in \cite{Hau}. 

\begin{prop}[\cite{Hau}, Proposition 5.50, Corollary 5.33, Lemma 5.28] \label{prop:regularglobal} \begin{enumerate}
    \item For every $A$ and every non-torsion character $V\in A^*$, the sequence
\[ 0\to L_A\xrightarrow{e_V\cdot} L_A \xrightarrow{\res^A_{\ker(V)}} L_{\ker(V)} \to 0 \]
is exact. In particular, all Euler classes $e_V\in L_A$ for non-torsion characters $V$ are non-zero divisors.
    \item For every $A$, the complete $L_A$-Hopf algebra $R$ of the universal $A$-equivariant formal group law is canonically isomorphic to the completion
    \[ \lim_{n\in \N,V_1,\hdots,V_n\in A^*} (L_{A\times \T})/I_{V_1}\cdots I_{V_n}, \]
    where $I_{V_j}$ is the kernel of the restriction map $(id,V_j)^*\colon L_{A\times \T}\to L_A$. More generally, $R^{\hotimes n}$ is a completion of $L_{A\times \T^n}$.
    Under this identification
    \begin{enumerate}
        \item the comultiplication $R\to R\hotimes R$ and the augmentations $\theta_V\colon R\to L_A$ are induced by the maps $(\id_A,m)^*\colon L_{A\times \T} \to L_{A\times \T\times \T}$ and $(\id_A,V)^*\colon L_{A\times \T}\to L_A$ on completion, and
        \item the elements $y(V)$ are the image of $e_{V\otimes \tau}\in L_{A\times \T}$ under the completion map, where $\tau\in \T^*$ is the tautological character.
    \end{enumerate}
\end{enumerate} 
\end{prop}
The special case of (2) for $A$ the trivial group is particularly important: Completing $L_{\T^n}$ at the kernel $I$ of the augmentation $L_{\T^n}\to L$ yields a power series ring on $n$ generators $L\llbracket y_1,y_2,\hdots,y_n\rrbracket $, where $y_i$ is the Euler class of the $i$-th projection $\T^n\to \T$. Moreover, the isomorphism
\[ (L_{\T^n})^{\wedge}_I \cong  L\llbracket y_1,y_2,\hdots,y_n\rrbracket \]
is natural in $\T^n$, where
\begin{itemize}
    \item the functoriality of $(L_{\T^n})^{\wedge}_I$ is induced by the global functoriality of $L_{\T^n}$, and
    \item the functoriality of $L\llbracket y_1,y_2,\hdots,y_n\rrbracket $ is through the universal formal group law over $L$.
\end{itemize}
For example, for any $n\in \N$ the square
\[ \xymatrix{L_\T \ar[r]^-c \ar[d]_{[n]^*} & L\llbracket y\rrbracket  \ar[d]^{[n]_F} \\
            L_\T \ar[r]_-c & L\llbracket y\rrbracket  } \]
commutes, where $[n]\colon \T\to \T$ 
is the $n$th power map, and $[n]_F\colon L\llbracket y\rrbracket \to L\llbracket y\rrbracket $ sends $y$ to the $n$-series $[n]_F(y)$ with respect to the universal formal group law. This implies that the Euler class $e_{\tau^n}\in L_\T$ for the $n$th power map on $\T$ is sent to the $n$-series $[n]_F(y)$ under the completion map. Similar statements hold for the collection of $L_{A\times \T^n}$ for fixed $A$ and varying $\T^n$.

We further record, where again $\tau\colon \T\to \T$ denotes the identity character:
\begin{cor} \label{cor:division} Let $x\in L_\T$ be an element whose image in $L\llbracket y\rrbracket$ is of the form $\lambda y^k + \text{higher order terms}$. Then $x$ is uniquely divisible by $e_\tau^k$ and the quotient $x/e_\tau^k$ restricts to $\lambda$ at the trivial group.
\end{cor}
\begin{proof} First we note that $e_\tau\in L_\T$ is a regular element by Proposition \ref{prop:regularglobal}. Hence division by~$e_\tau$ is always unique if possible. By induction on $k$ the corollary then follows from the following facts, for an element $z\in L_\T$ and its image $c(z)\in L_\T$:
\begin{enumerate}
    \item $z$ is divisible by $e_\tau$ if and only if $\res^\T_1(z)=0$.
    \item The leading coefficient of  $c(z)$ is equal to $\res^\T_1(z)$.
    \item If $\res^T_1(z)=0$ and hence $z$ is divisible by $e_\tau$, then $c(z/e_\tau)=c(z)/y$. \qedhere
\end{enumerate}
\end{proof}
Note that the global functoriality makes every equivariant Lazard ring $L_A$ an algebra over the non-equivariant Lazard ring $L$, and that all restriction maps $\alpha^*\colon L_A\to L_B$ are $L$-algebra maps. \rev{The following property is shown in \cite[Corollary 5.50]{Hau} as a consequence of work of Comeza\~{n}a \cite{Com96}.}
\begin{prop}\label{prop:free}\rev{For every abelian compact Lie group $A$, the equivariant Lazard ring} $L_A$ is free as a module over $L$. 
\end{prop}
\begin{cor} The exact sequences of $L$-modules in Part 1 of Proposition \ref{prop:regularglobal} are split exact. In particular, they remain exact after applying any additive functor.
\end{cor}
The following special case is of particular importance to us:
\begin{cor} \label{cor:regularmodIn} Let $n\in \overline{\N} = \N\cup\{\infty\}$ and $I_n\subseteq L$ be the ideal generated by $(v_0,\hdots,v_{n-1})$. Then for every $A$ and every non-torsion character $V\in A^*$, the sequence
\[ 0\to L_A/(I_n\cdot L_A)\xrightarrow{e_V\cdot} L_A/(I_n\cdot L_A) \xrightarrow{\res^A_{\ker(V)}} L_{\ker(V)}/(I_n\cdot L_{\ker(V)}) \to 0 \]
is exact. In particular, the Euler classes $e_V\in L_A$ for non-torsion characters $V\in A^*$ remain non-zero divisors in $L_A/I_n$. In the terminology of \cite{Hau}, the assignment
\[ A\mapsto L_A/(I_n\cdot L_A) \]
(together with the image of $e_\tau$ under $L_\T\to L_\T/(I_n\cdot L_\T)$) is a regular global group law.
\end{cor}

\section{The Lazard Hopf algebroid and its associated stack}

\subsection{Strict isomorphisms and the Lazard Hopf algebroid}
In this subsection, we will introduce one of our main objects of study, the Hopf algebroid $(L_A, S_A)$ for equivariant formal group laws. There is a hierarchy of notions of isomorphisms between (equivariant) formal group laws, namely
\begin{itemize}
    \item isomorphisms, which do not need to respect the coordinate and are thus really isomorphisms between the underlying (equivariant) formal groups;
    \item strict isomorphisms, which respect the coordinate up to quadratic terms; 
    \item very strict isomorphisms as in \cref{sec:Lazard}, which respect the coordinate strictly. 
\end{itemize}
Already classically, strict isomorphisms are especially relevant since the Hopf algebroid modeled on them gives $(MU_*, MU_*MU)$. 
\begin{definition} A \emph{strict isomorphism} between two $A$-equivariant formal group laws \[(k,R_1,\Delta_1,\theta_1,y(\varepsilon)_1) \quad\text{ and }\quad(k,R_2,\Delta_2,\theta_2,y(\varepsilon)_2)\] over the same ground ring $k$ is a $k$-linear isomorphism
\[ \varphi\colon R_1\xrightarrow{\cong} R_2 \]
of Hopf algebras over $k^{A^*}$ such that $y(\varepsilon)_1$ is sent to $y(\varepsilon)_2$ modulo $I_{\varepsilon}^2$, where $I_{\varepsilon}$ is the augmentation ideal in $R_2$. Explicitly, this means that $(\varphi\otimes \varphi)\circ \Delta_1=\Delta_2\circ \varphi$, $\theta_2\circ \varphi=\theta_1$ and $\varphi(y(\varepsilon)_1)=x\cdot y(\varepsilon)_2$ for some unit $x\in R_2$ which augments to $1\in k$.
\end{definition}
By definition, strict isomorphisms need not preserve the coordinate, hence they are generally not morphisms of formal group laws in the sense of \revm{\cref{def:morphism}}. On the other hand, every very strict isomorphism is both a strict isomorphism and an isomorphism in the category of $A$-equivariant formal group laws. 

Let $\SI$ be the category of strict isomorphisms of $A$-equivariant formal group laws. More precisely, its objects are quadruples $(k,F^1,F^2,\varphi)$ consisting of a commutative ring $k$, two $A$-equivariant formal group laws $F^1$ and $F^2$ over $k$ and a strict isomorphism $\varphi$ between them. Morphisms between two such quadruples $(k_1,F_1^1,F_1^2,\varphi_{1})$ and $(k_2,F_2^1,F_2^2,\varphi_2)$ are given by a pair of morphisms $f_1\colon F_1^1\to F_2^1$ and $f_2\colon F_1^2\to F_2^2$ \revm{ in the sense of \cref{def:morphism},} with the same underlying map $k_1\to k_2$, such that 
\[ \xymatrix{ R_1^1 \ar[r]^{\varphi_{1}} \ar[d]_{f_1} & R_1^2 \ar[d]^{f_2} \\
            R_2^1 \ar[r]_{\varphi_2} & R_2^2} \]
commutes.

\begin{proposition}\label{prop:SISA}
    The category $SI$ has an initial object
    , whose underlying ring $S_A$ is a localization of an infinite polynomial ring $L_A[a_1^f, a_2^f, \dots]$ over the Lazard ring $L_A$. 
\end{proposition}
By \cref{rem:cofibered}, we can equivalently say that \revm{there is an object $(S_A, F^{\uni}, \widetilde{F}^{\uni}, \varphi)\in \SI$ such that} the functor
\begin{align*} \mathrm{CAlg}_{S_A} &\to \SI \\
    (f\colon S_A\to k) & \mapsto \revm{(k,f_*F^{\uni}, f_*\widetilde{F}^{\uni},f_*\varphi)} \end{align*}
    is an equivalence of categories. Here, we use that $SI$ is again cofibered in groupoids over commutative rings. 

Before we prove the proposition, it will be good to review two general results about the maps $\theta_V\colon R \to k^{A^*} \to k$ for an $A$-equivariant formal group law $(k,R,\Delta,\theta,y(\varepsilon))$. Recall that after choosing a complete flag $f$ we can write every element $x\in R$ uniquely as 
\[ x=\sum_{n\in \N} a_n^f y(W_n).\]
Given $V\in A^*$, we have
\[ \theta_V(x)=\sum_{n\in \N} a_n^f \theta_V(y(W_n))=\sum_{n\in \N} a_n^f e_{VW_n}, \]
where $e_{VW_n}$ is defined as the product $e_{V\cdot V_n}e_{V\cdot V_{n-1}}\cdots e_{V\cdot V_1}$. This is a finite sum, since $e_{VW_n}=0$ if there exists some $i\leq n$ where $V_i=V^{-1}$. We obtain:
\begin{lemma} \label{cor:polynomials} The augmentation $\theta_V$ is a linear combination of $a_n^f$ whose coefficients are products of Euler classes.
\end{lemma}

Moreover, we have the following:
\begin{lemma} \label{lem:invertible} Let $F=(k,R,\Delta,\theta,y(\epsilon))$ be an $A$-equivariant formal group law. Then an element $x\in R$ is a unit if and only if $\theta_V(x)$ is a unit in $k$ for all $V\in A^*$.
\end{lemma}
\begin{proof} See \cite[Lemma 2.3]{Hau}.
\end{proof}

\begin{proof}[Proof of \cref{prop:SISA}]
    For every object $(k,F^1,F^2,\varphi)\in \SI$ we can define a new $A$-equivariant formal group law $\widetilde{F}^2$ for which the components $k,R,\Delta$ and $\theta$ agree with those of $F^1$, but $\widetilde{y}(\varepsilon)_2$ is defined as $\varphi^{-1}(y(\varepsilon)_2)$, the preimage of the coordinate of $F^2$ under $\varphi$. Then we obtain a new object $(k,F^1,\widetilde{F}^2,\id_{R_1})$ which is isomorphic in $\SI$ to $(k,F^1,F^2,\varphi)$ via the commutative square
\[ \xymatrix{ F^1 \ar[r]^{\id} \ar[d]_{\id} & \widetilde{F}^2 \ar[d]^{\varphi} \\
              F^1 \ar[r]^{\varphi} &  F^2} \]
Note that the two vertical maps are in fact very strict isomorphisms of $A$-equivariant formal group laws. In summary, every object of $\SI$ is isomorphic to one of the form $(k,F_1,F_2,\id)$, where $F_1$ and $F_2$ are given by the same $k^{A^*}$-augmented $k$-Hopf algebra and the strict isomorphism is the identity. This is the same data as a single $A$-equivariant formal group law $F$ together with a second choice of coordinate $y(\varepsilon)_2=x\cdot y(\varepsilon)_1$ for some unit $x\in R$ which augments to $1$. Up to very strict isomorphism we can further assume that $F$ is the push-forward along a map $L_A\to k$. We claim that the functor sending an $A$-equivariant formal group law to the set of all units 
$x\in R$
augmenting to $1$ is representable by an $L_A$-algebra $S_A$. Indeed, a presentation for $S_A$ is given by
\begin{equation}S_A=L_A[a_1^f,a_2^f,\hdots][P_V(1,a_1^f,a_2^f,\hdots)^{-1}\ |\ V\in A^*-\{\varepsilon\}] \end{equation}
where $f$ is a complete flag starting with $\varepsilon$ and $P_V$ is the linear combination expressing $\theta_V$ in terms of the coefficients with respect to $f$; see \cref{cor:polynomials} and \cref{lem:invertible}. Here we use that the \revm{elements} augmenting to $1$ are precisely the elements of the form $1 +\sum_{n\in \N^+} a_n^f y(W_n)$, with $y(W_n)$ as in the end of \cref{sec:basics}.

\revm{To conclude, let $F^{\uni}$ be the pushforward of the universal equivariant formal group law along $L_A\to S_A$, and let $\tilde{F}^{\uni}$ be the same formal group with the coordinate replaced with $(1 +\sum_{n\in \N^+} a_n^f y(W_n))\cdot y(\epsilon)$. The functor
\begin{align*} \mathrm{CAlg}_{S_A} &\to \SI \\
(f\colon S_A\to k) & \mapsto (k,f_*F^{\uni},f_*\widetilde{F}^{\uni},\id) \end{align*}
is fully faithful by \cref{sec:Lazard}, and we have demonstrated it to be essentially surjective. Thus, it is an equivalence of categories, as was to be shown.}
\end{proof}
\begin{remark}\label{rem:nonstrictisos}
    The same proof shows that the category of all (not necessarily strict) isomorphisms also has an initial object, whose underlying ring is $S_A[(a_0^f)^{\pm 1}]$. In fact, the only difference in the proof is that the equation for the unit $x$ is now of the form $\sum_{n\in\N}a_n^fy(W_n)$ so that the presentation for the analog of $S_A$ becomes 
    \[L_A[a_0^f, a_1^f,a_2^f,\hdots][(a_0^f)^{-1}, P_V(a_0^f,a_1^f,a_2^f,\hdots)^{-1}\ |\ V\in A^*-\{\varepsilon\}].\]
    Since $P_V(a_0^f,a_1^f,a_2^f,\hdots) = a_0^fP_V(1,\tfrac{a_1^f}{a_0^f},\tfrac{a_2^f}{a_0^f},\hdots)$, this ring is indeed isomorphic to 
    \[L_A[(a_0^f)^{\pm 1}, \tfrac{a_1^f}{a_0^f},\tfrac{a_2^f}{a_0^f},\hdots][P_V(1,\tfrac{a_1^f}{a_0^f},\tfrac{a_2^f}{a_0^f},\hdots)^{-1}\ |\ V\in A^*-\{\varepsilon\}] \cong S_A[(a_0^f)^{\pm 1}].\]
\end{remark}
There are functors
\[ s,t\colon \SI\to \text{A-FGL} \]
sending a strict isomorphism to its source and target, respectively, as well as an `identity' functor
\[\text{A-FGL} \to \SI, \]
an `inverse' functor
\[ i\colon \SI\to \SI \]
and a `composition' 
\[ c\colon \SI\times \SI\to \SI, \]
which restrict to the full subcategory of those objects of $\SI$ where the isomorphism $\varphi$ is given by the identity. By representability we obtain analogous source and target maps $s,t\colon L_A\to S_A$, an identity map $S_A\to L_A$, an inverse map $i\colon S_A\to S_A$ and a composition map $c\colon S_A\to S_A\otimes_{L_A} S_A$.
\begin{cor} The pair $(L_A,S_A)$ together with the above structure defines a Hopf algebroid i.e.\ a cogroupoid object in commutative rings. The associated functor
\[ \mathrm{CRing} \to \mathrm{Groupoids} \]
is equivalent to the one sending a commutative ring $k$ to the groupoid of $A$-equivariant formal group laws over $k$ and strict isomorphisms between them.
\end{cor}

\begin{remark}\label{rem:grading}
The Hopf algebroid $(L_A, S_A)$ has a natural grading. This can be constructed in three equivalent ways:
\begin{itemize}
    \item \cite[Corollary 5.6]{Hau} shows that the global Lazard ring admits a unique grading such that it defines a graded global group law. By the same arguments, the `universal global group law with a strict $n$-tuple of coordinates' $\mathbf{L}^{(n)}$ (cf., \cite[Section 5.8]{Hau}) also carries a unique grading for every $n$ such that the coordinates have degree $-2$. The source-target maps $s,t\colon \mathbf{L}\to \mathbf{L}^{(2)}$, the identity map $\mathbf{L}^{(2)}\to \mathbf{L}$, the inverse map $i\colon \mathbf{L}^{(2)}\to \mathbf{L}^{(2)}$ and the composition map $c\colon \mathbf{L}^{(2)}\to \mathbf{L}^{(3)}$ all preserve this grading, since they are defined through their effect on the respective coordinates. Evaluating these maps at a group $A$ yields the Hopf algebroids $(L_A,S_A)$, which hence inherit a grading compatible with all restriction and inflation maps.
    \item The groupoid-valued functor represented by $(L_A, S_A)$ admits a $\mathbb{G}_m$-action from multiplying the coordinate of the equivariant formal group law by a unit $u$ and acting on strict isomorphisms by multiplying $a_n^f$ by $u^n$. This corresponds to an even grading on $(L_A, S_A)$, putting e.g.\ $a_n^f$ in degree $2n$.
    \item By its interpretation as representing the groupoid of $A$-equivariant formal group laws and isomorphisms between them (see \cref{rem:nonstrictisos}), $(L_A, S_A[(a_0^f)^{-1}])$ obtains the structure of a Hopf algebroid. An element $s$ in $S_A$ is of degree $2n$ if applying the composition map $S_A[(a_0^f)^{-1}] 
    \to S_A[(a_0^f)^{-1}]\tensor_{L_A} S_A[(a_0^f)^{-1}]$ to $s$ gives $(a_0^f)^n c(s)$, where $c\colon S_A \to S_A \tensor_{L_A}S_A$ is the composition map of $(L_A, S_A)$.  
\end{itemize}
One can show that this is the same grading coming from the isomorphism 
\[(L_A, S_A) \cong (\pi_*^AMU_A, \pi_*^A MU_A \sm MU_A)\] 
from \cite[Theorem E]{Hau}.
\end{remark}

Given a strict isomorphism $\varphi\colon F_1\cong F_2$ of $B$-equivariant formal group laws and a group homomorphism $\alpha\colon B\to A$, we obtain an induced strict isomorphism $\alpha_*\varphi\colon \alpha_*F_1\cong \alpha_*F_2$ by completion. This assignment is compatible with composition of strict isomorphisms. Therefore, the functor $\mathbf{L}$ from \cref{sec:GlobalFunctorality} extends to a functor
\[ \mathbf{L}\colon \text{(abelian compact Lie groups)}^{\op}\to \text{Hopf algebroids}, \]
which we call the \emph{global Lazard Hopf algebroid}.

\subsection{The moduli stack of equivariant formal groups}\label{sec:moduli}
\rev{Throughout this subsection we let $A$ be an abelian compact Lie group. Our goal is to discuss the stack $\MM_{FG}^A$ of equivariant formal groups and its relationship to the Lazard Hopf algebroid. Our stacks will always be stacks for the fpqc-topology on commutative rings. More precisely, we will consider stacks as the full \revm{sub-$2$-category} of the $2$-category of (pseudo-)functors from commutative rings to groupoids on those (pseudo-)functors satisfying fpqc-descent on objects and morphisms (the usual stack conditions).}\footnote{\rev{We write (\emph{pseudo}-)functors since stacks are typically not strict functors into the $2$-category of groupoids. All $2$-morphisms in our setting are invertible; thus we can also decide to work in the setting of $(\infty,1)$-categories and look at all functors from commutative rings to the $(\infty,1)$-category of groupoids. In this setting, the `pseudoness' is \revm{implicit}; which is why we set `pseudo' in brackets.}} 

\rev{Let us define $\MM_{FG}^A$, leaving the verification that it is indeed an fpqc-stack to \cref{prop:MFGAHopfAlgebroid}.
\begin{defi}\label{def:MFGA}
    Let $\MM_{FG}^A$ be the (pseudo-)functor sending a commutative ring to the groupoid of $A$-equivariant formal groups over it (with morphisms arbitrary isomorphisms between them).
\end{defi}
}

We have discussed above that the Hopf algebroid $(L_A, S_A)$ represents the functor sending a commutative ring to the groupoid of $A$-equivariant formal group laws and strict isomorphisms between them. As discussed in \cref{rem:grading}, $(L_A, S_A)$ is naturally a graded Hopf algebroid. On the other hand, we have discussed in \cref{rem:nonstrictisos} the ungraded Hopf algebroid $(L_A, S_A[a_0^{\pm 1}])$ classifying $A$-equivariant formal group laws and all isomorphisms between them. It is easy to see that this is precisely the ungraded Hopf algebroid associated to the graded Hopf algebroid $(L_A, S_A)$ in the sense of \cite[Section 4.1]{MeierOzornova}.\footnote{\cite{MeierOzornova} uses an algebraic grading convention, while we use a topological one; thus one has to double all degrees.} We will follow \cite[Definition 4.1]{MeierOzornova} by defining the stack associated to a graded Hopf algebroid as the fpqc-stackification of the groupoid-valued functor corepresented by its associated ungraded Hopf algebroid. In particular, the stack associated to the graded Hopf algebroid $(L_A, S_A)$ is the same as the stack associated to the ungraded Hopf algebroid $(L_A, S_A[a_0^{\pm 1}])$. Equivalently, it is the quotient of the stack associated to $(L_A, S_A)$ by the $\mathbb{G}_m$-action induced by the grading.

\rev{
\begin{prop}\label{prop:MFGAHopfAlgebroid}
\begin{enumerate}
    \item\label{item:stack} The groupoid-valued presheaf $\MM_{FG}^A$ from \cref{def:MFGA} is an fpqc-stack.
    \item\label{item:equivalence} Sending an $A$-equivariant formal group law to its underlying $A$-equivariant formal group defines an equivalence from the stack associated to the graded Hopf algebroid $(L_A, S_A)$ to $\MM_{FG}^A$.
\end{enumerate}
\end{prop}
}
\begin{proof}
\rev{Let us first show that \eqref{item:stack} implies \eqref{item:equivalence}:} denote the stack associated to the graded Hopf algebroid $(L_A, S_A)$ by $\XX_A$; equivalently, this is the stack associated to the ungraded Hopf algebroid $(L_A, S_A[a_0^{\pm1}])$. Since the augmentation ideal of every $A$-equivariant formal group $(X, \varphi)$ is by definition fpqc-locally trivial, $(X,\varphi)$ comes after fpqc-base change from an $A$-equivariant formal group law (see \cref{lem:EFGEFGL}). Thus, $\XX_A\to \MM_{FG}^A$ is essentially surjective as a functor of stacks. Moreover, it is fully faithful since isomorphisms between $A$-equivariant formal group laws are precisely isomorphisms of the underlying $A$-equivariant formal groups. \rev{Thus, it remains to show \eqref{item:stack}, i.e.\ }that $\MM_{FG}^A$ satisfies fqpc-descent on morphisms and objects. 

Given two formal $k$-schemes $X$ and $Y$, the functor 
\[k\text{-algebras} \to \Set, \qquad K \mapsto \Hom_K(X\times_{\Spec k} \Spec K, Y\times_{\Spec k} \Spec K)\]
is an fpqc-sheaf. Indeed, if we view $X$ and $Y$ as ind-objects $(X_i)$ and $(Y_j)$, we can rewrite this $\Hom$ as $\lim_i\colim_j\Hom_K(X_i\times_{\Spec k} \Spec K, Y_j\times_{\Spec k} \Spec K)$ and equalizers commute with filtered colimits in sets. This easily implies that $\MM_{FG}^A$ satisfies fpqc-descent on morphisms.

Let $\Fin(A^*)$ denote the directed set of finite multi subsets\footnote{\rev{A \emph{multiset} is a version of a set where elements can occur more than once. Formally speaking, a \emph{multi subset} of a set $S$ is a function $S\to \mathbb{Z}_{\geq0}$, indicating which element occurs how often. It is \emph{finite} if the function has finite support. A multi subset $F$ of $S$ is \emph{contained} in another multi subset $G$ of $S$ iff $F(s)\leq G(s)$ for all $s\in S$.}} of $A^*$, ordered by inclusion. Sending an $A$-equivariant formal group $(X, \varphi)$ over $k$ to the system $(\Spec \OO_X/\II)$, where $\II$ runs over all finite products of the $\ker(\OO_X \xrightarrow{\phi(V)^*} k)$ for $V\in A^*$, defines a functor from $A$-equivariant formal groups over $k$ \rev{to the subcategory $\CC(k)$ of} $\Fun(\Fin(A^*), \Aff_k)$ \rev{where all transition maps are closed immersions.} \rev{The associated object in $\mathrm{Ind}(\Aff_k)$ defines the underlying formal scheme of $X$. Moreover, the functors $\MM_{FG}^A(k)\to \CC(k)$ assemble into a natural transformation $\MM_{FG}^A\to \CC$ of (pseudo)-functors from commutative rings to groupoids.} 

\rev{Given a descent datum $D$ for $A$-equivariant formal groups for an fpqc-cover $k\to K$, we thus obtain a descent diagram for $\CC$ along $k\to K$. This in turn defines a $\Fin(A^*)$-diagram of descent data for $\Aff_k$ along $k\to K$. As $\Aff$ and closed immersions satisfy fpqc-descent, we obtain a $\Fin(A^*)$-diagram in $\Aff_k$ such that all transition maps are closed immersions.}
This 
defines a formal $k$-scheme $X$. \rev{Denoting the functor
\[\text{(formal schemes over }k) \;\to \; \text{(descent data for formal schemes along } k\to K)\]
by $\DD$, one checks that $\DD(X)$ is the underlying descent datum of formal schemes for $D$. Moreover, $\DD$ preserves products.} By descent for morphisms between formal $k$-schemes, $X$ obtains a group structure and also a group homomorphism $\varphi\colon S\times A^* \to X$. Conditions \rev{\eqref{item:EFG1} and \eqref{item:EFG2}} for an $A$-equivariant formal group are fulfilled by construction. 
\end{proof}

\begin{remark}
    As every equivariant formal group comes Zariski-locally from an equivariant formal group law, in our case only a Zariski-stackification was necessary to pass from $(L_A, S_A[u^{\pm 1}])$ to $\MM_{FG}^A$. 
\end{remark}

\begin{prop}\label{prop:OpenClosedSubstacks}
    Let $B\subseteq A$ be a subgroup of an abelian group $A$. Denote by $\alpha\colon B \to A$ the inclusion and by $q\colon A\to A/B$ the projection. 
    \begin{enumerate}[(i)]
        \item  The functor $q^*$ induces an open immersion $\MM_{FG}^{A/B} \to \MM_{FG}^A$ whose image is the common non-vanishing locus of the Euler classes $e_V$ for all $V\notin \im((A/B)^*\to A^*)$. 
        \item  The functor $\alpha_*$ induces a closed immersion $\MM_{FG}^B \to \MM_{FG}^A$, inducing an equivalence of $\MM_{FG}^B$ to the common vanishing locus of the $e_V$ for all $V\in \ker(A^*\to B^*)$.
    \end{enumerate}
\end{prop}
\begin{proof}
    The first part is a reformulation of \cref{cor:fullyfaithful}. The immersion is open since it is open after pullback to $\Spec L_A$. 

    For the second, note that by \cref{prop:EulerClassInvertibleOrZero} every $A$-equivariant formal group such that $e_V = 0$ for all $V \in \ker(A^* \to B^*)$ is of the form $\alpha_*G$ for $G$ a $B$-equivariant formal group. Moreover, by construction, this vanishing of Euler classes is true for all $A$-equivariant formal groups of the form $\alpha_*G$ and thus characterizes the image of $\alpha_*$. The substack given by this image is closed since it is closed after pullback to $\Spec L_A$. Moreover, $\alpha_*$ is fully faithful by \cref{prop:alphastarfullyfaithful}.
\end{proof}

\begin{remark}
With notation as in the preceding proposition, we have $H \cong q_*q^*H$ for every $A/B$-equivariant formal group $H$. Thus, the inclusion of substacks $\MM_{FG}^{A/B} \to \MM_{FG}^A$ in the preceding proposition allows for a retraction (up to isomorphism) in the $2$-category of stacks.
\end{remark}

\revm{
\begin{remark}\label{rem:stratification}Let $A$ be a compact abelian Lie group. 
\cref{prop:OpenClosedSubstacks} allows to define for every sugroup $B\subseteq A$ an immersion $\MM_{FG}\to \MM_{FG}^A$ as the composition of the open immersion $\MM_{FG}^{B/B}\to \MM_{FG}^B$ and the closed immersion  $\MM_{FG}^B\to \MM_{FG}^A$. \cref{prop:ClassificationPointsMFGA} allows to view the collection of all such immersions as a stratification of $\MM_{FG}^A$. 
\end{remark}
}

\begin{example} \cref{prop:OpenClosedSubstacks} gives closed immersions of $\MM_{FG}$ and $\MM_{FG}^{C_2}$ into $\MM_{FG}^{C_4}$. The first is the common vanishing locus of all Euler classes (which equals the vanishing locus of the Euler class of one of the two generators of $(C_4)^*$, cf.\ \cref{prop:ClassificationOverFields}). Its complement is the open substack given by the non-vanishing locus of the Euler class of the generators and hence equivalent to $\MM_{FG}^{C_4/C_2}$.
    The second, i.e. the closed immersion of $\MM_{FG}^{C_2}$, equals the vanishing locus of $e_{[2]}$ for $[2]\colon C_4 \xrightarrow{[2]} C_4 \hookrightarrow \T$. Its complement is an open substack equivalent to $\MM_{FG}^{C_4/C_4}$, the non-vanishing locus of all Euler classes of non-trivial characters.
    
\definecolor{closed}{RGB}{170,170, 220}
\definecolor{open}{RGB}{255,238,240}
\definecolor{clopen}{RGB}{249,246,254}
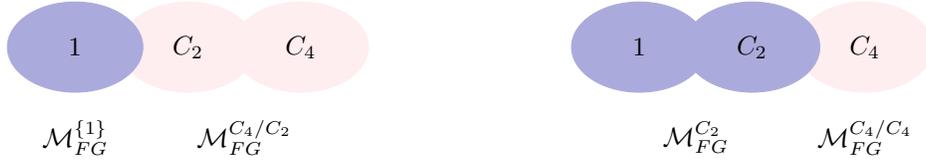
\begin{figure}
    \centering
    \label{fig:C4}
    \begin{tikzpicture}
        \def\un{0.15}
        \filldraw [open] (20*\un,0) ellipse ({6*\un} and {4*\un});
         \filldraw [open] (10*\un,0) circle ({6*\un} and {4*\un});
        \filldraw [closed] (0,0) ellipse ({6*\un} and {4*\un});
        \node at (0,0) {$1$};
        \node at (10*\un,0) {$C_2$};
        \node at (20*\un,0) {$C_4$};
        \node at (0,-8*\un) {$\MM_{FG}^{\{1\}}$};
        \node at (15*\un,-8*\un) {$\MM_{FG}^{C_4/C_2}$};
        \begin{scope}[shift={(50*\un,0)}]
        \filldraw [open] (20*\un,0) ellipse ({6*\un} and {4*\un});
        \filldraw [closed] (0,0) ellipse ({6*\un} and {4*\un});
        \filldraw [closed] (10*\un,0) circle ({6*\un} and {4*\un});
        \node at (0,0) {$1$};
        \node at (10*\un,0) {$C_2$};
        \node at (20*\un,0) {$C_4$};
        \node at (5*\un,-8*\un) {$\MM_{FG}^{C_2}$};
        \node at (20*\un,-8*\un) {$\MM_{FG}^{C_4/C_4}$};
        \end{scope}
    \end{tikzpicture}
    \caption{Decompositions of $\MM_{FG}^{C_4}$ into open and closed substacks, using misty rose for open and lavender for closed. \revm{Every ellipse stands for the image of an immersion $\MM_{FG}\to \MM_{FG}^{C_4}$, as in \cref{rem:stratification}.} }
    \end{figure}

For a \revm{general} group $A$ the situation is more complicated, and we cannot expect that the complement of the closed substack $\MM_{FG}^B$ in $\MM_{FG}^A$ can be expressed as a single open substack $\MM_{FG}^{A/C}$ in general and vice versa. In general, the complement of $\MM_{FG}^B$ in $\MM_{FG}^A$ can be written as the union of the open substacks $\MM_{FG}^{A/C}$ where $C$ runs over the minimal subgroups of $A$ not contained in $B$. We have indicated the situation for $A = C_2\times C_2$ in \cref{fig:C2xC2}.
   
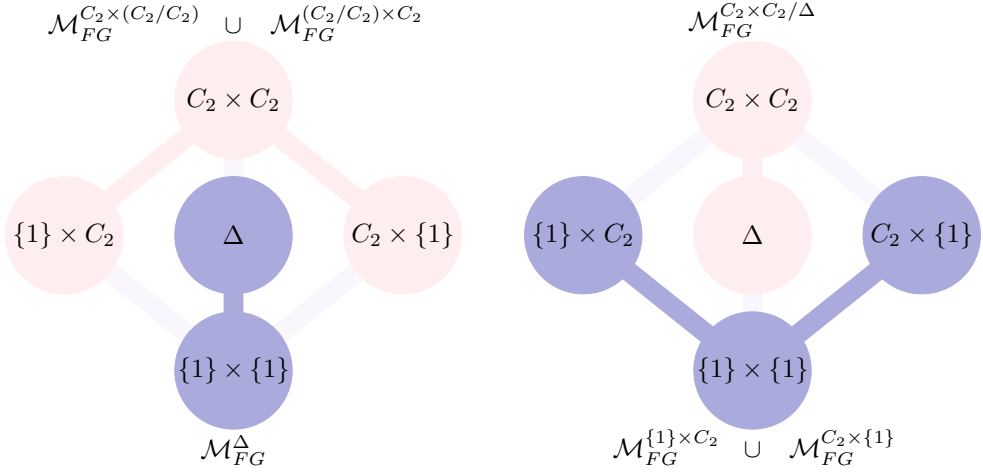
\begin{figure}
    \centering
    \begin{tikzpicture}
        \def\un{0.15}
        \def\radi{5.2}
        \begin{scope}[shift ={(-23*\un, 0)}]
         \draw [open, line width = 50*\un] (0,12*\un) -- (-15*\un,0);
         \draw [open, line width = 50*\un] (0,12*\un) -- (15*\un,0);
         \draw [closed, line width = 50*\un] (0,-12*\un) -- (0,0);
          \draw [clopen, line width = 50*\un] (0,12*\un) -- (0,0);
          \draw [clopen, line width = 50*\un] (0,-12*\un) -- (-15*\un,0);
          \draw [clopen, line width = 50*\un] (0,-12*\un) -- (15*\un,0);
        \filldraw [open] (0,12*\un) circle (\radi*\un);
        \node at (0,12*\un) {$C_2\times C_2$};
        \filldraw [open] (-15*\un, 0) circle (\radi*\un);
        \node at (-15*\un,0) {$\{1\}\times C_2$};
        \filldraw [open] (15*\un, 0) circle (\radi*\un);
        \node at (15*\un,0) {$C_2\times\{1\}$};
        \filldraw [closed] (0, 0) circle (\radi*\un);
        \node at (0,0) {$\Delta$};
        \filldraw [closed] (0,-12*\un) circle (\radi*\un);
       \node at (0,-12*\un) {$\{1\}\times \{1\}$};
       \node at (0, -19.1*\un) {$\MM_{FG}^{\Delta}$};
       \node at (0, 18.8*\un) {$\,\MM_{FG}^{C_2\times (C_2/C_2)}\quad\!\!\cup\quad\!\MM_{FG}^{(C_2/C_2)\times C_2}$};
        \end{scope}
        \begin{scope}[shift ={(23*\un, 0)}]
         \draw [clopen, line width = 50*\un] (0,12*\un) -- (-15*\un,0);
         \draw [clopen, line width = 50*\un] (0,12*\un) -- (15*\un,0);
         \draw [clopen, line width = 50*\un] (0,-12*\un) -- (0,0);
          \draw [open, line width = 50*\un] (0,12*\un) -- (0,0);
          \draw [closed, line width = 50*\un] (0,-12*\un) -- (-15*\un,0);
          \draw [closed, line width = 50*\un] (0,-12*\un) -- (15*\un,0);
        \filldraw [open] (0,12*\un) circle (\radi*\un);
        \node at (0,12*\un) {$C_2\times C_2$};
        \filldraw [closed] (-15*\un, 0) circle (\radi*\un);
        \node at (-15*\un,0) {$\{1\}\times C_2$};
        \filldraw [closed] (15*\un, 0) circle (\radi*\un);
        \node at (15*\un,0) {$C_2\times\{1\}$};
        \filldraw [open] (0, 0) circle (\radi*\un);
        \node at (0,0) {$\Delta$};
        \filldraw [closed] (0,-12*\un) circle (\radi*\un);
       \node at (0,-12*\un) {$\{1\}\times \{1\}$};
       \node at (0, 19.1*\un) {$\MM_{FG}^{C_2\times C_2/\Delta}$};
       \node at (0, -18.8*\un) {$\,\MM_{FG}^{\{1\}\times C_2}\quad\!\!\cup\quad\!\MM_{FG}^{C_2\times\{1\}}$};
        \end{scope}
   \end{tikzpicture}
    \caption{Decompositions of $\MM_{FG}^{C_2\times C_2}$ into open and closed substacks, $\Delta$ being the diagonal subgroup. \revm{Every circle stands for the image of an immersion $\MM_{FG}\to \MM_{FG}^{C_2\times C_2}$, as in \cref{rem:stratification}.}}
    \label{fig:C2xC2}
\end{figure}
\end{example}

\section[Points of $\MM_{FG}^A$ and invariant prime ideals]{Points of the moduli stack of equivariant formal groups and invariant prime ideals}
The goal of this section is to classify the points of $\MM_{FG}^A$ and the invariant prime ideals of $(L_A, S_A)$. Although the latter could be done without the former, we feel that both questions are of the same importance and the stack point of view makes some issues more transparent. 

\subsection{The space associated to a stack}
As mentioned above, given a graded flat Hopf algebroid $(A,\Gamma)$, we can associate an ungraded Hopf algebroid $(A, \Gamma[u^{\pm1}])$ to it (see e.g.\ \cite[Section 4.1]{MeierOzornova}).\footnote{Standard conventions force us to use $A$ as part of the notation of a general Hopf algebroid, while $A$ stands for a compact abelian Lie group in most of this article. We trust that this does not cause confusion.} The category of comodules over the latter is equivalent to that of graded comodules over $(A,\Gamma)$. The stack $\XX$ associated to $(A,\Gamma)$ is by definition the stack associated to $(A,\Gamma[u^{\pm1}])$, i.e.\ the fpqc-stackification of the presheaf of groupoids represented by $(A,\Gamma[u^{\pm1}])$ on the category of all schemes. We denote the resulting morphism $\Spec A \to \XX$ by $\pi$.

\begin{defi}
    Let $(A, \Gamma)$ be a Hopf algebroid with units $\eta_L$ and $\eta_R$. An ideal $I\subseteq A$ is called \emph{invariant} if $\eta_L(I)\Gamma = \eta_R(I)\Gamma$.  If $(A, \Gamma)$ is graded, we will assume that $I$ is also graded, i.e.\ generated by homogeneous elements.
\end{defi}

It is easy to check that invariant ideals in a graded Hopf algebroid $(A,\Gamma)$ correspond exactly to graded subcomodules of $A$ and thus to ideal sheaves on $\XX$. Here, we use that $\pi^*\colon \QCoh(\XX) \to \QCoh(\Spec A) \simeq \Mod_A$ refines to an equivalence from quasi-coherent $\OO_{\XX}$-modules to graded $(A,\Gamma)$-comodules \rev{\cite[Proposition 4.3]{MeierOzornova}}.  

Following \cite[Section 5]{LMB00} and \cite[\href{https://stacks.math.columbia.edu/tag/04XL}{Tag 04XL}]{STACKS} in the case of Artin stacks, we can associate a topological space to $\XX$. To that purpose, \rev{we define a morphism $\UU \to \XX$ of stacks to be an} \emph{open immersion} if the pullback $\UU \times_{\XX}\Spec A \to \Spec A$ is an open immersion.

\begin{defi}
    For $\XX$ as above, define the underlying set of $|\XX|$ to consist of equivalence classes of morphisms $x\colon \Spec K \to \XX$ for $K$ a field; the equivalence relation is generated by isomorphisms and $x \sim (\Spec L \to \Spec K \xrightarrow{x} \XX)$, where $L$ is a field extension of $K$.
    We call a subset of $|\XX|$ \emph{open} if it is the image of $|\UU| \to |\XX|$ for an open immersion $\UU \to \XX$.
\end{defi}

Equivalently, we can characterize the opens as the images of those opens in $\Spec A$ that are invariant, i.e.\ have the same preimage along both the left and right unit $\Spec \Gamma[u^{\pm1}] \to \Spec A$. Indeed: \rev{since morphisms of stacks satisfy descent along the fpqc-cover $\Spec A \to \XX$ and $\Spec A\times_{\XX}\Spec A \simeq \Spec \Gamma[u^{\pm1}]$,} an open immersion $\UU \to \XX$ corresponds to an open immersion $\VV \to \Spec A$ with an isomorphism $\VV \times_{\Spec A}\Spec \Gamma[u^{\pm1}] \cong \Spec \Gamma[u^{\pm1}]\times_{\Spec A}\VV$ over $\Spec \Gamma[u^{\pm1}]$ satisfying a cocycle condition. But the category of open immersions into some $X$ with isomorphisms over $X$ between them is equivalent to the discrete category of open subsets of $X$, yielding the equivalence \rev{we claimed}. Since invariant opens in $\Spec A$ form a topology, we deduce that the opens in $|\XX|$ form a topology. One further checks that the map induced by any morphism of stacks is continuous. 

Our definition \rev{of the topological space associated to a stack} coincides with that of \cite{LMB00} and \cite{STACKS} \rev{(who do it for Artin stacks)} in the intersection of their domains, e.g.\ when $\XX$ is an affine scheme, where we get the usual topology. 

\begin{prop}\label{prop:InvariantPrimeIdealsAndPoints}
Let $(A,\Gamma)$ be a graded Hopf algebroid with associated stack $\XX$. 
Then:
\begin{enumerate}
    \item For every invariant prime ideal $I\subseteq A$, the image of $V(I) \subseteq |\Spec A|$ is closed in $|\XX|$ and $\eta_I = |\pi|(\eta)$ for $\eta$ the generic point of $V(I)$ (i.e.\ the point in $|\Spec A|$ corresponding to $I$) is generic (i.e.\ $\overline{\{\eta_I\}} = |\pi|(V(I))$).
    \item Mapping $I$ to $\eta_I$ defines an injection from the set of invariant prime ideals $\Spec^{\inv}(A)$ in $A$ to $|\XX|$. Equipping $\Spec^{\inv}(A)$ with the subspace topology from $\Spec A$, this map is continuous; if it is a bijection, it is a homeomorphism. 
\end{enumerate}
\end{prop}

\begin{proof}
For the first point, observe first that for every invariant ideal $I$, the set $V(I)$ is invariant and hence the complement of $V(I)$ defines an invariant open. Thus the image in $|\XX|$ is open. Moreover, $|\pi|^{-1}(|\pi|(V(I))) = V(I)$. Thus $|\pi|(V(I))$ is the complement of $|\pi|(|\Spec A| \setminus V(I))$ and thus closed. 

Assume now that $I$ is an invariant prime ideal. Let $\eta_I$ be the image of the generic point $\eta$ of $V(I)$. Then $\overline{\{\eta_I\}} \subseteq |\pi|(V(I)) = |\pi|\overline{\{\eta\}} \subseteq \overline{\{\eta_I\}}$ and hence $\overline{\{\eta_I\}} =  |\pi|(V(I))$.

For the injectivity of $\Spec^{\inv}(A) \to |\XX|$, let $I, J\subseteq A$ be two invariant prime ideals with $\eta_I = \eta_J$. By the first point, this implies that $|\pi|(V(I)) = |\pi|(V(J))$ and hence $V(I) = V(J)$. Thus, $I = J$. 

The continuity of $\Spec^{\inv}(A) \to |\XX|$ follows from that of $\Spec A \to |\XX|$. An arbitrary closed set of $\Spec^{\inv}(A)$ is of the form $V(I) \cap \Spec^{\inv}(A)$ for some ideal $I\subseteq A$. Set $I' = \bigcap_{I\subseteq J}J$, where the $J\subseteq A$ run over all invariant ideals containing $I$. Then  $V(I) \cap \Spec^{\inv}(A) = V(J)\cap \Spec^{\inv}(A)$. If $|\pi|\colon \Spec^{\inv}(A) \to |\XX|$ is a bijection, $|\pi|^{-1}(|\pi|(V(J))) = V(J)$ implies $|\pi|(V(J)\cap \Spec^{\inv}A) ) = |\pi|(V(J))$ and this is closed by the first part.
\end{proof}

We warn the reader that in general, the preimage of an irreducible closed subset of $|\XX|$ won't be irreducible in $|\Spec A|$ and thus does not correspond to an invariant prime ideal.

We recall that $|\MM_{FG}|$ has been computed by Honda: its points are classified by a pair $(p, n)$, where $p\geq 0$ is the characteristic of the field and $n \in \overline{\N}= \N \cup\{\infty\}$ is the height of the formal group, with $n=0$ iff $p=0$. \revm{More precisely, Honda defined for each prime field of characteristic $p$ the Honda formal group law of height $n$, and showed that any formal group law over a separably closed field is isomorphic to the pushforward of a Honda formal group law. As any field embeds into a separably closed field, the classification of points of $\MM_{FG}$ follows. We will often identify points of $\MM_{FG}$ with pairs $(p,n)$.} 

\begin{proposition}\label{prop:ClassificationPointsMFGA}
    Let $A$ be a compact abelian Lie group. \rev{We obtain a bijection of sets
    \[|\MM_{FG}^A| \to |\MM_{FG}| \times \Sub(A)\]
    (with $\Sub(A)$ being the set of closed subgroups of $A$) as follows:} For a given $A$-equivariant formal group $G$ over a field, the \rev{associated} point in $|\MM_{FG}|$ is \rev{defined by} the pushforward of $G$ along $p\colon A \to \{e\}$, and the \rev{associated subgroup of $A$} is Pontryagin dual to $A^*/\{V\in A^*: e_V = 0\}$.
\end{proposition}
\begin{proof}  Given a closed subgroup $B$ of $A$ and a non-equivariant formal group $\Gamma$ over a field $k$, we obtain an $A$-equivariant formal group via $q^*i_*\Gamma$, where $A \xrightarrow{q} A/B \xleftarrow{i} \{1\}$. By \cref{prop:ClassificationOverFields}, every $A$-equivariant formal group $G$ over a field $k$ is isomorphic to one of this form.
Here, the projection $A \to A/B$ is Pontryagin dual to the subgroup \rev{inclusion $\{V\in A^*: e_V = 0\}\subseteq A^*$,} and thus $B$ is uniquely defined by $G$. Moreover, necessarily $\Gamma \cong p_*G$. This implies that \rev{$|\MM_{FG}^A| \to |\MM_{FG}| \times \Sub(A)$ is surjective and} that two $A$-equivariant formal groups $G_1$ and $G_2$ defined over the same field $k$ are isomorphic if and only if they have the same associated closed subgroup and their completions $p_*G_1$ and $p_*G_2$ are isomorphic as non-equivariant formal groups.

\rev{To show that $|\MM_{FG}^A| \to |\MM_{FG}| \times \Sub(A)$ is injective, assume that two points
\[G_i\colon \Spec K_i\to \MM_{FG}^A\]
for $i=1,2$ define the same element of $|\MM_{FG}|\times \Sub(A)$ for separable closed fields $K_i$. 
By the above we know that $G_1\cong q^*i_*p_* G_1$ and $G_2\cong q^*i_*p_*G_2$, where $i$ and $q$ are defined with respect to the same subgroup $B$ of $A$. Moreover, the heights of $p_*G_1$ and $p_*G_2$ and the  characteristics of $K_1$ and $K_2$ agree by assumption. Let $\Gamma$ be a formal group law of the same height as the $p_*G_i$, but defined over the common prime field $K$ of the $K_i$. Both $p_*G_i$ are isomorphic to the pushforward of $\Gamma$ to the $K_i$. Thus, the $A$-equivariant formal group $q^*i_*\Gamma$ defines the same element in $|\MM_{FG}^A|$ as the $G_i \cong q^*i_*p_*G_i$.  Hence $G_1$ and $G_2$ define the same element of $|\MM_{FG}^A|$.}
\end{proof}

\begin{remark}\label{rem:IdentifyPointsAlongOpenImmersion}
    By \cref{prop:OpenClosedSubstacks}, there is an open immersion $\MM_{FG}^{A/B} \to \MM_{FG}^A$ for any closed subgroup $B\subseteq A$ and hence $|\MM_{FG}^{A/B}|$ is homeomorphic to an open subset of $|\MM_{FG}^A|$.  Given $C\subseteq A/B$, the point corresponding to $(C, p, n)$ is $(q^{-1}(C), p, n)$ for $q\colon A \to A/B$ the projection. Indeed, for a $A/B$-equivariant formal group $G$ over a field corresponding to $(C, p, n)$, we have 
    \[\{V\in A^*: e_V = 0\} = q^*(\{V\in (A/B)^*: e_V = 0\}) = q^*(((A/B)/C)^*).\] 
    The Pontryagin dual of $A^*/q^*(((A/B)/C)^*)$ is precisely $q^{-1}(C)$ since $A/q^{-1}(C) \cong (A/B)/C$. 
\end{remark}

\subsection{Invariant prime ideals of $(L_A, S_A)$}
\revm{Our goal in this section is to classify the invariant prime ideals of the Hopf algebroid $(L_A, S_A)$. For that purpose let us introduce and recall some notation.} 
As before, let $\Phi^BL =L_B[e_V^{-1}]$ be the localization of $L_B$ away from 
all Euler classes $e_V$ for $V\neq \epsilon$ and let $\Phi^B S = S_B \tensor_{L_B} \Phi^BL$. 
\revm{We denote by $v_n$ the coefficient of $x^{p^n}$ in the $p$-series $[p]_{F^{\uni}}(x)$ of the universal formal group law $F^{\uni}$ over $L$. We further} denote by $I_{p,n}$ the ideal $(\revm{v_0=}p, v_1, v_2, \dots, v_{n-1}) \subseteq L$, i.e.\ the unique invariant prime ideal at height $n$ containing $p$. We also include the case of $I_{p,0}=0$. In this section, we will denote this ideal by $I_{0,0}$ to uniformize notation with respect to the residue characteristic.\footnote{\revm{Our choice of $v_n$ is different from the more commonly chosen Araki generators $v_n'$, which satisfy $[p](x) = \sum_n^{F^{\uni}}v_n'x^{p^n}$ (see e.g.\ \cite[A.2.2.4]{Rav86}). Thus, modulo $I_{p,n}' = (v_0'=p, v_1', \dots, v_{n-1}')$, the $p$-series is of the form $v_n'x^{p^n}+\cdots$. Therefore, $v_n \equiv v_n' \mod I_{p,n}'$.  Inductively, we see that $I_{p,n}=I_{p,n}'$.}}

\begin{construction}For every triple $(B,p, n)$ of a closed subgroup $B$ of $A$, a prime $p$ and $n\in \overline{\mathbb{N}}$ we define an invariant ideal $I^A_{B,p, n}\subseteq L_A$ as the preimage of $\Phi^BL\cdot I_{p,n}\subseteq \Phi^B L$ along the composite map of Hopf algebroids
\[ (L_A,S_A) \to (L_B,S_B) \to (\Phi^B L,\Phi^B S). \]
\end{construction}
Note that the $\Phi^BL\cdot I_{p,n}$ are indeed prime ideals since by \cref{prop:GeometricFixedPoints} the quotient ring $\Phi^BL/\Phi^BL\cdot I_{p,n}$ is of the form $(L/I_{p,n})[(b_0^V)^{\pm 1}, b_i^V\ |\ i>0, V\in A^*-\{\epsilon\}]$ and hence an integral domain. Thus, the $I^A_{B,p,n}$ are prime as well. 

To simplify notation, we will from now on often write $\Phi^BL/I_{p,n}$ instead of $\Phi^BL/\Phi^BL\cdot I_{p,n}$ and likewise in similar situations.

\begin{theorem}\label{thm:ClassificationOfInvariantPrimes} The assignment
    \begin{eqnarray*} \Sub(A)\times |\MM_{FG}| & \to & \Spec^{\inv}(L_A) \\
        (B,p,n)& \mapsto & I^A_{B,p,n} \end{eqnarray*}
    is a bijection. In other words, the ideals $I^A_{B,p,n}$ are pairwise different and constitute all the invariant prime ideals in $L_A$. 
    
    The map $\Spec^{\inv}(L_A) \to |\MM_{FG}^A|$ from \cref{prop:InvariantPrimeIdealsAndPoints} is a homeomorphism.
    \end{theorem}
\begin{proof}
    By \cref{prop:InvariantPrimeIdealsAndPoints}, we know that $\Spec^{\inv}(L_A)$ injects into $|\MM_{FG}^A|$ and the latter we computed to be $\Sub(A) \times |\MM_{FG}|$ as a set. Thus, it suffices to show that the element of $\Sub(A) \times |\MM_{FG}|$ associated to $I^A_{B, p,n} \in \Spec^{\inv}(L_A)$ is precisely $(B, (p,n))$, where $p$ is a prime number if $n>0$ and $0$ if $n=0$.  

    To spell this out concretely, let $k$ be the field of fractions of $L_A/I^A_{B, p,n}$. We denote by $F$ the pushed forward $A$-equivariant formal group law over $k$ and by $G$ the corresponding $A$-equivariant formal group. Since $\Spec k \to \Spec L_A$ hits the point corresponding to the prime ideal $I^A_{B, p,,n}$, the corresponding point in $|\MM_{FG}^A|$ is represented by $\Spec k \xrightarrow{G} \MM_{FG}^A$. By the classification in \cref{prop:ClassificationPointsMFGA}, we need to show three things:
    \begin{enumerate}
        \item the set of $V\in A^*$ such that $e_V = 0$ in $k$ is precisely $\ker(A^* \to B^*)$, 
        \item $k$ has characteristic $p$ (which is clear), and
        \item the pushforward $p_*G$ along $A \xrightarrow{p} \{1\}$ has height $n$. 
    \end{enumerate}
For the first, recall from \cref{lem:EulerClassesGenerateKernel} that $I^A_B = \ker(\res\colon L_A \to L_B)$ is generated by the Euler classes $e_V$ for all $V\in \ker(A^*\to B^*)$. These Euler classes must vanish in $k$ since $L_A \to k$ factors through $L_A/I_A^B$. If $V$ is not in $\ker(A^* \to B^*)$, then $e_V \neq 0$ in $L_B$ and hence also in $\Phi^BL/I_{p,n}$ (as else $\Phi^BL/I_{p,n} = 0$). Since $L_A/I^A_{B, p,,n}$ injects into $\Phi^BL/I_{p,n}$, the Euler class $e_V$ is actually nonzero in $L_A/I^A_{B, p,,n}$ and hence invertible in $k$. This shows the first point.

The pushforward $p_*G$ is classified by the composite
\[g\colon L \to L_A \to L_A/I^A_{B,p,n} \to k.\]
The ideal $I_{p,n}\cdot L_A$ maps to $0$ in $\Phi^BL/I_{p,n}$ and is hence contained in $I^A_{B,p,n}$. Therefore, $g$ factors through $L/I_{p,n}$ and the height of $p_*G$ is at least $n$. \revm{If $n=\infty$, we are done. If $n<\infty$, i}t remains to show that $v_n$ is non-zero in $L_A/I^A_{B,p,n}$ and hence in $k$. By definition of $I^A_{B,p,n}$, the map $L_A \to \Phi^BL/I_{p,n}$ factors over $L_A/I^A_{B,p,n}$. We know by \cref{prop:GeometricFixedPoints} that $\Phi^BL/I_{p,n}$ is an integral domain of the form $(L/I_{p,n})[(b_0^V)^{\pm 1}, b_i^V]$. In particular, $v_n$ is non-trivial in $\Phi^BL/I_{p,n}$ and hence in $L_A/I^A_{B,p,n}$. This shows that $p_*G$ is of height $n$ as desired, which finishes the proof.
\end{proof}

Unraveling the definition, we obtain the following description of the ideal $I^A_{B,p,n}$, where $I^A_B$ still denotes the kernel of the restriction map $L_B\to L_A$.
\begin{lemma} \label{lem:characterizationofideals} An element $x\in L_A$ lies in $I^A_{B,p,n}$ if and only if there exists an $A$-representation $W$ with $W^B=0$, such that
\[ x\cdot e_W \in (I^A_B,I_{p,n}).\]
\end{lemma}
\begin{proof} By definition, $I^A_{B,p,n}$ is the kernel of the composition
\[ L_A\xrightarrow{\res^A_B}L_B\to L_B/L_B\cdot I_{p,n}\to \Phi^BL/\Phi^BL\cdot I_{p,n}. \]
The composition $L_A\to L_B\to L_B/I_{p,n}$ is surjective with kernel $(I^A_B,I_{p,n})$, and the map $L_B/I_{p,n}\to \Phi^BL/I_{p,n}$ inverts all Euler classes $e_{\overline{W}}$ for $B$-representations $\overline{W}$ with $\overline{W}^B=0$. The product of Euler classes is an Euler class again. Hence, if $x\in L_A$ is contained in $I^A_{B,p,n}$, its image $\overline{x}$ in $L_B/I_{p,n}$ must be annihilated by such an Euler class $e_{\overline{W}}$.  We can extend $\overline{W}$ to an $A$-representation $W$ and find that $x\cdot e_W$ is contained in $(I^A_B,I_{p,n})$, as desired.

For the opposite direction, if we assume that $x\cdot e_W \in (I^A_B,I_{p,n})$ for some $A$-representation $W$ with $W^B=0$, then $\overline{x}\cdot e_{\res^A_BW}=0$ in $L_B/I_{p,n}$. Since $\res^A_BW$ has trivial $B$-fixed points, $e_{\res^A_BW}$ becomes invertible in $\Phi^BL/I_{p,n}$ and hence $\overline{x}$ is taken to $0$ there. Therefore, $x$ is contained in $I^A_{B,p,n}$.
\end{proof}
When $B$ is a torus, the ideal $I^A_{B,p,n}$ is easy to describe explicitly:
\begin{cor} \label{cor:idealsfortori} If $B$ is a torus, then $I^A_{B,p,n}=(I^A_B,I_{p,n})$.
\end{cor}
\begin{proof} When $B$ is a torus, every non-trivial character is non-torsion and thus the map
\[ L_B/I_{p,n}\to \Phi^BL/I_{p,n} \]
is injective (Corollary \ref{cor:regularmodIn}). Hence, $I_{B,p,n}^A$ is equal to the kernel of
\[ L_A\to L_A/I_{p,n}\to L_B/I_{p,n}, \]
which is generated by $I^A_B$ and $I_{p,n}$.
\end{proof}

We note the following useful corollary:
\begin{cor} 
\revm{Let $A$ be a torus, $I_{p,n}\subseteq L$ an invariant prime ideal with $n<\infty$,} and consider the augmentation ideal \[I=\ker(L_A/I_{p,n}\to L/I_{p,n}),\]
i.e., the ideal generated by all the Euler classes.
Then the intersection $J=\cap_{k\in \N}I^k\subseteq L_A/I_{p,n}$ equals the $0$-ideal.
\end{cor}
For example, this shows that no element in the $\T$-equivariant Lazard ring $L_\T$ is infinitely often divisible by the Euler class $e$.
\begin{proof} Since $I$ is an invariant ideal of $L_A/I_{p,n}$, so are all its powers and the intersection thereof. Moreover, $J$ equals the kernel of the completion map (cf.\ Section \ref{sec:completion})
\[ L_A/I_{p,n}\to L/I_{p,n}\llbracket y_1,\hdots,y_r\rrbracket ,\]
where $r$ is the rank of $A$ and the $y_i$ are the images of the Euler classes ranging through a basis of $A^*$. Since $L/I_{p,n}\llbracket y_1,\dots,y_r\rrbracket $ is an integral domain, $J$ must be prime and hence an invariant prime ideal.

Therefore $J$ must be of the form $I^{A}_{B,p,m}$ (or rather its image under the projection $L_{A}\to L_{A}/I_{p,n}$) for some subgroup $B$ and height $m\geq n$. Note that $v_n$ is not contained in $I$, hence in particular not in $J$. This means that we must have $m=n$. Moreover, given a character $V=V_{1}^{\otimes k_1}\otimes \dots \otimes V_{r}^{\otimes k_r}\in A^*$ expressed in the chosen basis $V_1,\hdots,V_r$ above, the image of $e_V$ under the completion map is given by
\[ F([k_1]_F(y_1),\dots,[k_r]_F(y_r)) \in L/I_{p,n}\llbracket y_1,\hdots,y_r\rrbracket,\]
where $F$ is the universal formal group law pushed forward to $L/I_{p,n}$. Since we assumed $n$ to be a finite height, the $[k]$-series of $F$ is non-trivial whenever $k$ is non-zero. It follows that for any non-trivial $V$ the image of $e_V$ is non-trivial in $L/I_{p,n}\llbracket y_1,\hdots,y_r\rrbracket $. Hence $J$ contains no Euler class $e_V$ and we must have $B=A$, i.e., 
\[ J=I^A_{A,p,n}=0\subseteq L_A/I_{p,n},\]
as desired.  
\end{proof}

\section{Inclusions between invariant prime ideals} \label{sec:inclusions}
Non-equivariantly the ideals $I_{p,n}\subseteq L$ form ascending towers
\[ (0)=I_{p,0}\subseteq I_{p,1}\subseteq I_{p,2} \subseteq \hdots, \]
essentially by definition. Except for the overlap at $(0)$, there are no inclusions between the towers for different primes $p$. For invariant prime ideals in the equivariant Lazard ring $L_A$ we saw that we have one tower
\[ I^A_{B,p,0}\subseteq I^A_{B,p,1} \subseteq I^A_{B,p,2} \subseteq \hdots \]
for every pair of a closed subgroup $B$ and prime $p$. Again, there will be no interplay between the towers associated to different primes (except for the overlap at height $0$). However, there are additional inclusions connecting the towers for different subgroups $B,B'$ at the same prime~$p$. This relationship between the heights at different subgroups is one of the essential properties of equivariant formal groups. It is closely related to the blue-shift phenomenon in stable homotopy theory. We say more about this in Section \ref{sec:comparison} below.

To see that there is no inclusion between towers associated with different primes we note that $p=v_0\in I^A_{B,p,n}$ whenever $n\geq 1$. It is easy to see that $p$ maps non-trivially under $L_A\to \Phi^{B'}L/I_{q,n'}$ for $q\neq p$ (since the target is free over $L/I_{q,n}$ by \cref{prop:GeometricFixedPoints}). 
Hence there cannot be an inclusion $I^A_{B,p,n}\subseteq I^A_{B',q,n'}$ for $n\geq 1$ and $p\neq q$. Moreover, if $n=0$ we have $I^A_{B,p,0}=I^A_{B,q,0}$. Hence we can reduce to studying containments between invariant prime ideals associated to the same prime $p$.

\medskip

\textbf{New convention}: For this reason and to simplify notation we from now on and for the rest of the paper implicitly localize at a fixed prime $p$. That is, we consider the $p$-localized Lazard ring $L_A$ and denote its invariant prime ideals simply by $I^A_{B,n}$, omitting the chosen prime $p$. We further sometimes abbreviate $I^A_{A,n}$ to $I_{A,n}$.
\medskip

Hence our goal is to understand for which pairs of subgroups $B,B'$ and natural numbers $n,n'$ there is an inclusion
\[  I^A_{B,n} \subseteq I^A_{B',n'}.\]
We will show the following:
\begin{theorem} \label{thm:inclusions} There is an inclusion $I^A_{B,n} \subseteq I^A_{B',n'}$ if and only if the following conditions are satisfied:
\begin{enumerate}
    \item $B'$ is a subgroup of $B$ and $\pi_0(B/B')$ is a $p$-group.
    \item We have $n'\geq n+\rank_p(\pi_0(B/B'))$.
\end{enumerate}
\end{theorem}
Hence, for example there are inclusions $I^\T_{\mathbb{T},n}\subseteq I^\T_{1,n}$ and $I^{C_{p^k}}_{C_{p^k},n}\subseteq I^{C_{p^k}}_{1,n+1}$, but $I^{C_{p}^k}_{C_{p}^k,n}$ is not contained in $I^{C_{p}^k}_{1,n+k-1}$. In fact the theorem can be formally reduced to checking those three special cases, as we will see below. Note also that the theorem in particular says that given a chain of inclusions $B'\subseteq B\subseteq A$, the question whether $I^A_{B,n}$ is contained in $I^A_{B',n'}$ does not depend on the ambient group $A$, but only on $B,B',n$ and $n'$.

\cref{thm:inclusions} can be interpreted as a statement about the heights of geometric fixed points of localizations of $L_A$, in the following way. Recall from \cite[Definition 4.1]{HoveyStricklandComodules} that the \emph{height} $ht(R)$ of an $L$-algebra $R$ is the maximal $n$ such that $R/(I_n\cdot R)\neq 0$; equivalently, it is the minimal number $n$ such that $I_{n+1}\cdot R=R$. If there is no such $n$, the height is understood to be infinite. If $R=0$, the height is $-1$. \revm{Denoting by $\Phi^A((L_A)_{I^A_{B,n}})$ the localization of $L_A$ inverting all Euler classes and all elements outside of $I^A_{B,n}$,} we have the following:
\begin{cor} \label{cor:blueshift} Let $B\subseteq A$ be a closed subgroup, $n\in \mathbb{N}$. If $\pi_0(A/B)$ is not a $p$-group, or if $\pi_0(A/B)$ is a $p$-group but $\rank_p(\pi_0(A/B))>n$, then the geometric fixed points $\Phi^A((L_A)_{I^A_{B,n}})$ are trivial. Otherwise, their height is given by
\[ ht(\Phi^A((L_A)_{I^A_{B,n}})) = n - \rank_p(\pi_0(A/B)). \]
\end{cor}
\begin{proof} We have $I_m\cdot \Phi^A((L_A)_{I^A_{B,n}})=\Phi^A((L_A)_{I^A_{B,n}})$ if and only if there exists an element of $L_A$ not contained in $I^A_{B,n}$ which is mapped to $I_m\cdot \Phi^A L_A$ under the geometric fixed point map. Since $I_{A,m}$ is defined precisely as the preimage of $I_m\cdot \Phi^A L_A$, this in turn is equivalent to $I^A_{A,m}$ not being contained in $I^A_{B,n}$. 

By Theorem \ref{thm:inclusions} we know that if $\pi_0(A/B)$ is not a $p$-group or if $\pi_0(A/B)$ is a $p$-group but $\rank_p(\pi_0(A/B))>n$, then $I^A_{A,0}$ is not contained in $I^A_{B,n}$. Since $I_0=(0)$, this implies that the geometric fixed points $\Phi^A((L_A)_{I^A_{B,n}}) = I_0\cdot \Phi^A((L_A)_{I^A_{B,n}})$ are trivial. 

If $\pi_0(A/B)$ is a $p$-group and $r=\rank_p(\pi_0(A/B))\leq n$, then the theorem tells us that $I^A_{A,n-r}\subseteq I^A_{B,n}$ and $I^A_{A,n-r+1}\not\subseteq I^A_{B,n}$. Hence we have $I_{n-r}\cdot \Phi^A((L_A)_{I^A_{B,n}})\neq \Phi^A((L_A)_{I^A_{B,n}})$ and $I^A_{n-r+1}\cdot \Phi^A((L_A)_{I^A_{B,n}})= \Phi^A((L_A)_{I^A_{B,n}})$, as claimed.
\end{proof}

\begin{remark} The techniques of this paper can be used to compute the height of geometric fixed points for many complex oriented theories. We give one example of this in Proposition \ref{prop:blueshift} below.
\end{remark}

\begin{remark}\label{rem:InclusionsAndStackClosure}
    There is an inclusion $I^A_{B,n} \subseteq I^A_{B',n'}$ if and only if in $|\MM_{FG, (p)}^A|$, the point corresponding to $I^A_{B',n'}$ lies in the closure of the point corresponding to $I^A_{B,n}$. Thus, \cref{thm:inclusions} can be interpreted as a result about the topology of $|\MM_{FG, (p)}^A|$. 
\end{remark}

The proof of Theorem \ref{thm:inclusions} takes up the remainder of this section.

\subsection{Formal reduction to the $p$-toral case} \label{sec:reduction} Our first step is the following:
\begin{lemma} \label{lem:mustbesubgroup} If there is an inclusion $I^A_{B,n} \subseteq I^A_{B',n'}$, then $B'$ is a subgroup of $B$ and $n'\geq n$.
\end{lemma}
\begin{proof} When $B'$ is not a subgroup of $B$ we can choose a character $V\in A^*$ which is trivial when restricted to $B$ but non-trivial when restricted to $B'$. Its Euler class $e_V$ then restricts to $0$ in $L_B$, in particular it is contained in $I^A_{B,n}$ for all $n$. On the other hand, its restriction to $B'$ becomes an invertible element in the non-trivial ring $\Phi^{B'}L/I_{n'}$ and is hence non-trivial. Therefore $e_V$ is not an element of $I_{B',n'}^A$. It follows that $I^A_{B,n}$ cannot be contained in $I^A_{B',n'}$.

If $n'<n$, then $v_{n'}$ is contained in $I_n$ but not in $I_{n'}$. This implies that $v_{n'}$ (now thought of as an element of $L_A$) is contained in $I^A_{B,n}$ but not in $I^A_{B',n'}$, since $\Phi^{B'}L/I_{n'}$ is a non-trivial free module over $L/I_{n'}$ by \cref{prop:GeometricFixedPoints}. Hence, again, $I^A_{B,n}$ cannot be contained in $I^A_{B',n'}$.
\end{proof}

\begin{lemma} \label{lem:reduction} Let $B'\subseteq B$ be an inclusion of subgroups of $A$ and $n,n'\in \mathbb{N}$. Then there is an inclusion 
\[ I_{B,n}^A \subseteq I_{B',n'}^A\]
if and only if there is an inclusion
\[ I_{B,n}^B \subseteq I_{B',n'}^B\]
if and only if there is an inclusion 
\[ I_{B/B',n}^{B/B'}\subseteq I_{B'/B',n'}^{B/B'}.\]
\end{lemma}
\begin{proof} The first two statements are equivalent since the restriction map $\res^A_B\colon L_A\to L_B$ identifies $L_B$ with a quotient of $L_A$, and the ideals $I^A_{B,n}$ and $I^A_{B',n'}$ are the preimages of the ideals $I^B_{B,n}$ and $I^B_{B',n'}$ under the quotient projection.

Phrased differently, the projection $L_A\to L_B$ induces a closed embedding of the stack of $B$-equivariant formal groups into the stack of $A$-equivariant formal groups. On spectra, the image consists precisely of these $I^A_{B'',n}$ with $B''\subseteq B$. This implies the desired equivalence by \cref{rem:InclusionsAndStackClosure}.

For the second equivalence, recall the open embedding $|\MM_{FG}^{B/B'}| \subseteq |\MM_{FG}^B|$ from \cref{rem:IdentifyPointsAlongOpenImmersion}, sending the point $(B''/B',n)$ to $(B'',n)$ for every $B'\subseteq B''\subseteq B$. 
Since the closure relation among points in a subspace can be detected in the subspace, \cref{rem:InclusionsAndStackClosure} gives the result.
\end{proof}
Taken together, the previous two lemmas allow us to reduce to the case $B=A$ and $B'=1$ and understand under what conditions there is an inclusion
\[ I^A_{A,n} \subseteq I^A_{1,n'}, \]
or in other words whether the restriction map $L_A\to L$ maps $I^A_{A,n}$ into the ideal $I_{n'}$. 

Our next goal is to show that we can further reduce to the case where $\pi_0 A$ is a $p$-group. 
For this we choose a prime $q$ and consider the Euler class $e_{\tau^q}\in L_\T$, i.e., the pullback of $e_\tau\in L_\T$ along the $q$th power map $[q]\colon \T\to \T$. The Euler class $e_{\tau^q}$ restricts to $0$ at the trivial group and is hence uniquely divisible by $e_\tau$. We set $\mv_0^{(q)}\in L_\T$ to be the unique element satisfying $e_{\tau^q}= \mv_0^{(q)}\cdot e_\tau$ (the reason for this choice of notation will become clear in Section \ref{sec:noinclusions}). 

Under the completion map $L_\T\to L\llbracket e_\tau\rrbracket $,
the Euler class $e_{\tau^q}$ is sent to the $q$-series $[q]_F(e_\tau)$ of the universal formal group law. Hence, $\mv_0^{(q)}$ is sent to the quotient $[q]_F(e_\tau)/e_\tau$, whose leading coefficient equals $q$. Since the restriction of any element in $L_\T$ to the trivial group equals the leading coefficient of its image in $L\llbracket e_\tau\rrbracket $, we see that $\res^\T_1 \mv_0^{(q)} = q\in L.$
We further set $\overline{\mv}_0^{(q)}\in L_{C_q}$ to be the restriction of $ \mv_0^{(q)}$, and find that it satisfies:
\begin{align*}
     \overline{\mv}_0^{(q)} \cdot e_{\overline{\tau}} & =  0 \\
    \res^{C_q}_1  \overline{\mv}_0^{(q)} &= q 
\end{align*}
Here, $\overline{\tau}$ denotes the restriction of the tautological character $\tau\in \T^*$ to $C_q$. We are now ready to show:
\begin{lemma} If $\pi_0A$ is not a $p$-group, then there is no inclusion of the form $I^A_{A,n}\subseteq I^A_{1,n'}$.
\end{lemma}
\begin{proof} If $\pi_0 A$ is not a $p$-group we can choose a surjection $f\colon A\to C_q$ with $q\neq p$ a prime. Then the element $f^*\overline{\mv}_0^{(q)}\in L_A$ satisfies the equation $f^*\overline{\mv}_0^{(q)}\cdot e_{f^*\overline{\tau}}=0$. Since $f$ is surjective, the character $f^*\overline{\tau}\in A^*$ is non-trivial. Hence, $f^*\overline{\mv}_0^{(q)}$ is an element of $I^A_{A,0}$ and hence also of $I^A_{A,n}$. Its restriction to the trivial group equals that of $ \overline{\mv}_0^{(q)}$, which is $q$ and hence a unit in the ($p$-localized) ring $L/I_{n'}$. In other words, $f^* \overline{\mv}_0^{(q)}$ is not an element of $I^A_{1,n'}$. Hence $I^A_{A,0}$ does not include into~$I^A_{1,n'}$. \end{proof}

Combined with Lemmas \ref{lem:mustbesubgroup} and \ref{lem:reduction} we obtain:
\begin{cor}
\label{cor:necessary} If there is an inclusion $I^A_{B,n} \subseteq I^A_{B',n'}$, then $B'$ is a subgroup of $B$, $n'\geq n$ and the quotient $B/B'$ is $p$-toral, i.e.\ a product of a $p$-group and a torus.
\end{cor}

\subsection{Proof of inclusions} \label{sec:noinclusions}
The next goal is to prove the `if' part of Theorem \ref{thm:inclusions}, i.e., to show that if the conditions on $B,B',n$ and $n'$ stated there are satisfied we do have an inclusion $I^A_{B,n}\subseteq I^A_{B',n'}$. We start with the easiest case:
\begin{lemma} \label{lem:toralinclusions} Let $B'\subseteq B$ be an inclusion of subgroups of $A$ such that $B/B'$ is a torus. Then there is an inclusion $I^A_{B,n}\subseteq I^A_{B',n}$ for all $n$.
\end{lemma}
\begin{proof} By Lemma \ref{lem:reduction} we can reduce to $B=A$ and $B'=1$ the trivial group. Hence $A$ is a torus.  Lemma \ref{cor:idealsfortori} then implies that $I^A_{1,n}$ is equal to the ideal generated by $I_n$ and the augmentation ideal $I^A_1$, whereas $I^A_{A,n}$ is generated by $I_n$ only. Clearly the latter is contained in the former.
\end{proof}
We now turn to showing that there are inclusions $I^{C_{p^k}}_{C_{p^k},n}\subseteq I^{C_{p^k}}_{1,n+1}$. 
For this we show that $I^{C_{p^k}}_{C_{p^k},n}$ is generated by $I_n$ plus one additional element which reduces to $v_n$ under the restriction map $L_{C_{p^k}}/I_n\to L/I_n$. We start with the case $k=1$ and recall again that the Euler class $e_{\tau^p}$ is sent to the $p$-series $[p]_F(e_\tau)$ under the completion map $L_\T\to L\llbracket e_\tau\rrbracket $. Modulo $I_n$, this $p$-series is of the form $v_n e_{\tau}^{p^n}$ + higher order terms. Hence \cref{cor:division} implies that there exists a unique element $\psi_p^{(n)}\in L_\T/I_n$ such that $e_{\tau^p}=\psi_p^{(n)} e_\tau^{p^n}$. This element satisfies $\res^\T_1(\psi_p^{(n)})=v_n$\revm{, i.e.\ it is the coefficient of $x^{p^n}$ in the $p$-series of the universal formal group law.} 

We then set
\[ \psi_{p^k}^{(n)}=[p^{k-1}]^*(\psi_p^{(n)})\in L_\T/I_n \]
for all $k\geq 2$, where $[p^{k-1}]$ is the multiplication-by-$p^{k-1}$ map on the circle. By functoriality, $\psi_{p^k}^{(n)}$ also restricts to $v_n$ at the trivial group. In fact, it already restricts to $v_n$ at $L_{C_{p^{k-1}}}/I_n$. This is because $C_{p^{k-1}}$ is the kernel of $[p^{k-1}]$ and hence the restriction map factors through the trivial group. Applying $[p^{k-1}]^*$ to the defining equation for $\psi_p^{(n)}$, we also obtain $e_{\tauexp{p^k}} = \psi_{p^k}^{(n)}e_{\tauexp{p^{k-1}}}^{p^n}$; here and in the following, we will often abbreviate $e_{\tau^m}$ to $e_m$. 

\begin{prop} \label{prop:generators} For every $k\geq 1$ and $n\in \mathbb{N}$ the element $\psi_{p^k}^{(n)}$ generates the kernel of
    \[ \phi^\T_{C_{p^k}}\colon L_\T/I_n\to L_{C_{p^k}}/I_n\to \Phi^{C_{p^k}}L/I_n. \]
Hence, $I^\T_{C_{p^k},n}$ is generated by the regular sequence $v_0,v_1,\hdots,v_{n-1},\psi_{p^k}^{(n)}$.
\end{prop}
\begin{cor} \label{cor:generators} The ideal $I^{C_{p^k}}_{C_{p^k},n}$ is generated by $I_n$ and the restriction $\overline{\psi}^{(n)}_{p^k}$ of $\psi_{p^k}^{(n)}$ to $L_{C_{p^k}}/I_n$.
\end{cor}
\begin{proof}[Proof of Proposition \ref{prop:generators}] Let $x\in L_{\T}/I_n$ be an element mapping to $0$ in $\Phi^{C_{p^k}}L/I_n$, i.e.\ \revm{an element} in the image of $I_{C_{p^k}, n}^{\T}$\revm{ under the reduction map $L_\T\to L_\T/I_n$}. By Lemma \ref{lem:characterizationofideals} we know that we have an equation of the form
\begin{equation} \label{eq:kernel} x\cdot e_1^{b_1}\cdots e_{\tauexp{p^{k-1}}}^{b_{p^k-1}}=y\cdot e_{\tauexp{p^k}}=y'\cdot \psi_{p^k}^{(n)} \end{equation}
for some $b_i\in \N$, $y\in L_\T/I_n$ and $y'=y\cdot (e_{\tauexp{p^k}}/\psi_{p^k}^{(n)})$ since $I_{C_{p^k}}^{\T} = (e_{\tauexp{p^k}})$ by \cref{lem:EulerClassesGenerateKernel}. 

If $l$ is coprime to $p$, then $e_{\tauexp{p^il}}$ and $e_{\tauexp{p^i}}$ become multiples of one another modulo $e_{\tauexp{p^k}}$: indeed, the corresponding characters in $(C_{p^k})^*$ generate the same subgroup and thus \cref{lem:EulerClassesGenerateKernel} implies that $e_{\tauexp{p^il}}$ and $e_{\tauexp{p^i}}$ generate the same ideal in $L_{C_{p^k}}$. It follows that Equation \eqref{eq:kernel}  gives rise to an equation
\[ x\cdot e_{1}^{a_0}\cdot e_{\tauexp{p}}^{a_1}\cdots e_{\tauexp{p^{k-1}}}^{a_{k-1}}=y''\cdot \psi_{p^k}^{(n)}, \]
with only Euler classes for powers of $p$ appearing on the left hand side. We claim that $y''$ must be divisible by the entire product $e_{1}^{a_0}\cdot e_{p}^{a_1}\cdots e_{p^{k-1}}^{a_{k-1}}$.
 \revm{ Since $L_\T/I_n$ is an integral domain, this implies that $x$ is divisible by $\psi_{p^k}^{(n)}$, as desired.} To see the claim, recall that $\psi_{p^k}^{(n)}$ restricts to $v_n$ in $L_{C_{p^l}}/I_n$ for all $l<k$. Hence if $a_l>0$ we have that 
 \[ 0=\res^\T_{C_{p^l}} (y''\cdot \psi_m^{(n)})=\res^\T_{C_{p^l}} (y'')\cdot v_n\in L_{C_{p^l}}/I_n. \]
 Since $v_n$ is a regular element in $L_{C_{p^l}}/I_n$ by \cref{prop:free}, this implies that $\res^\T_{C_{p^l}} (y'')=0$ and hence $y''$ is (uniquely) divisible by $e_{\tauexp{p^l}}$. This argument can be iterated by replacing $y''$ by $y''/e_{\tauexp{p^l}}$, and the statement follows.
\end{proof}
\begin{remark} One can show that more generally there exist elements $\psi_m^{(n)}\in L_\T/I_n$ for all $m\in \mathbb{N}$ uniquely determined by the equations
\begin{equation} \label{eq:psitn} e_{m}= \prod_{t|m} (\psi_t^{(n)})^{p^{\nu_p(\frac{m}{t})\cdot n}},\end{equation}
where $\nu_p(-)$ denotes the $p$-adic valuation of a natural number. The element $\psi_m^{(n)}$ generates the kernel of
    \[ \phi^\T_{C_m}\colon L_\T/I_n\to L_{C_m}/I_n\to \Phi^{C_m}L/I_n\]
and its restriction to the trivial group is given by
\[\res^\T_1(\psi^{(n)}_m)=\begin{cases} 0 & \text{if $m=1$} \\
        v_{n} & \text{if $m=p^l$ and $l>0$} \\
        q & \text{if $m=q^l$ with $q\neq p$ prime and $l>0$} \\
                1 & \text{otherwise.} \end{cases} \]
We note also that every $\psi_m^{(n)}$ is prime, since the geometric fixed points $\Phi^{C_m}L/I_n$ are integral domains. The elements $\psi_m^{(0)}$ were previously considered in \cite[Proposition 5.46]{Hau}, denoted $\psi_m$ there.
\end{remark}

\begin{remark}\label{rem:generators} The proof of Proposition \ref{prop:generators} applies in a more general context. Let $X$ be a global group law in the sense of \cite[Definition 5.1]{Hau}. As the global Lazard ring $\mathbf{L}$ is the initial global group law, there is a unique map $\mathbf{L}\to X$. Assume that the map $L=\mathbf{L}(1)\to X(1)$ sends $I_n$ to $0$, and that for every $l=0,\hdots,k$ the Euler classes $e_{p^l}$ in $X(\T)$ are regular elements and $v_n$ is a regular element in $X(C_{p^l})=X(\T)/e_{p^l}$ (for example this is the case if $v_n$ is regular in $X(\T)$ and $e_{p^l}$ remains a regular element modulo $v_n$). Then the image of $\psi_{p^k}^{(n)}$ in $X(\T)$ generates the kernel of the composition
\[ \Phi^\T_{C_{p^k}}\colon X(\T)\to X(C_{p^k})\to \Phi^{C_{p^k}} X. \]
For example, this applies to the coefficients of many Borel-equivariant complex oriented spectra, which can be used to compute their blue-shift numbers. We make use of this in Proposition \ref{prop:blueshift}.
\end{remark}

\begin{cor} \label{cor:cyclicinclusion} We have an inclusion $I^{C_{p^k}}_{C_{p^k},n}\subseteq I^{C_{p^k}}_{1,n+1}$.
\end{cor}
\begin{proof} By Corollary \ref{cor:generators}, $I^{C_{p^k}}_{C_{p^k},n}$ is generated by $I_n$ and $\overline{\psi}^{(n)}_{p^k}$. Since $I_n$ is clearly contained in $I^{C_{p^k}}_{1,n+1}$ (even in $I^{C_{p^k}}_{1,n}$), we can reduce modulo $I_n$ and need to show that $\overline{\psi}^{(n)}_{p^k}$ is taken to $0$ under the composition
\[ L_{C_{p^k}}/I_n\xrightarrow{\res^\T_1} L/I_n \to L/I_{n+1}. \]
But this is clear, since $\overline{\psi}^{(n)}_{p^k}$ restricts to $v_n$ at the trivial group and $v_n$ lies inside $I_{n+1}$.
\end{proof}
We now have all the ingredients to prove the `if'-direction in Theorem \ref{thm:inclusions}:
\begin{cor} \label{cor:inclusions} Let $B'\subseteq B$ be an inclusion of subgroups of $A$ such that $B/B'$ is $p$-toral, and $n',n\in \N$ such that $n'\geq n+\rank_p(\pi_0(B/B'))$. Then there is an inclusion $I^A_{B,n}\subseteq I^A_{B',n'}$. 
\end{cor}
\begin{proof} By Lemma \ref{lem:reduction} we can assume that $A=B$ and that $B'=1$ is the trivial subgroup. Hence, $A$ is a $p$-toral group. Let $A^0$ denote the path component of the identity. We have $I^A_{A^0,n'}\subseteq I^A_{1,n'}$ by Lemma \ref{lem:toralinclusions}. Hence it suffices to show that $I^A_{A,n}$ is contained in $I^A_{A^0,n'}$. For this, making use of Lemma \ref{lem:reduction} once more, we can assume that $A^0=1$ and hence $A$ is a finite abelian $p$-group. We write $m=\rank_p(A)$ and choose a filtration of subgroups
\[ 1\subseteq A_1\subseteq A_2\subseteq \dots \subseteq A_m=A \]
such that every subquotient $A_i/A_{i-1}$ is a cyclic $p$-group. By Corollary \ref{cor:cyclicinclusion} and Lemma \ref{lem:reduction} we see that
\[ I^A_{A,n} \subseteq I^A_{A_{m-1},n+1} \subseteq I^A_{A_{m-2},n+2} \subseteq \dots \subseteq I^A_{1,n+m} \subseteq I^A_{1,n'}. \]
The final inclusion follows from the assumption $n'\geq n+\rank_p(\pi_0(B/B'))=n+m$. This finishes the proof.
\end{proof}

\subsection{Proof of non-inclusions}\label{subsec:non-inclusions}
For the `only if' direction we need to rule out further inclusions between prime ideals by constructing elements whose restrictions exhibit a large `height shift'. \revm{We will abbreviate $I^A_{A,n}$ to $I_{A,n}$ in the following, for various groups $A$.} The goal is for every $n\in \N$ to construct an element $x\in L_{C_p^{n+1}}$ which lies in the ideal $I_{C_p^{n+1},0}$ and whose restriction to the trivial group lies outside of $I_n\subseteq L$. 

It turns out to be more natural to define such an element modulo a subideal of $I_{C_p^{n+1},0}$, namely the \revm{ideal generated by the} inflation of the ideal $I_{C_p^n,0}$ along the projection $p_{C_p^n}\colon C_p^{n+1}\cong C_p^n\times C_p\to C_p^n$. That is, we construct an element
\[ \overline{\mv}_n \in I_{C_p^n\times C_p,0}/p^*_{C_p^n}I_{C_p^n,0}. \]
By Corollary \ref{cor:inclusions} the restriction map $L_{C_{p^n}}\to L$ takes $I_{C_p^n,0}$ into $I_n$. Therefore, the restriction map $L_{C_p^{n+1}}\to L$ takes $p^*I_{C_p^n,0}$ into $I_n$ and we obtain an induced restriction map
\[ L_{C_p^n\times C_p,0}/p^*_{C_p^n}I_{C_p^n,0}\to L/I_n. \]
We will see that, under this restriction, $\overline{\mv}_n$ is sent to $v_n$. Later in Section \ref{sec:generators} we will show that $\overline{\mv}_n$ in fact forms a generator of the quotient $I_{C_p^n\times C_p,0}/p^*_{C_p^n}I_{C_p^n,0}$ and that suitable inflations and restrictions of these elements generate all the invariant prime ideals at elementary abelian $p$-groups.

We now turn to the construction of the element $\overline{\mv}_n$. We set $A=C_p^n$ and first define an element $\mv_n$ in the ring $L_{A\times \T}/p_A^* I_{A,0}$, whose restriction to $L_{A\times C_p}/p^*_AI_{A,0}$ then yields $\overline{\mv}_n$. 

For every character $V\in A^*$, 
\cref{prop:regularglobal} yields a short exact sequence
\[0 \to L_{A\times \T} \xrightarrow{e_{V\otimes \tau}\cdot} L_{A\times \T} \xrightarrow{(\id, V^{-1})^*} L_A \to 0,\] 
where we use that $(\id, V^{-1})\colon A \to A\times \T$ identifies $A$ with the kernel of $(V\otimes \id)\colon A\times \T \to \T$. The inflation map $L_A \to L_{A\times \T}$ provides an $L_A$-linear splitting of this exact sequence if we view $L_{A\times \T}$ as an $L_A$-module via the same inflation map. Thus, we obtain a short exact sequence 
\begin{equation} \label{eq:exact} 0\to L_{A\times \T}/p_A^*I_{A,0} \xr{e_{V\otimes \tau}} L_{A\times \T}/p_A^*I_{A,0}\xr{(\id,V^{-1})^*} L_A/I_{A,0} \to 0.  \end{equation}
of $L_A/I_{A,0}$-modules.
\begin{remark} We will see below in Section \ref{sec:generators} that the ideal in $L_{A\times \T}$ generated by $p_A^*I_{A,0}$ equals the invariant prime ideal $I_{A\times \T,0}$. In particular, $L_{A\times \T}/p_A^* I_{A,0}$ is again an integral domain. At this point it is only clear that $p_A^*I_{A,0}$ is contained in $I_{A\times \T,0}$.
\end{remark}
We now consider the Euler class $e_{\epsilon \otimes \tau^p}\in L_{A\times \T}/I_{A,0}$. We have $(\id,V^{-1})^*(e_{\epsilon \otimes \tau^p})=e_{V^{-p}}=0$, since $A$ is an elementary abelian $p$-group and hence every character is $p$-torsion. By exactness, it follows that $e_{\epsilon \otimes \tau^p}$ is divisible by $e_{V\otimes \tau}$, for all $V\in A^*$. We want to define $\overline{\mv}_n$ as the quotient of $e_{\epsilon \otimes \tau^p}$ by the product over all the $e_{V\otimes \tau}$. For this we first need to check that the different $e_{V\otimes \tau}$ are coprime. To understand this, we consider the following:

\begin{lemma} \label{lem:divisibility}Let $x,y\in L_{A\times \T}/p_A^*I_{A,0}$, and assume that $y$ restricts to a non-zero element under $L_{A\times \T}/p_A^*I_{A,0} \xr{(\id,V^{-1})^*} L_A/I_{A,0}$. Then $x$ is divisible by $e_{V\otimes \tau}$ if and only if $x\cdot y$ is divisible by $e_{V\otimes \tau}$.
\end{lemma}
\begin{proof} By exactness of \eqref{eq:exact}, $x$ is divisible by $e_{V\otimes \tau}$ if and only if $(\id,V^{-1})^*x=0$ in $L_A/I_{A,0}$. Since $(\id,V^{-1})^*(y)$ is non-zero by assumption and $L_A/I_{A,0}$ is an integral domain, this is the case if and only if 
\[ (\id,V^{-1})^*(x\cdot y)=(\id,V^{-1})^*x\cdot (\id,V^{-1})^*y=0, \]
which in turn is equivalent to $x\cdot y$ being divisible by $e_{V\otimes \tau}$.
\end{proof}
Given a character $V_2\neq V$, the restriction of $e_{V_2\otimes \tau}$ along $(\id,V^{-1})^*$ equals the non-trivial Euler class $e_{V_2V^{-1}}\in L_A/I_{A,0}$. Hence the lemma applies to $y$ being any product of Euler classes of the form $e_{V'\otimes \tau}$ with $V'\neq V$. We find that in the quotient $L_{A\times \T}/p_A^*I_{A,0}$, the Euler class $e_{\epsilon\otimes \tau^p}$ is uniquely divisible by the product $\prod_{V\in A^*}e_{V\otimes \tau}$. To summarize:
\begin{definition} \label{def:vn} Let $n\in \N$  and $A=C_p^n$. We define $\mv_n \in L_{A\times \T}/p^*_A I_{A,0}$ to be the unique element satisfying
\[ e_{\epsilon\otimes \tau^p}=\mv_n \cdot \prod_{V\in A^*} e_{V\otimes \tau}\in L_{A \times \T}/p^*_A I_{A,0}, \]
and we set
\[ \overline{\mv}_n =\res^{A\times \T}_{A\times C_p}(\mv_n) \in L_{A\times C_p}/p_A^*I_{A,0}. \]
\end{definition}
\begin{remark}\label{rem:mv0x0}
    \lennart{The elements $\mv_0\in L_\T$ and $\overline{\mv}_0\in L_{C_p}$ agree with the elements $\mv_0^{(p)}$ and $\overline{\mv}_0^{(p)}$ respectively from \cref{sec:reduction}.
    
    Furthermore, the element $\psi^{(n)}_p\in L_\T/I_n$ introduced before \cref{prop:generators} is equal to the restriction of $\mv_n\in L_{A\times \T}/p^*I_{A,0}$ to $L_\T/I_n$. (Note that the latter makes sense since we know from \cref{cor:inclusions} that the restriction of $I_{A,0}\subseteq L_{C_p^n}$ to the non-equivariant Lazard ring $L$ lands in $I_n$.) Indeed, the defining relation for $\mv_n$ above restricts to the equation
    \[ e_{\tau^p}=\res^{A\times \T}_{\T} (\mv_n) \cdot \prod_{V\in A^*} \res^{A\times \T}_{\T}(e_{V\otimes \tau})=\res^{A\times \T}_{\T} (\mv_n)\cdot e_\tau^{p^n} \in L_T/I_n\]
    since each $e_{V\otimes \tau}$ restricts to $e_\tau$ at the circle group. Since $\psi^{(n)}_p$ was defined as the unique element of $L_T/I_n$ satisfying $e_{\tau^p}=\psi^{(n)}_p\cdot e_\tau^{p^n}$, the claim follows.
    Consequently, $\overline{\psi}^{(n)}_p\in L_{C_p}/I_n$ agrees with the restriction of $\overline{\mv}_n\in L_{A\times C_p}/p^*I_{A,0}$ to $L_{C_p}/I_n$.}
\end{remark}
\begin{remark} Lemma \ref{lem:divisibility} also applies in the ring $L_{A\times \T}$ itself (i.e., before quotienting by $p_A^*I_{A,0}$) if we demand that $(\id,V)^*y$ is a regular element, rather than just being non-zero. These two conditions are equivalent in $L_A/I_{A,0}$ since it is an integral domain. However, the Euler classes are not regular elements in $L_A$, hence the lemma does not apply for $y$ a product of the $e_{V'\otimes \tau}$. In fact $e_{\epsilon\otimes \tau^p}$ is not divisible by $\prod_{V\in A^*} e_{V\otimes \tau}$ before dividing out $p_A^* I_{A,0}$, even though it is divisible by each individual $e_{V\otimes \tau}$. One can see this by restricting from $A\times \T$ to $\T$:
If $e_{\epsilon \otimes \tau^p}$ was divisible by the product of all the $e_{V\otimes \tau}$, this would imply that its restriction $e_{\tau^p}\in L_\T$ was divisible by $e_\tau^{p^n}$, since every $e_{V\otimes \tau}$ restricts to $e_\tau$. But  $e_{\tau^p}\in L_\T$ is divisible by $e_\tau$ precisely once, since $e_{\tau^p}/e_\tau$ restricts to $p$ at the trivial group, cf. Section \ref{sec:inclusions}. It is for this reason that it is most natural to define $\mv_n$ and $\overline{\mv}_n$ in this quotient. As we will see now, this matches nicely with the fact that $v_n$ is most naturally defined in the quotient $L/I_n$.
\end{remark}

\begin{prop} \label{prop:xn}
\begin{enumerate}
    \item The element $\overline{\mv}_n$ defines a class in the ideal $I_{A\times C_p,0}/p_A^*I_{A,0}$, i.e., it is sent to $0$ under the map $L_{A\times C_p,0}/p_A^*I_{A,0}\to \Phi^{A\times C_p}L.$
    \item The restriction map
    \[ L_{A\times C_p,0}/p_A^*I_{A,0}\to L/I_n\]
    takes $\overline{\mv}_n$ to $v_n$.
\end{enumerate}
\end{prop}
\begin{proof} Part 1: The equation
\[ e_{\epsilon\otimes  \tau^p} = \mv_n \cdot \prod_{V\in A^*}e_{V\otimes \tau}\]
reduces to the equation
\[ 0 = \overline{\mv}_n \cdot \prod_{V\in A^*}e_{V\otimes \overline{\tau}}\]
in $L_{A\times C_p}/p_A^*I_{A,0}$, where $\overline{\tau}$ denotes the restriction of $\tau\in \T^*$ to $C_p$. Note that each $(A\times C_p)$-character of the form $V\otimes \overline{\tau}$ is non-trivial. Hence $\overline{\mv}_n$ forms Euler-power torsion and therefore maps to $0$ in the geometric fixed points. 

Part 2: By Remark \ref{rem:mv0x0}, the intermediate restriction $\res_\T^{A\times \T} \mv_n\in L_\T/I_n$ equals the element $\psi_{p}^{(n)}$ constructed before \cref{prop:generators}. Since $\res^\T_1\psi_{p}^{(n)}=v_n$, this finishes the proof.
\end{proof}

\begin{cor} \label{cor:resxn} If $x_n$ is a preimage of $\overline{\mv}_n$ under the projection $L_{C_p^{n+1}}\to L_{C_p^{n+1}}/p_{C_p^n}^*I_{C_p^n,0}$ and $B\subseteq C_p^{n+1}$ is a subgroup of rank $0\leq m\leq n+1$, then
\[ x_n \in I^{C_p^{n+1}}_{B,n+1-m}-I^{C_p^{n+1}}_{B,n-m}. \]
\end{cor}
\begin{proof} By the previous proposition, $x_n$ is an element of $I^{C_p^{n+1}}_{C_p^{n+1},0}$. As $A/B$ has rank $(n+1-m)$, we know by Corollary \ref{cor:inclusions} that $x_n$ must lie in $I^{C_p^{n+1}}_{B,n+1-m}$.

If $x_n$ were an element of $I^{C_p^{n+1}}_{B,n-m}$, then applying Corollary \ref{cor:inclusions} to the inclusion of the trivial group into $B$ shows that $x_n$ is also an element of $I^{C_p^{n+1}}_{1,n}$. This contradicts the fact that, modulo $I_n$, we have $\res^{C_p^{n+1}}_1(x_n)=\res^{C_p^{n+1}}_1(\overline{\mv}_n)=v_n$.
\end{proof}
\begin{cor} If $B'\subseteq B$ is a $p$-toral inclusion of subgroups of $A$ (i.e.\ $B/B'$ is $p$-toral) and $n'<n+\rank_p(\pi_0(B/B'))$, then  $I^A_{B,n}$ does not include into $I^A_{B',n'}$.
\end{cor}
\begin{proof} By Lemma \ref{lem:reduction} we can reduce to the case $A=B$ and $B'=1$. Let $r=\rank_p(\pi_0(A))$, and $q\colon A\to C_p^r$ be a surjection. Let $x_{n+r-1}\in L_{C_p^{n+r}}$ as in Corollary \ref{cor:resxn}. Then, by the corollary, the restriction $\res^{C_p^{n+r}}_{C_p^r}(x_{n+r-1})$ is an element of $I^{C_p^{r}}_{C_p^{r},n}$ but not an element of $I^{C_p^{r}}_{1,n+r-1}$. 
Therefore $x=q^*(\res^{C_p^{n+r}}_{C_p^r}(x_{n+r-1}))$ is an element of $I^A_{A,n}$ whose restriction to the trivial group is not contained in $I_{n+r-1}$. In other words, $x$ is an element of $I^A_{A,n}$ but not an element of $I^A_{1,n+r-1}$. Since by assumption we have $n'\leq n+r-1$ and hence $I^A_{1,n'}\subseteq I^A_{1,n+r-1}$, this proves that $I^A_{A,n}$ does not include into $I^A_{1,n'}$.
\end{proof}

Combined with Corollaries \ref{cor:necessary} and \ref{cor:inclusions}, this proves Theorem \ref{thm:inclusions}.

\section{Generators for invariant prime ideals} \label{sec:generators}

In this section we show that over elementary abelian $p$-groups the elements $\overline{\mv}_n$ -- together with the Euler classes -- generate all invariant prime ideals under restriction and inflation maps. 
More precisely, we show the following theorem:
\begin{thm} \label{thm:generators1}
\begin{enumerate}
    \item 
    For every torus $B$ and $n\in \mathbb{N}$, the ideal $I_{C_p^{n}\times B,0} = I^{C_p^n\times B}_{C_p^{n}\times B,0}$ is generated by the elements
\[ p_1^*(\overline{\mv}_0),p_2^*(\overline{\mv}_1),\hdots,p_{n-1}^*(\overline{\mv}_{n-2}),p_n^*(\overline{\mv}_{n-1}),\]
where $p_i\colon C_p^{n}\times B\to C_p^{i}$ is the projection to the first $i$ factors.
    \item For every $m\in \mathbb{N}$ and every inclusion $i\colon C_p^n\to C_p^{n+m}$, the restriction map 
    \[ (i\times B)^*\colon L_{C_p^{n+m}\times B}\to L_{C_p^n\times B} \]
    maps $I_{C_p^{n+m}\times B,0}$ surjectively onto $I_{C_p^n\times B,m}$.
\end{enumerate}
\end{thm}
\begin{rmk}
    Implicit in the statement of the theorem is that each $p_i^*(\overline{\mv}_{i-1})$ is well-defined modulo the ideal generated by $p_1^*(\overline{\mv}_{0}),\hdots,p_{i-1}^ *(\overline{\mv}_{i-2})$. By definition, $p_i^*(\overline{\mv}_{i-1})$ is an element of the quotient by the subideal generated by $p_{i-1}^*I_{C_p^{i-1},0}$. Applying the theorem to rank $n-1$ and $B=0$ we see that this ideal is indeed generated by $p_1^*(\overline{\mv}_{0}),\hdots,p_{i-1}^ *(\overline{\mv}_{i-2})$, so the sequence of elements makes sense. Hence, the theorem and sequence should be interpreted in an inductive manner.
\end{rmk}
Combining both parts it follows that $I_{C_p^{n}\times B,m}$ is generated by 
\[ (i\times B)^*p_1^*(\overline{\mv}_0),(i\times B)^*p_2^*(\overline{\mv}_1),\hdots,(i\times B)^*p_{n+m-1}^*(\overline{\mv}_{n+m-2}),(i\times B)^*\overline{\mv}_{n+m-1},\]
where $i\colon C_p^n\to C_p^{n+m}$ is any inclusion.
The choice of inclusion will generally affect the resulting generators. For example, setting $n=m=1$ and $B$ the trivial group: If we choose $i_1\colon C_p\to C_p^2$ to be the inclusion into the first factor, the composite $p_1\circ i_1$ becomes the identity. Hence we obtain that $I_{C_p,1}$ is generated by the elements $\overline{\mv}_0$ and $i_1^*(\overline{\mv}_1)$. If we alternatively use the inclusion $i_2\colon C_p\to C_p^2$ into the second factor the composite $p_1\circ i_2$ becomes the constant map, yielding the generators $v_0=p$ and $i_2^*(\overline{\mv}_1)$ (i.e., the same ones as in Corollary \ref{cor:generators}, as $\overline{\psi}_{p}^{(1)}$ equals $i_2^*(\overline{\mv}_1)$). Furthermore, it follows that generators for ideals of the form $I^{A}_{C_p^n\times B,m}$ with $B$ a torus can be obtained as the union of Euler classes $(e_V)_{V\in \mathcal{B}}$ for a basis $\mathcal{B}$ of $\ker(A^*\to (C_p^n\times B)^*)$ together with a choice of generators for $I^{C_p^n\times B}_{C_p^n\times B,m}$.

\lennart{We prove both parts of Theorem \ref{thm:generators1} by induction on $n$, alongside the following statement:
\[ \text{$(\star)$ For every torus $B$ and $m\in \N$ the ideal $I_{C_p^n\times B,m}$ is generated by $p^*_{C_p^n}I_{C_p^n,m}$.}\]}
Part 1 of the induction start $n=0$ is the statement that $\revm{I_{B,0} =}I^B_{B,0}$ is the $0$-ideal for any torus $B$ (\Cref{cor:idealsfortori})\revm{; this also shows $(\star)$ for $n=0$}. For Part 2 we need to see that the restriction $L_{C_p^m\times B}\to L_B$ maps $I_{C_p^m\times B,0}$ surjectively onto $I_{B,m}$, which we know is generated by $I_m$ by \cref{cor:idealsfortori}. For $i=0,\hdots,m-1$ we can consider the elements $\overline{\mv}_i\in I_{C_p^{i+1},0}/p_{i}^*I_{C_p^i,0}$, which reduce to $v_i \in L/I_{i-1}$. It follows that the inflation of $\overline{\mv}_i$ to $C_p^m$ via any choice of surjection $C_p^m\times B\to C_p^{i+1}$ gives an element of a quotient of $I_{C_p^{m}\times B,0}$ which reduces to $v_i$ in $L_B/I_{i-1}$. Since $I_i/I_{i-1}$ is generated by $v_i$, the claim follows.

We now assume that Theorem \ref{thm:generators1} holds for an elementary abelian $p$-group $A$ of rank $n$ and show it also holds for $A\times C_p$. For any $m\in \N$ we consider the surjection
\[ L_{A\times \T\times B}/p^*_AI_{A,m}\to L_{A\times C_p\times B}/p^*_AI_{A,m},\]
with kernel generated by $e_{\tau^p}$ for $\tau$ the tautological $\T$-character pulled back to $A\times \T\times B$. We first claim that if $V\in A^*$ is non-trivial, then the Euler class $p_A^*(e_V)$ is a non-zero divisor in $L_{A\times C_p\times B}/p^*_AI_{A,m}$. To see this, we use that since $\T\times B$ is a torus we can apply the induction hypothesis to $L_{A\times \T\times B}$. In particular, 
\lennart{we know that the statement $(\star)$ holds for $A\cong C_p^n$ and hence} $p^*_AI_{A,m}$ generates the ideal $I_{A\times \T\times B,m}$. Therefore, $L_{A\times \T\times B}/p_A^*I_{A,m}$ is an integral domain. So we have to show that $p^*_A(e_V)$ still acts regularly modulo $e_{\tau^p}$. Since both Euler classes are regular, this is equivalent to showing that $e_{\tau^p}$ is regular modulo $p^*_A(e_V)$. We have an isomorphism
\[ L_{A\times \T\times B}/(p_A^*I_{A,m},p^*_A(e_V))\cong L_{\ker(V)\times \T\times B}/p^*_{\ker(V)}(\res^A_{\ker(V)}I_{A,m}).\]
By Part 2 of the induction hypothesis, we know that $I_{A,m}$ restricts onto $I_{\ker(V),m+1}$, hence the latter quotient identifies with $L_{\ker(V)\times \T\times B}/p^*_{\ker(V)}I_{\ker(V),m+1}$. Again we know by the induction hypothesis that this quotient is an integral domain, and $e_{\tau^p}$ is clearly a non-trivial element. So the claim follows and we have shown that $p^*_A I_{A,m}$ generates the Euler power torsion in $L_{A\times C_p\times B}$ (at height $m$) for characters inflated up from $A$. 

Hence to understand the full ideal $I_{A\times C_p\times B,m}$ it suffices to further divide by the Euler-power torsion for the remaining torsion characters in $(A\times C_p\times B)^*$ (there is no Euler-power torsion for non-torsion characters by \cref{prop:regularglobal}). These torsion characters are of the form $V\otimes \overline{\tau}^k$, where $V\in A^*$, $\overline{\tau}$ is the restriction of $\tau\in \T^*$ to $C_p$ and $k\in \{1,\hdots,p-1\}$. Furthermore we can assume that $k=1$: Any $V\otimes \overline{\tau}^k$ has some power of the form $V'\otimes \overline{\tau}$ and hence $e_{V'\otimes \overline{\tau}}$ is a multiple of $e_{V\otimes \overline{\tau}^k}$. Thus, $I_{A\times C_p\times B,m}/p^*_A I_{A,m}$ is generated by Euler-power torsion for characters of the form $V\otimes \overline{\tau}$.

Again it is beneficial to pass to the integral domain $L_{A\times \T\times B}/p_A^*I_{A,m}$ to understand the Euler-power torsion for these characters. We have the following:
\begin{lemma} \label{lem:inductionstep} Let $A=C_p^n$ be an elementary abelian $p$-group, $B$ a torus and $m\in \N$. Further let $x\in L_{A\times \T\times B}/p_A^*I_{A,m}$ be an element satisfying
    \[x\cdot \prod_{V\in A^*} e_{V\otimes \tau}^{n_V} = y \cdot e_{\tau^p}\]
for some $y$ and collection of natural numbers $n_V$. Then $x$ lies in the ideal generated by \[ p_{A\times \T}^*\res^{A\times C_p^m\times \T}_{A\times \T}(\mv_{n+m}). \]
\end{lemma}
\begin{proof} By applying $p_{A\times C_p^m\times \T}^*$ to the defining property of $\mv_{n+m}$ (Definition \ref{def:vn}) we obtain the equation
\[ e_{\tau^p}= p_{A\times C_p^m\times \T}^* (\mv_{n+m}) \cdot \prod_{V\in (A\times C_p^m)^*} e_{V\otimes \tau}\]
in $L_{A\times C_p^m\times \T \times B}/p_{A\times C_p^m}^*I_{A\times C_p^m,0}$. 
Restricting from $A\times C_p^m$ to $A = C_{p}^n$ yields
\begin{equation} e_{\tau^p}= p_{A\times \T}^*\res^{A\times C_p^ m\times \T}_{A\times \T}(\mv_{n+m}) \cdot \prod_{V\in A^*} e_{V\otimes \tau}^{p^m} \end{equation}
in the quotient $L_{A\times \T \times B,m}/p_A^*I_{A,m}$. This uses that every character of $A$ extends to $p^m$ different characters of $A\times C_p^m$ and that the restriction of $I_{A\times C_p^m,0}\subseteq L_{A\times C_p^m}$ lands in the ideal $I_{A,m}\subseteq L_A$. For the rest of the proof we write $z$ for the element $p^*_{A\times \T}\res^{A\times C_p^ m\times \T}_{A\times \T}(\mv_{n+m})$. With $x$ as in the statement of the lemma, we hence obtain an equation of the form
\begin{equation} \label{eq:etaup}  x\cdot \prod_{V\in A^*} e_{V\otimes \tau}^{n_V}
=z\cdot \prod_{V\in A^*} e_{V\otimes \tau}^{p^m} \cdot y\end{equation}
and we need to show that $x$ is a multiple of $z$. The Euler classes $e_{V\otimes \tau}$ fit into short exact sequences of the form
\[ 0\to L_{A\times \T\times B}/p_A^* I_{A,m} \xr{e_{V\otimes \tau}} L_{A\times \T\times B}/p_A^* I_{A,m} \xr{((\id,V^{-1})\times B)^*} L_{A\times B}/p_A^*I_{A,m} \to 0,  \]
analogously to Equation \ref{eq:exact}. The induction hypothesis implies that the quotient $L_{A\times B}/p_A^*I_{A,m}$ is an integral domain. We know that $z$ restricts to $v_{n+m}\in L/I_{n+m}$ at the trivial group. In particular it must restrict non-trivially under each $((\id,V^{-1})\times B)^*$. Hence the above short exact sequence implies: If an Euler class $e_{V\otimes \tau}$ divides an element of the form $z\cdot \alpha$, then $e_{V\otimes \tau}$ divides $\alpha$. Applying this iteratively to Equation \ref{eq:etaup} (and using that $L_{A\times \T\times B}/p_A^*I_{A,m}$ is an integral domain as well by the induction hypothesis) we see that $\prod_{V\in A^*} e_{V\otimes \tau}^{n_V}$ must divide the term $\prod_{V\in A^*}e_{V\otimes \tau}^{p^m}\cdot y$. Dividing on both sides shows that $x$ is a multiple of $z$, as desired.
\end{proof}
\begin{cor} \label{cor:quotient} The quotient
\[ I_{A\times C_p\times B,m}/p^*_A I_{A,m}\]
is generated by the element
\[ \res^{A\times C_p^m\times C_p\times B}_{A\times C_p\times B}(p_{A\times C_p^m\times C_p}^*\overline{\mv}_{n+m})=p_{A\times C_p}^*(\res^{A\times C_p^m\times C_p}_{A\times C_p}\overline{\mv}_{n+m}).\]
\end{cor}
\begin{proof} We saw above that the quotient $I_{A\times C_p\times B,m}/p^*_A I_{A,m}$ is generated by Euler-power torsion for characters of the form $V\otimes \overline{\tau}$. An element $\overline{x}$ of $L_{A\times C_p\times B}/p_A^* I_{A,m}$ is such a torsion element if and only if it is the reduction of an element $x\in L_{A\times \T\times B}/p_A^* I_{A,m}$ satisfying the conditions of the lemma. Since the reduction of $p_{A\times \T}^*\res^{A\times C_p^m\times \T}_{A\times \T}(\mv_{n+m})$ equals
$p_{A\times C_p}^*\res^{A\times C_p^m\times C_p}_{A\times C_p}(\overline{\mv}_{n+m})$, it follows that $\overline{x}$ lies in the ideal generated by the latter.

As $p_{A\times C_p}^*\res^{A\times C_p^m\times C_p}_{A\times C_p}(\overline{\mv}_{n+m})$ is Euler-power torsion itself, it hence forms a generator of $I_{A\times C_p\times B,m}/p^*_A I_{A,m}$.
\end{proof}
To finish the proof of Theorem \ref{thm:generators1}: Setting $m=0$ in the corollary shows that $\overline{\mv}_n$ generates the quotient $I_{A\times C_p\times B,0}/p_A^*I_{A,0}$. By the induction hypothesis we know that $I_{A,0}$ is generated by $p_1^*(\overline{\mv}_0),\hdots,p_{n-1}^*(\overline{\mv}_{n-2}),\overline{\mv}_{n-1}$. Combined this proves Part 1 for the group $A\times C_p$. 

For Part 2 and general $m$, we first note that it suffices to show the statement for any choice of injection $i\colon A \times C_p \to C_p^{n+m+1}$ since any two only differ by postcomposition with an automorphism of $C_p^{n+m+1}$. We can hence pick the canonical inclusion $A\times C_p\to A\times C_p^m\times C_p$ avoiding $C_p^m$. By the induction hypothesis we know that $I_{A\times C_p^m,0}$ surjects onto $I_{A,m}$. From the diagram
\[
\xymatrix{
I_{A\times C_p^m, 0}\ar[d]^{p_{A\times C_p^m}^*} \ar@{->>}[r]^{\res} & I_{A,m} \ar[d]^{p_A^*}\\ 
I_{A\times C_p^m\times C_p\times B,0} \ar[r]_-{\res} & I_{A\times C_p\times B, m}
}
\]
we see that $p_A^*(I_{A,m})$ is contained in the image of the lower horizontal arrow. Furthermore, Corollary \ref{cor:quotient} implies that $I_{A\times C_p\times B,m}/p^*_A I_{A,m}$ is generated by the restriction of an element of $I_{A\times C_p^m\times C_p\times B,m}$. \revm{Finally, \lennart{we need to see that $(\star)$ holds true for $C_p^{n+1}\cong A\times C_p$. This} is a direct consequence of Corollary \ref{cor:quotient}, as by definition both $p_A^*I_{A,m}$ and $p_{A\times C_p}^*(\res^{A\times C_p^m\times C_p}_{A\times C_p}\overline{\mv}_{n+m})$ lie in the image of $p_{A\times C_p}^* I_{A\times C_p,m}$.}
This finishes the proof.

\begin{remark} \label{rem:generatorstori} Unlike the sequence $v_0,\hdots,v_{n-1}$, the sequence 
\[p_1^*(\overline{\mv}_0),p_2^*(\overline{\mv}_1),\hdots,p_{n-1}^*(\overline{\mv}_{n-2}),\overline{\mv}_{n-1}\]
isn't regular. In fact, the ideal $I_{C_p^n,0}$ generated by these elements is precisely that of Euler-\lennart{power} torsion.

This can be corrected by passing to a torus: The ideal $I^{\T^n}_{C_p^n,0}$ is generated by the sequence $p_1^*(\mv_0),p_2^*(\mv_1),\hdots,p_{n-1}^*(\mv_{n-2}),\mv_{n-1}$. Here, each $p_{i+1}^*(\mv_i)$ is the element of
\begin{equation}\label{eq:regulariso} L_{\T^n}/(p_1^*(\mv_0),p_2^*(\mv_1),\hdots,p_i^*(\mv_{i-1}))\cong L_{C_p^{i}\times \T^{n-i}}/I_{C_p^i\times \T^{n-i},0} \end{equation}
obtained as the inflation of $\mv_i\in L_{C_p^i\times \T}/I_{C_p^i\times \T,0}$ along the projection to the first coordinate of $\T^{n-i}$.  Since each successive quotient $L_{C_p^{i}\times \T^{n-i}}/I_{C_p^i\times \T^{n-i},0}$ is an integral domain, the regularity of the sequence is clear once we have demonstrated the isomorphism claimed in \eqref{eq:regulariso}. Similarly one shows that $I^{\T^n}_{C_p^n,m}$ is generated by a regular sequence of length $n+m$.

To establish the isomorphism, first note that $(e_{\tau^p}) \subseteq (p^*_{C_p^{i-1}\times \T}(\mv_{i-1}))$ in the quotient ring $L_{C_p^{i-1}\times \T^{n-i+1}}/p^*_{C_p^{i-1}}I_{C_p^{i-1},0}$ by \cref{lem:inductionstep}. Thus, \cref{cor:quotient} with $m=0$ gives
\[L_{\T^n}/(p_1^*(\mv_0),p_2^*(\mv_1),\hdots,p_i^*(\mv_{i-1}))\cong L_{C_p^{i}\times \T^{n-i}}/(p_1^*(\overline{\mv}_0),p_2^*(\overline{\mv}_1),\hdots,p_i^*(\overline{\mv}_{i-1})),\]
and the claimed isomorphism becomes \cref{thm:generators1}.
\end{remark}

\section{The Zariski topology on the spectrum of invariant prime ideals}

\begin{figure}
    \centering
    \begin{tikzpicture}
        \node[draw, circle, fill, inner sep=1pt] (1) at (0,0) {};
        \node[left] at (1.west) {$(e)=I_{\{1\},0}$};
        
        \node[draw, circle, fill, inner sep=1pt] (2) at (0,1.5) {};
        \node[left] at (2.west) {$(e, v_0)=I_{\{1\},1}$};
        
        \node[draw, circle, fill, inner sep=1pt] (3) at (0,3) {};
        \node[left] at (3.west) {$(e, v_0, v_1)=I_{\{1\},2}$};
        
        \node[draw, circle, fill, inner sep=1pt] (4) at (0,4.5) {};
        \node[left] at (4.west) {$(e, v_0, v_1, v_2)=I_{\{1\},3}$};
        
        \node[draw, circle, fill, inner sep=1pt] (5) at (0,8.5) {};
        \node[left] at (5.west) {$(e, v_0, v_1, v_2,\hdots)=I_{\{1\},\infty}$};
        \draw[dotted] (0, 5) -- (0,8.5);

        \node[draw, circle, fill, inner sep=1pt] (C1) at (2,0) {};
        \node[right] at (C1.east) {$I_{C_p,0}=(\overline{\mv}_0)$};
        
        \node[draw, circle, fill, inner sep=1pt] (C2) at (2,1.5) {};
        \node[right] at (C2.east) {$I_{C_p,1}=(\overline{\mv}_0,i_1^*\overline{\mv}_1)=(v_0, i_2^*\overline{\mv}_1)$};
        
        \node[draw, circle, fill, inner sep=1pt] (C3) at (2,3) {};
        \node[right] at (C3.east) {$I_{C_p,2}=(\overline{\mv}_0,i_1^*\overline{\mv}_1,i_1^*\overline{\mv}_2)=(v_0,i_2^*\overline{\mv}_1,i_2^*\overline{\mv}_2)=(v_0, v_1, i_3^*\overline{\mv}_2)$};
        
        \node[draw, circle, fill, inner sep=1pt] (C4) at (2,4.5) {};
        \node[right] at (C4.east) {$I_{C_p,3}=(\overline{\mv}_0,i_1^*\overline{\mv}_1,i_1^*\overline{\mv}_2,i_1^*\overline{\mv}_3)=\hdots=(v_0, v_1, v_2, i_4^*\overline{\mv}_3)$};
        
        \node[draw, circle, fill, inner sep=1pt] (C5) at (2,8.5) {};
        \node[right] at (C5.east) {$I_{C_p,\infty}=(\overline{\mv}_0,i_1^*\overline{\mv}_1,i_1^*\overline{\mv}_2,\hdots)$};

    \draw[dotted, ] (2, 5) -- (2,8.5);
        \draw (2) -- (C1);
        \draw (3) -- (C2);
        \draw (4) -- (C3);
        \draw (C4) -- (1,5.25);
        \draw[dotted] (1, 5.25) -- (0.2, 5.85);
        \draw (5) -- (C5);
        \draw (1) -- (2);
        \draw (2) -- (3);
        \draw (3) -- (4);
        \draw (4) -- (0,5);
        \draw (C1) -- (C2);
        \draw (C2) -- (C3);
        \draw (C3) -- (C4);
        \draw (C4) -- (2,5);
        \fill[yellow,opacity=0.5] (C2.center) -- (C5.center) -- (5.center) -- (3.center) -- cycle;
    \end{tikzpicture}
    \caption{A picture of $\Spec^{\inv}(L_{C_p})$, localized at $p$, including different choices of generators, with $i_j\colon C_p\to C_p^k$ denoting the $j$th canonical inclusion (generators arising from the further inclusions $C_p\to C_p^k$ are omitted). The yellow area depicts the closure of the point $I_{C_p,1}$.}
\end{figure}
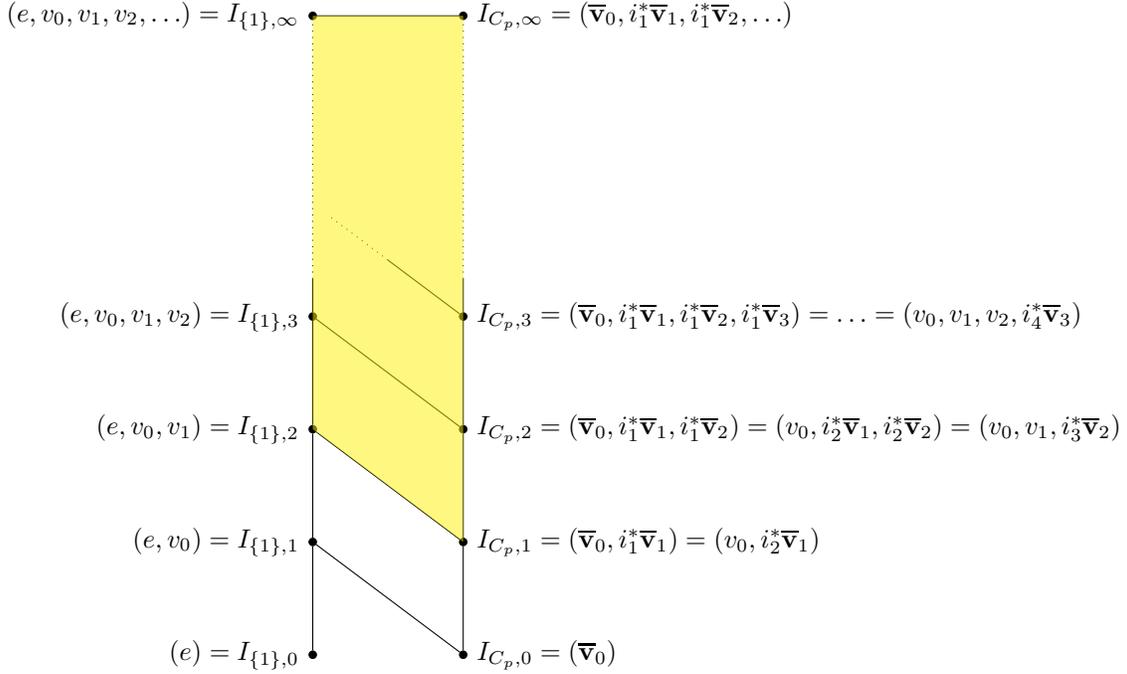

The goal of this section is to describe the Zariski topology on  $\Spec^{\inv}(L_A)$, or equivalently the topology on the space $|\MM_{FG}^A|$ (\Cref{thm:ClassificationOfInvariantPrimes}). By definition, the closed subsets of $\Spec^{\inv}(L_A)$ are the subsets of the form
\[ V(X)=\{ I^{\lennart{A}}_{B,n}  \in \Spec^{\inv}(L_A)\ |\ X\subseteq I^{\lennart{A}}_{B,n} \} \]
for some subset $X$ of $L_A$. Hence we need to determine the collections of invariant prime ideals that arise as $V(X)$ for some $X$. We now fix a subset $X$. Since a containment $X\subseteq I^A_{B,n}$ automatically implies $X\subseteq I^A_{B,n'}$ for all $n'\geq n$, it suffices to understand -- for every closed subgroup $B$ of $A$ -- the smallest value of $n\in \N\cup \{\infty\}$ such that $X\subseteq I^A_{B,n}$. In other words we need to determine the function
\[ \height_X\colon \Sub(A)\to \overline{\N}_- = \{-1\}\cup \mathbb{N}\cup \{\infty\} \]
defined by
\[ \height_X(B)=\sup \{ n\ |\ X\not \subseteq I_{B,n}^{A}\}, \]
where we set $\sup(\varnothing) = -1$. We note that with this definition the height function of the image of $v_n\in L$ in $L_A$ (thought of as a one-element set) is constantly $n$. This follows from the fact that $\Phi^BL$ is a free $L$-module by \cref{prop:GeometricFixedPoints}; cf.\ also the proof of \cref{thm:ClassificationOfInvariantPrimes}. Moreover, we have:
\begin{example} Let $x_n\in L_{C_p^{n+1}}$ be a lift of $\overline{\mv}_n\in L_{C_p^{n+1}}/p_{C_p^n}^*I_{C_p^n,0}$. Then Corollary \ref{cor:resxn} implies that $\height_{x_n}(B)=n-\rk(B)$.
\end{example}

 Our goal is to understand which functions $\Sub(A)\to \overline{\N}_-$ arise as such height functions. In the previous section we showed that there are inclusions between invariant prime ideals associated to different subgroups of $A$. These translate to conditions between the different values of $\height_X$: If $B'\subseteq B$ is an inclusion of subgroups of $A$ such that $B/B'$ is $p$-toral, and $X$ is contained in $I^A_{B,n}$, then $X$ is also contained in $I^A_{B',n+\rank_p(\pi_0(B/B'))}$. In terms of the height function this translates to the inequality
\[ \height_X(B')\leq \height_X(B)+\rank_p(\pi_0 (B/B')). \]
This leads us to the following definition:
\begin{definition} \label{def:admissible} A function $f\colon \Sub(A)\to \overline \N_-$ is called {\em admissible} 
if it satisfies the inequality
\[ f(B')\leq f(B) + \rank_p(\pi_0 (B/B')) \]
for every $p$-toral inclusion $B'\subseteq B$ of closed subgroups of $A$. Here, $B'\subseteq B$ is \emph{$p$-toral} if $B/B'$ is a product of a torus and a $p$-group.
\end{definition}
By the above considerations, any height function $\height_X$ is admissible. When the group $A$ is finite, it turns out that the converse also holds: Any admissible function is realized by a height function $\height_X$. For positive dimensional $A$ there is an additional condition on top of admissibility. To state this condition, we recall that choosing an invariant Riemannian metric $d$ on $A$ also equips the set of closed subgroups $\Sub(A)$ with the Hausdorff metric, the underlying topology of which does not depend on the chosen metric on $A$. This turns $\Sub(A)$ into a compact totally-disconnected metric space, in which a sequence $(B_i)_{i\in \mathbb{N}}$ of closed subgroups converges to another closed subgroup $B\in \Sub(A)$ if and only if almost all $B_i$ are subgroups of $B$ and for every element $b\in B$ the distance function $d(b,B_i)$ converges to zero (see \cite[Section 5.6]{tomDieckTransformationRepresentation}). 
If $B_i \to B$, we have the following two implications about representations: 
\begin{enumerate}
    \item Let $W$ be a (finite dimensional) representation of $B$ with $W^B = 0$. Thus, writing $W$ as a sum of characters $V_i$, none of the $V_i$ is trivial. For sufficiently large $i$, no $\ker(V_i)$ contains $B_i$ (since $\ker(V_i) \subseteq B$ is a closed proper subgroup) and thus $W^{B_i} = 0$.
    \item Let $V$ and $W$ be two characters of $B$ such that $\res^B_{B_i}V = \res^B_{B_i}W$ for all sufficiently large $i$. If $V\neq W$, then $(V\cdot W^{-1})^B = 0$, in contradiction with the previous point. Thus $V = W$. 
\end{enumerate}
We have the following:
\begin{prop} 
For every {\em finite} subset $X\subseteq L_A$, the height function $\height_X$ is a locally constant function on $\Sub(A)$.
\end{prop}
\begin{proof}
We start by noting that
\[ \height_X(B)=\max(\height_{\{x\}}(B)\ |\ x\in X). \]
The maximum of a finite number of locally constant functions is again locally constant. Hence it suffices to understand that $\height_{\{x\}}$ is locally constant for any element $x$ of $L_A$.

Now let $(B_i)_{i\in \mathbb{N}}$ denote a sequence of subgroups of $A$ converging to a subgroup $B$. We need to show that $\height_{\{x\}}(B_i)=\height_{\{x\}}(B)$ for almost all $i$. Without loss of generality we can assume that the $B_i$ are subgroups of $B$. Replacing $x$ by $\res^A_B(x)$ if necessary we can further \revm{restrict to the case} $A=B$.

We first assume that $x\in I^A_{A,n}$ for some $n$, and show that then also $x\in I^A_{B_i,n}$ for almost all $B_i$. If $x\in I^A_{A,n}$, there exists an $A$-representation $W$ with $W^A=0$, such that $e_W\cdot x$ lies in the ideal $L_A\cdot I_n$ \revm{(see Lemma \ref{lem:characterizationofideals})}. For all $i$ large enough we have $W^{B_i}=0$, meaning that for these $i$ the restriction $e_{\res^A_{B_i} W}$ becomes invertible in $\Phi^{B_i} L$. It follows that $\res^A_{B_i} x$ maps to the ideal generated by $I_n$ in $\Phi^{B_i} L$, in other words $x$ is contained in $I^A_{B_i,n}$.

For the other direction, we assume that $x$ is not contained in $I^A_{A,n}$ for some $n$, and show that then also $x\notin I^A_{B_i,n}$ for almost all $B_i$. For this we recall from Remark \ref{rem:naturalitybi} the construction of elements $\gamma_j^V$ for every character $V\in G^*$ and $j\in \N$ satisfying the following three properties:
\begin{enumerate}
    \item $\alpha^*(\gamma_j^V)=\gamma_j^{\alpha^*(V)}$ for every group homomorphism $\alpha\colon G'\to G$.
    \item $\gamma_0^V=e_V$.
    \item \revm{The canonical map $L[e_V^{\pm 1}, \gamma_j^V\ |\ V\in G^*-\{\epsilon\},j>0]\to \Phi^G L$ is an isomorphism for all abelian compact Lie groups $G$.}
\end{enumerate}
Now, if $x$ is not contained in $I^A_{A,n}$, it maps to a non-trivial element in 
\[ \Phi^ AL/I_n= L/I_n[e_V^{\pm 1},\gamma_j^V\ |\ V\in A^*-\{\epsilon\},j> 0]. \]
In other words, there exists an $A$-representation $W$ with $W^A=0$ and pairwise different non-trivial characters $V_1,\hdots,V_k$ such that $e_W\cdot x\in L_A$ is a polynomial over $L$ in the classes $\gamma_j^{V_l}$, $l=1,\hdots,k, j\geq 0$, not all of whose coefficients are contained in $I_n$. For all $i$ large enough we have that (i) $W^{B_i}=0$ and that (ii) all the characters $V_1,\hdots,V_k$ restrict to pairwise different and non-trivial characters of $B_i$. It then follows that for these $i$ the element $e_{\res^A_{B_i} W}\cdot \res^A_{B_i}x$ equals the corresponding polynomial in the classes $\gamma_j^{\res^A_{B_i}V_l}$, implying that it maps to a non-trivial element in $\Phi^{B_i}L/I_n$. In other words, $x$ is not contained in $I^A_{B_i,n}$ for $i$ large enough. This finishes the proof.
\end{proof}
Together with admissibility, this property characterizes the height functions of finite subsets of~$L_A$:
\begin{proposition} \label{prop:heightfunctions} Given a function $f\colon \Sub(A)\to \overline{\mathbb{N}}_-$, the following are equivalent:
\begin{enumerate}
    \item There exists a finite subset $X\subseteq L_A$ such that $f=\height_X$. 
    \item The function $f$ is admissible and locally constant.
\end{enumerate}
\end{proposition}
\begin{proof} We have already shown the implication 1. $\Rightarrow$ 2.

It remains to show that given a locally constant admissible function $f$, there exists a finite subset $X\subseteq L_A$ with $f=\height_X$. We start with the following claim: If $f$ is admissible, then given any pair of subgroups $B,B'\subseteq A$, there exists an element $x_{B,B'}\in L_A$ such that $\height_{x_{B,B'}}(B)=f(B)$ and $\height_{x_{B,B'}}(B')\leq f(B')$. To see this, we distinguish between three cases: 
\begin{enumerate}
    \item[(i)] If $B$ is not a subgroup of $B'$, we can choose a character $V\in A^*$ which restricts to the trivial character over $B'$ but to a non-trivial character over $B$. Then $x_{B,B'}=e_V\cdot v_{f(B)}$ has the desired properties, since $\height_{e_V\cdot v_{n}}(B)=n$ and $\height_{e_V\cdot v_{n}}(B')=-1$. Here and in the following, we set $v_{-1} = 0$, $v_0 = p$ and $v_{\infty}= 1$. 
    \item[(ii)] If $B$ is a subgroup of $B'$ with $\pi_0(B'/B)$ not a $p$-group, we choose a prime $q\neq p$ dividing the order of $\pi_0(B'/B)$ and a surjection $g\colon B'\to C_q$ containing $B$ in the kernel. Then we set $y=v_{f(B)}\cdot g^*( \lennart{\overline{\mv}_0^{(q)}})\lennart{\in L_{B'}}$, where $ \lennart{\overline{\mv}_0^{(q)}}\in L_{C_q}$ is the element introduced in Section \ref{sec:reduction}. Then $ \lennart{\overline{\mv}_0^{(q)}}$ is an element of $I^{C_q}_{C_q,0}$ and its restriction to the trivial group is given by $q$ and hence a unit. It follows that $\height_y(B')=-1$, since $\height_{g^*( \lennart{\overline{\mv}_0^{(q)}})}(B')=-1$. Moreover, $\height_y(B)=f(B)$, since $\res^{B'}_B(y)=v_{f(B)}\cdot q$. Hence, we can set $x_{B,B'}$ to be any lift of $y$ to an element of $L_A$. 
    \item[(iii)] The remaining case is when $B$ is a subgroup of $B'$ with $\pi_0(B'/B)$ a $p$-group. Let $r$ be the minimum of the $p$-rank of $\pi_0(B'/B)$ and the number $f(B)+1$, and choose a surjection $g\colon B'/B\to C_p^{r}$. By \cref{cor:resxn} we know for $f(B)<\infty$ that there exist an element $x_{f(B)}\in L_{C_p^{f(B)+1}}$ such that $\height_{x_{f(B)}}(1)=f(B)$ and $\height_{x_{f(B)}}(C_p^{f(B)+1})=-1$. (We set $x_{-1} = 0$.) We can choose an embedding $C_p^r\to C_p^{f(B)+1}$ and restrict $x_{f(B)}$ to an element $y\in L_{C_p^r}$. Then we have  $\height_y(1)=f(B)$ and $\height_y(C_p^r)=f(B)-r$ (see \cref{cor:resxn}). If $f(B) = \infty$, choose $y=1$. It follows that $\height_{g^*(y)}(B)=f(B)$ and $\height_{g^*(y)}(B')\leq f(B)-r\leq f(B')$, since $f$ is admissible. Hence, $g^*(y)$ 
has the desired properties. 
\end{enumerate}
Now given any such pair $(B,B')$ there exists an open neighbourhood $U_{B'}$ of $B'$ on which both $\height_{x_{B,B'}}$ and $f$ are constant. The $U_{B'}$ for varying $B'$ form an open cover of the compact space $\Sub(A)$. Let $U_{B_1'},\dots,U_{B_k'}$ be a finite subcover. We then set
\[ x_B=x_{B,B_1'}\cdot x_{B,B_2'}\cdots x_{B,B_k'} \]
to be the product of the corresponding elements. For any closed subgroup $B'$ we have
\[ \height_{x_B}(B')=\min(\height_{x_{B,B_j'}}(B')\ |\ j=1,\dots,k). \]
For $B'=B$ this gives $\height_x(B)=f(B)$, since $\height_{x_{B,B_j'}}(B')=f(B)$ for all $j$. Any $B'$ is contained in $U_{B_i'}$ for some $i$, yielding
\[ \height_{x_B}(B')\leq \height_{x_{B,B_i'}}(B')=\height_{x_{B,B_i'}}(B_i')\leq f(B_i')=f(B'). \]
In summary, the height function $\height_{x_B}$ is less than or equal to $f$ everywhere and agrees with $f$ at $B$ itself. Once more we can apply that $\height_{x_B}$ and $f$ are locally constant to find that $\height_{x_B}$ and $f$ in fact agree on a neighborhood $V_B$ of $B$. Letting $B$ vary, this yields an open cover of $\Sub(A)$, for which we can choose a finite subcover $V_{B_1},\dots,V_{B_l}$. Hence, for every $B\in \Sub(A)$, there exists some $i\in \{1,\dots,l\}$ such that $f(B)=\height_{x_{B_i}}$. Finally, we define $X$ to be the set $\{ x_{B_1},\dots,x_{B_l}\}$. Since $\height_X$ is given by the maximum of the functions $\height_{B_i}$,  it follows that $X$ has the desired property.
\end{proof}
\revm{For the following theorem, we remind the reader of our convention that everything is implicitly localized at a fixed prime $p$.}
\begin{thm} \label{thm:zariski} The Zariski topology on $\Spec^{\inv}(L_A)$ has as a basis the closed sets
\[ V_f=\{ I_{B,n} \ |\ n> f(B) \} \]
for all locally constant, admissible functions $f\colon \Sub(A)\to \overline{\mathbb{N}}_-$. 
\end{thm}
\begin{proof} A basis for the Zariski topology is given by the sets $V(X)$ for all finite subsets $X$ of $L_A$. As we saw above, $V(X)$ is determined by its height function $\height_X$ as
\[ V(X)=\{ I_{B,n} \ |\ n> \height_X(B)\}. \]
By Proposition \ref{prop:heightfunctions}, the functions $\Sub(A)\to \overline{\N}_-$ that occur as height functions of finite subsets are precisely the locally constant admissible functions, which finishes the proof.
\end{proof}

\section{Comparison with $A$-spectra} \label{sec:comparison}

In this final section, we discuss the relationship of the algebraic results of the previous sections with the theory of $A$-spectra. \revm{Throughout, we will fix a prime $p$ and implicitly localize at it.} 

\subsection{The universal support theory via $MU_A$-homology}\label{sec:universalsupport}
We begin by comparing our classification of invariant prime ideals with the Balmer spectrum of compact $p$-local $A$-spectra. We recall from \cite{BalmerSpectrum} that a prime ideal $\mathfrak{p}$ of a tensor-triangulated category $\mathcal{C}$ is defined to be a thick tensor-ideal with the additional property that if $X\otimes Y$ is contained in $\mathfrak{p}$, then $X\in \mathfrak{p}$ or $Y\in \mathfrak{p}$. The set of all prime ideals assembles to a topological space $\Spec(\mathcal{C})$, the Balmer spectrum, with the topology generated by the closed sets $\supp(X)=\{ \mathfrak{p}\ |\ X\notin \mathfrak{p}\}$ for all objects $X\in \mathcal{C}$.

Here, the support function $\supp(-)$, assigning a closed set of the Balmer spectrum to every object of $\mathcal{C}$, is the \emph{universal support theory} in the sense of Balmer. This means that it is terminal among pairs $(T,\sigma)$ of a topological space $T$ and a function 
\[ \sigma\colon \operatorname{ob}(\mathcal{C})\to \{\text{closed subsets of } T\} \]
satisfying $\sigma(0)=\emptyset$, $\sigma(1)=T$, $\sigma(X\oplus Y)=\sigma(X)\cup \sigma(Y)$, $\sigma(\Sigma X)=\sigma(X)$ and $\sigma(X\otimes Y)=\sigma(X)\cap \sigma(Y)$ for all $X,Y\in \operatorname{ob}(\CC)$, and $\sigma(X)\subseteq \sigma(Y)\cup \sigma(Z)$ whenever there exists a distinguished triangle $X\to Y\to Z\to \Sigma X$. See \cite{BalmerSpectrum} for more details.

In the case of compact $p$-local $A$-spectra $\Sp_A^c$ for an abelian compact Lie group $A$, the Balmer spectrum was computed in \cite{BalmerSanders, BarthelHausmannNaumannNikolausNoelStapleton, BarthelGreenleesHausmann}. Given a closed subgroup $B$ and $n\in \mathbb{N}\cup \{\infty\}$, one defines a prime ideal
\[ \lennart{P^A_{B,n}} = \{ X\in \Sp_A^c\ |\ K(n)_*(\Phi^B X)=0\},\]
where $K(n)$ denotes the nth Morava $K$-theory, and $\Phi^B X$ is the $B$-geometric fixed point spectrum of $X$, a compact $p$-local spectrum.
\begin{theorem}[\cite{BalmerSanders, BarthelHausmannNaumannNikolausNoelStapleton, BarthelGreenleesHausmann}] The map
\begin{align*} \Sub(A)\times \overline{\mathbb{N}}  & \to  \Spec(\Sp_A^c) \\
                (B,n) & \mapsto  \lennart{P^A_{B,n}}
\end{align*}
defines a bijection. Moreover, the topology on $\Spec(\Sp_A^c)$ has a basis given by the closed sets defined by
\[ \{ \lennart{P^A_{B,n}}\ |\ n\geq f(B) \} \]
where $f$ ranges through all admissible functions $\Sub(A)\to \overline{\N} = \mathbb{N}\cup \{\infty\}$.
\end{theorem}
Here, `admissible' is meant in the sense of Definition \ref{def:admissible}.
Hence, comparing to Theorem \ref{thm:zariski}, we see that the assignment $I^{\lennart{A}}_{B,n}\mapsto \lennart{P^A_{B,n}}$ defines a homeomorphism from $\Spec^{\inv}(L_A)$ to $\Spec(\Sp_A^c)$.\footnote{The reader may have noticed that in this section our admissible functions $f$ take values in $\overline{\N}$, while in the last section they took values in $\overline{\N}_-$. Likewise, we consider the condition $n\geq f(B)$ here, and considered $n> f(B)$ before. A shift by one shows that the topologies agree. The shift is caused by wanting to have $v_n$ having height $n$ in the last section.} The computation of $\Spec(\Sp_A^c)$ is also analogous to the one of $\Spec^{\inv}(L_A)$ in the way that computing the underlying set is relatively straightforward (and is in fact known for all compact Lie groups), with most work going into understanding the topology.

Hence, we can view the universal support theory of $\Sp_A^c$ to take values in the invariant prime ideals of $L_A$. The goal of the remainder of this section is to construct this universal support theory more intrinsically using $MU_A$-homology and the structure of equivariant formal group laws described in this paper. The idea is the following: Given a compact $p$-local $A$-spectrum $X$, we can consider its equivariant complex bordism homology $(\underline{MU}_A)_* X$. Here, underlining $MU_A$ indicates that we take the $A$-Mackey functor valued homology of $X$, i.e., we record the collection of $(MU_B)_*(\res^A_B X)$ for all closed subgroups $B$ of $A$, together with restriction and transfer maps between them. We \revm{reiterate that we} will always work at the fixed prime $p$ and $p$-localize everything implicitly.

Since the coefficients $\pi_*^A MU_A$ are isomorphic to the Lazard ring $L_A$ and moreover the cooperations $\pi_*^A MU_A\wedge MU_A$ agree with $S_A$ (see \cite[Theorem 5.52]{Hau} and the discussion thereafter), the groups $(MU_A)_* X$ form a graded $A$-Mackey functor in $(L_A,S_A)$-comodules. As such, we can take its support in the invariant prime ideals
\[ \supp((\underline{MU}_A)_* X)=\{ I^A_{B,n} \ |\ ((\underline{MU}_A)_* X)_{I^A_{B,n}} \revm{\neq} 0 \}\subseteq \Spec^{\inv}(L_A). \]
\begin{remark}\label{rem:counterexample}
    In general, $\supp((\underline{MU}_A)_* X)$ is different from $\supp((MU_A)_* X)$. Take for example $A = C_2$, $p=2$ and $X=S^{\sigma}$, the circle with action given by reflection at a line. We have 
    \[MU_*(\res_1^{C_2}S^{\sigma})_{I_{1,0}} = MU_*(S^1)_{\Q} \neq 0.\] 
    The module $(MU_{C_2})_*(S^{\sigma})_{I_{1,0}^{C_2}}$ is rational as well since $\lennart{2}\notin I_{1,0}^{C_2}$. Rationally, $(MU_{C_2})_*(S^{\sigma})$ splits into the coinvariants $(MU_{C_2})_*(S^{\sigma})_{C_2} = 0$ and the geometric fixed points $(MU_{C_2})_*(S^{\sigma})[e^{-1}]$, for $e$ the Euler class of the unique non-trivial character. The element $\overline{\mv}_0$ from \ref{def:vn} becomes zero in the geometric fixed points, but restricts to $2$ and is thus not in $I_{1,0}^{C_2}$; thus $(MU_{C_2})_*(S^{\sigma})[e^{-1}]$ becomes zero as well after localization at $I_{1,0}^{C_2}$, and $(MU_{C_2})_*(S^{\sigma})_{I_{1,0}^{C_2}} = 0$.
\end{remark}
\begin{remark}[Transfer maps]
    The isomorphisms $L_A \cong \pi^A_*MU_A$ imply that on top of the contravariant restriction maps along group homomorphisms there also exist transfer maps $L_B \to L_A$ for inclusions $B\subseteq A$ of finite index.  In other words, the collection of all equivariant Lazard rings $L_A$ forms a `global Green functor' on the family of abelian compact Lie groups, in the sense of \cite[Definition 5.1.3]{SchGlobal}.
    In \cite[Section 23]{StricklandMulti}, Strickland defines transfers for equivariant formal groups and shows that they agree with the topological transfers for complex oriented equivariant cohomology theories, in particular $MU_A$.
    
   Using the methods from this paper, transfers on equivariant Lazard rings can be determined as follows: By Frobenius reciprocity, it suffices to compute $\tr_B^A\colon L_B \to L_A$ on $1 \in L_B$ since the restriction is surjective. Inductively, we can further assume that $A/B \cong C_p$. Furthermore, $\tr_B^A(1) = q^*\tr_{\{1\}}^{C_p}(1)$ for $q\colon A \to A/B\cong C_p$ since transfers are compatible with inflation maps (see, e.g., \cite[Theorem 4.2.6 ff.]{SchGlobal}). Hence it suffices to identify $\tr_{\{1\}}^{C_p}(1)$. We claim that it equals $\overline{\mv}_0 \in L_{C_p}$. Indeed, any transfer maps to zero in the geometric fixed points and is thus an element of $I_{C_p,0}$. We know that $\overline{\mv}_0$ generates $I_{C_p,0}$. Hence, we can write $\tr_{\{1\}}^{C_p}(1) = x \cdot \overline{\mv}_0$ for $x\in L_{C_p}$. Writing $x$ as $x_0 + x' \cdot e$ for $e$ a non-trivial Euler class and $x_0 \in L$, we obtain $\tr_{\{1\}}^{C_p}(1) = x_0\cdot \overline{\mv}_0$ since $e \cdot \overline{\mv}_0 = 0$ by the definition of $\overline{\mv}_0$. We obtain $x_0 \cdot p = \res_{\{1\}}^{C_p}\tr_{\{1\}}^{C_p}(1) = p$ since $\overline{\mv}_0$ restricts to $p$, and thus $x_0 = 1$ and $\tr_{\{1\}}^{C_p}(1) = \overline{\mv}_0$.
\end{remark}

We will now see that our notion of support is another model for the universal support theory on compact $A$-spectra.

We first note a major inconvenience: It is unclear whether $(MU_A)_* X$ is a finitely generated $L_A$-module, even for compact $X$. This is in contrast with the non-equivariant situation, where finite generation of $MU_*X$ for compact $X$ follows from the fact that $L$ is a polynomial ring and hence coherent. An analogous statement is unknown for equivariant Lazard rings. In particular, it is a priori unclear that $\supp((\underline{MU}_A)_* X)$ is indeed a \emph{closed} subset of $\Spec^{\inv}(L_A)$ and we have to prove this by hand, see Proposition \ref{prop:supportclosed} below.

The following proposition gives the relationship between the $MU_A$-homology support theory described above and geometric fixed points.

\begin{prop} \label{prop:supportgeom} Let $B$ be a subgroup of $A$, $n\in \overline{\N}$ and $X$ a compact $A$-spectrum. Then the following are equivalent:
\begin{enumerate}
    \item[(i)] The Mackey functor $((\underline{MU}_A)_*(X))_{I^A_{B,n}}$ is trivial.
    \item[(ii)] The $(MU_B)_*$-module $((MU_B)_*(\res^A_B X))_{I^B_{B,n}}$ is trivial.
  \item[(iii)] The $B$-geometric fixed points $\Phi^B(X)$ are of chromatic type $\geq n+1$.
\end{enumerate}
\end{prop}
We give two proofs of this proposition below, one using the results of \cite{BalmerSanders, BarthelHausmannNaumannNikolausNoelStapleton, BarthelGreenleesHausmann} and one using the methods from this paper. The latter one is complicated by the fact that we don't know whether $(MU_A)_*X$ is always finitely generated. 

We have the following corollary:
\begin{prop}\label{prop:supportclosed} Let $X$ be a compact $A$-spectrum. Then its support $\supp((\underline{MU}_A)_* X)$ is a closed subset of $\Spec^{\inv}(L_A)$.

Moreover, the assignment
\[ \operatorname{ob}(\Sp_A^c)\to \Spec^{\inv}(L_A)\ ;\ X\mapsto \supp((\underline{MU}_A)_* X)\]
is a support theory on compact $A$-spectra.
\end{prop}
\begin{proof} We consider the type function
\[ \type_X\colon \Sub(A)\to \overline{\mathbb{N}}\ ;\  X \mapsto \type(\Phi^BX), \]
which is locally constant by \cite[Proposition 4.3]{BarthelGreenleesHausmann}. By the previous proposition, $I_{B,n}$ is an element of $\supp((\underline{MU}_A)_* X)$ if and only if $\type_X(B) < n+1$. Since the support $\supp((\underline{MU}_A)_* X)$ is closed under inclusion, \cref{thm:inclusions} implies that $I_{B',n} \in \supp((\underline{MU}_A)_* X)$ if $I_{B, n-\rk_p(\pi_0(B/B'))}\in \supp((\underline{MU}_A)_* X)$ for every $p$-toral inclusion $B'\subseteq B$. Thus $\type_X$ is admissible in the sense of Definition \ref{def:admissible}. Theorem \ref{thm:zariski} implies that $\supp((\underline{MU}_A)_* X)$ is closed, as desired.

For the second part, all required properties of a support theory follow easily from exactness of localization $(-)_{I^ A_{B,n}}$ except for the one on the interplay with smash products. This in turn follows from the third characterization in Proposition \ref{prop:supportgeom}, since the type of a smash product of two compact spectra is the maximum of the two types.
\end{proof}

By the universal property of the Balmer spectrum, we obtain a continuous map
\[ \Spec^{\inv}(L_A)\to \Spec(\Sp_A^c). \]
Proposition \ref{prop:supportgeom} makes it clear that this map sends $I^A_{B,n}$ to $\lennart{P^A_{B,n}}$ and is hence bijective. Therefore we can conclude from the results of this paper that the topology on $\Spec(\Sp_A^c)$ is at least as coarse as the one on $\Spec^{\inv}(L_A)$. In other words, our proof of the existence of inclusions $I_{B,n}\subseteq I_{B',n'}$ gives another proof of the analogous inclusion $\lennart{P^A_{B',n'}\subseteq P^A_{B,n}}$ on the topological side. The fact that there are no further topological inclusions requires additional arguments, namely the existence of compact $A$-spectra with `maximal type shifting behaviour'. See \cite[Section 4]{BarthelHausmannNaumannNikolausNoelStapleton} or \cite[Section 7]{KuhnLloydChromatic}. Knowing this, we see that $X\mapsto \supp((\underline{MU}_A)_* X)$ is a universal support theory on compact $A$-spectra.

It remains to give the proof of Proposition \ref{prop:supportgeom}, which we will do in three steps:

$(i) \Rightarrow (ii)$: Since the $L_A$-action on $(MU_B)_*(X)$ factors through $\res^A_B\colon L_A\to L_B$ and $I^A_{B,n}=(\res^A_B)^{-1})(I^B_{B,n})$, it follows that we have an isomorphism $((MU_B)_*(X))_{I^A_{B,n}}\cong ((MU_B)_*(X))_{I^B_{B,n}}$. In particular the vanishing of the entire Mackey-functor $(\underline{MU}_*(X))_{I^A_{B,n}}$ implies the vanishing at the subgroup $B$ as a special case.

$(ii)\Longleftrightarrow (iii)$: Note that $I^B_{B,n}$ contains none of the non-trivial Euler classes for $B$, hence the non-trivial Euler classes act invertibly on $((MU_B)_*(X))_{I^B_{B,n}}$. It follows that we have an isomorphism
\[ ((MU_B)_*(X))_{I^B_{B,n}}\cong ((\Phi^BMU_B)_*\Phi^B X)_{I^B_{B,n}} \cong \Phi^B((MU_B)_{I^B_{B,n}})_*\Phi^B X. \]
Modulo $I_n$, the ring $\Phi^B((MU_B)_{I^B_{B,n}})_*$ embeds into the field of fractions of $\Phi^BL/I_n$, which is non-trivial. Moreover, $v_n$ is not contained in $I^B_{B,n}$ and hence is invertible in the localization. It follows that the localization $(\Phi^BMU_B)_{I^B_{B,n}}$ is an $MU$-algebra of height $n$, i.e., its vanishing detects compact spectra of type $\geq n+1$.

$(iii)\Rightarrow (i)$: Let $X$ be a compact $A$-spectrum such that $\Phi^B(X)$ is of type $\revm{>}n$. By the previous paragraph we know that $((MU_B)_*(X))_{I^A_{B,n}}=0$, and we have to show that $((MU_{\widetilde{B}})_*(X))_{I^{A}_{B,n}}=0$ for all other closed subgroups $\widetilde{B}$, too. We first assume that $\widetilde{B}$ does not contain $B$. Then there exists a character $V\in A^*$ which restricts to the trivial character for $\widetilde{B}$ but to a non-trivial character for $B$. It follows that $e_V$ is not contained in $I^A_{B,n}$ and hence acts invertibly on $(MU_{\widetilde{B}})_{I^{A}_{B,n}}$. On the other hand $e_V$ restricts to $0$ in $MU_{\widetilde{B}}$, so it follows that $(MU_{\widetilde{B}})_{I^{A}_{B,n}}=0$ and in particular $((MU_{\widetilde{B}})_*(X))_{I^{A}_{B,n}}=0$, as desired. Hence we can assume that $\widetilde{B}$ contains $B$ as a subgroup. Since the statement then no longer depends on the ambient group, we can reduce to the case $\widetilde{B}=A$.

Hence we need to show that $(MU_{A})_*(X))_{I^{A}_{B,n}}=0$. By induction on the pair (dimension, cardinality of $\pi_0(A/B)$) it follows that all the localizations at smaller intermediate groups $B\subseteq \widetilde{B} \subseteq A$ vanish, and hence the homotopy groups of $(MU_A)_{I^A_{B,n}}\wedge X$ are concentrated at $A$. In particular this implies that the map $((MU_A)_{I^A_{B,n}})_* X\to ((\Phi^AMU_A)_{I^A_{B,n}})_* \Phi^AX$ is an isomorphism and all Euler classes $e_V$ for non-trivial characters act invertibly on $((MU_A)_{I^A_{B,n}})_* X$. Now, if $V$ restricts to the trivial character in $B^*$ (such a $V$ always exists since we can assume that $B$ is a proper subgroup of $A$), its Euler class $e_V$ lies in the maximal ideal $I^A_{B,n}$ of $((MU_A)_{I^A_{B,n}})_*$.

If we knew that $((MU_A)_{I^A_{B,n}})_* X$ is finitely generated over $(MU_A)_*$ we could apply Nakayama's lemma to see directly that $((MU_A)_{I^A_{B,n}})_* X=0$, as desired. Since we do not know this, we have to argue differently:  By Corollary \ref{cor:blueshift}, we know that $\Phi^A ((MU_A)_{I^A_{B,n}})_*$ is trivial if $\pi_0(A/B)$ is not a $p$-group, and is of height $n-\rank_p(\pi_0(A/B))$ otherwise (where negative heights again mean that the theory is trivial). Hence, what we want to show is that if $\Phi^BX$ is of type $\geq n+1$, then $\Phi^A X$ is of type $\geq n+1-\rank_p(\pi_0(A/B))$. Indeed, this implies that $((MU_A)_{I^A_{B,n}})_* X\cong ((\Phi^AMU_A)_{I^A_{B,n}})_* \Phi^AX$ is trivial.

This precise statement about $\Phi^AX$ is one of the main results of \cite{BarthelGreenleesHausmann}, building on \cite{BalmerSanders} and \cite{BarthelHausmannNaumannNikolausNoelStapleton}. Hence, using these results, Proposition \ref{prop:supportgeom} follows. Alternatively, rather than importing we can reprove the above statement using the methods from this paper. Note that by induction on the rank of $A/B$ and by replacing $X$ by the compact $A/B$-spectrum $\underline{\Phi^B} X$ it suffices to show two special cases:
\begin{enumerate}
    \item[(1)] If $X$ is a compact $C_{p^k}$-spectrum of underlying type $\geq n+1$, then the type of $\Phi^{C_{p^k}}X$ is at least $n$.
    \item[(2)] If $X$ is a compact $\T$-spectrum of underlying type $\geq n+1$, then the type of $\Phi^{\T}X$ is also $\geq n+1$.
\end{enumerate}

For (1) it suffices to find a complex oriented theory $E$ of height $n$ \lennart{(meaning here that $E_*$ vanishes on a compact spectrum $X$ if and only if $X$ is of type $\geq n+1$)} such that $\Phi^{C_{p^k}}\underline{E}$ is of height $\geq n-1$, where $\underline{E} = F(EG_+,E)$ denotes the Borel theory associated to $E$ (see the proofs of \cite[Corollary 3.12]{BarthelHausmannNaumannNikolausNoelStapleton} or \cite[Proposition 6.10]{BarthelGreenleesHausmann} for details on this argument).  In \cite{BarthelHausmannNaumannNikolausNoelStapleton} it was shown that Morava $E$-theory $E=E_n$ has this property, building on results of Hopkins-Kuhn-Ravenel on the $p$-disivible group associated to $E_n$ \cite{HKR00} and extending earlier work of Greenlees, Hovey and Sadofsky. Using Remark \ref{rem:generators} we obtain similar results more generally:
\begin{prop} \label{prop:blueshift} Let $E$ be any complex oriented ring spectrum of height $n$ which is Landweber exact over $MU/I_{n-1}=MU/(v_0,\hdots,v_{n-2})$, i.e., $I_{n-1}$ acts trivially on $\pi_* E$, $v_{n-1}\in \pi_* E$ is a regular element and $v_n$ is a unit modulo $v_{n-1}$. Then $\Phi^{C_{p^k}}\underline{E}$ is of height $n-1$.
\end{prop}
\begin{proof} We apply Remark \ref{rem:generators} and check that its assumptions are satisfied. We have $(\underline{E}^{\T})_*=E_*\llbracket e\rrbracket$, with Euler classes $e_n$ given by the $n$-series $[n]_F(e)$ for the formal group law associated to $E$. If $n$ is a power of $p$, the leading term of this Euler class is a power of $v_{n-1}$, which we assumed to be regular. Modulo $v_{n-1}$, the leading term becomes $v_n$ which is a unit since $E/v_{n-1}$ is of height $n$. Hence by \cref{rem:generators} we find that the pushforward of the element $\psi_{p^k}^{(n-1)}$ to $E_*\llbracket e\rrbracket $, i.e., the element $\sum_{i=0}^{\infty}a_i[p^{k-1}]_F(e)$ for $\sum_{i=0}^{\infty}a_ie^i = [p]_F(e)/e^{p^{n-1}}$, generates the kernel of the composite
\[ E_*\llbracket e\rrbracket\to E_*\llbracket e\rrbracket/[p^k]_F(e)=(\underline{E}^{C_{p^k}})_*\to (\Phi^{C{p^k}} \underline{E})_*.\]
The leading term of $\psi_{p^k}^{(n-1)}$ equals $v_{n-1}$, which is not a unit. Hence $\Phi^{C{p^k}} \underline{E}$ is non-trivial and since $I_{n-1}$ acts trivially it must be of height $\geq n-1$. Furthermore, reducing modulo $v_{n-1}$ the leading coefficient of $\psi_{p^k}^{(n-1)}$ equals a power of $v_n$.  Since $v_n$ is a unit modulo $v_{n-1}$, it follows that so is $\psi_{p^k}^{(n-1)}$. Therefore $(\Phi^{C{p^k}}\underline{E})_*/v_{n-1}$ is trivial as an algebra over $E_*\llbracket e\rrbracket/(v_{n-1},\psi_{p^k}^{(n-1)})=0$, and hence $\Phi^{C{p^k}} \underline{E}$ is of height exactly $n-1$.
\end{proof}
\begin{example}
    Given any Landweber exact complex oriented ring spectrum $E$ of height $n$, its quotient $E/I_{n-1}$ satisfies the assumptions of the previous proposition. It follows that also in this case the height of $\Phi^{C_{p^k}}E$ equals $n-1$. For example this applies to Johnson-Wilson spectra, to $MU[v_{n}^{- 1}]$ or to Morava $E$-theories.
\end{example} 

Similarly, for (2) one needs to find a complex oriented theory $E$ of height $n$ such that $\Phi^{\T}\underline{E}$ is also of height $n$. This is more elementary and satisfied by any $p$-local complex oriented theory $E$ of height $n$. To see this, note again that the Euler classes $e_n$ are given by the $n$-series $[n]_F(e)\in E_*\llbracket e\rrbracket $. Writing $n=p^k\cdot m$ with $m$ coprime to $p$, we find that $[n]_F(e)$ is a unit multiple of $[p^k]_F(e)$. Modulo $I_n$, for $k>0$ the leading term of $[p^k]_F(e)$ is a power of $v_n$, which by assumption is a unit in $E_*/I_n$. It follows that, modulo $I_n$, the coefficients of $\Phi^{\T} \underline{E}$ are given by $E_*/I_n ((e))$, which is always non-trivial when $E_*/I_n$ is.\newline

\subsection{Change of groups and the structure of $\MM_{FG}^A$}
For any $A$-spectrum $X$, the groups $MU_{2*}^A(X)$ come equipped with the structure of a graded $(L_A, S_A)$-comodule. Since the stack $\MM_{FG}^A$ is the stack associated to this graded Hopf algebroid, we obtain an associated quasi-coherent sheaf $\FF^A(X)$ on $\MM_{FG}^A$ (see \cite[Proposition 4.3]{MeierOzornova}). This is the $0$-th graded piece of a $\QCoh(\MM_{FG}^A)$-valued homology theory on $A$-spectra, whose $i$-th piece $\FF_i^A(X)$ is given by $MU_{2*+i}^A(X)$. We end this section by a closer look upon how the structure of $\MM_{FG}^A$ relates to this homology theory. 

    Recall from \cref{prop:OpenClosedSubstacks} that for a closed subgroup $B\subseteq A$, there is an open immersion $j\colon \MM_{FG}^{A/B} \to \MM_{FG}^A$ and a closed immersion $i\colon \MM_{FG}^B \to \MM_{FG}^A$. We obtain corresponding adjunctions
\begin{equation*}
\begin{tikzcd}
\QCoh(\MM_{FG}^A) \arrow[r,shift left=.5ex,"j^*"]
&
\QCoh(\MM_{FG}^{A/B}) \arrow[l,shift left=.5ex,"j_*"]
\end{tikzcd}
\text{ and }
\begin{tikzcd}
\QCoh(\MM_{FG}^A) \arrow[r,shift left=.5ex,"i^*"]
&
\QCoh(\MM_{FG}^{B}). \arrow[l,shift left=.5ex,"i_*"]
\end{tikzcd}
\end{equation*}
Believing that the structure of $\MM_{FG}^A$ dictates the structure of the $\infty$-category $\Sp_A$ of $A$-spectra, we expect a relation to the adjunctions
\begin{equation*}
\begin{tikzcd}
\Sp_A \arrow[rr,shift left=.5ex,"\underline{\Phi}^B"]
&&
\Sp_{A/B} \arrow[ll,shift left=.5ex,"P^*_{A/B}"]
\end{tikzcd}
\quad\qquad\text{ and }\qquad\quad
\begin{tikzcd}
\Sp_A \arrow[rr,shift left=.5ex,"\res_B^A"]
&&
\Sp_B. \arrow[ll,shift left=.5ex,"\coind_B^A"]
\end{tikzcd}
\end{equation*}
Here, denoting by $\FF[B]$ the family of subgroups of $A$ not containing $B$ and by $q\colon A \to A/B$ the projection, we define $\underline{\Phi}^B(X) = (\widetilde{E}\FF[B] \sm X)^B$ and $P^*_{A/B}(Y) = \widetilde{E}\FF[B] \sm q^*X$. Note that the definition of $\underline{\Phi}^B$ is made so that its underlying spectrum is $\Phi^B$ and more generally $\Phi^{C/B}\underline{\Phi}^B \simeq \Phi^C$ for every $B\subseteq C\subseteq A$. For more details on the first adjunction, see \cite[Section II.9]{LewisMaySteinberger}, \cite[Section 4.1]{HillPrimer} and \cite[Section 2.2]{MeierShiZeng}. For the benefit of the reader, we give a brief sketch of its basic properties: The adjunction between $q^*$ and $(-)^H$ induces maps 
\begin{align*}\id \to \underline{\Phi}^BP^*_{A/B} \qquad &\text{and}\qquad P^*_{A/B}\underline{\Phi}^B \to \widetilde{E}\FF[B] \sm (-),
\intertext{which can be checked to be equivalences on geometric fixed points. The inverses}
\underline{\Phi}^BP^*_{A/B} \xrightarrow{\simeq} \id\qquad &\text{and} \qquad\id \to \widetilde{E}\FF[B] \sm (-) \xrightarrow{\simeq} P^*_{A/B}\underline{\Phi}^B\end{align*}
form the counit and unit of the adjunction. In particular, $P^*_{A/B}$ is fully faithful and its image agrees with that of $\widetilde{E}\FF[B] \sm$. Since smashing with $\widetilde{E}\FF[H]$ is idempotent and hence symmetric monoidal, $P^*_{A/B}$ is symmetric monoidal as well and so is $\underline{\Phi}^B$. 

\begin{prop}\label{prop:FACompatibility} \revm{For every abelian compact Lie group $A$ and closed subgroup $B\subseteq A$, the} diagram 
\[
\begin{tikzcd}
\Sp_A \arrow[r,shift left=.5ex,"\underline{\Phi}^B"]\arrow[d, "\FF^A"]
&
\Sp_{A/B} \arrow[l,shift left=.5ex,"P^*_{A/B}"]\arrow[d, "\FF^{A/B}"]
\\
\QCoh(\MM_{FG}^A) \arrow[r,shift left=.5ex,"j^*"]
&
\QCoh(\MM_{FG}^{A/B}) \arrow[l,shift left=.5ex,"j_*"]
\end{tikzcd} \]
commutes, i.e.\ there are 
 natural isomorphisms 
$j^*\FF^A(X) \cong \FF^{A/B}(\underline{\Phi}^BX)$ for $X\in \Sp_A$ and $j_*\FF^{A/B}(Y) \cong \FF^A(P^*_{A/B}Y)$ for $Y\in \Sp_{A/B}$.

Moreover, there is a natural isomorphism
$i_*\FF^B(Z) \cong \FF^A(\coind_B^AZ)$ for $Z\in \Sp_B$.
\end{prop}
\noindent No such isomorphism can be expected for $i^*$ and $\res^A_B$ in general. One reason is that $i^*$ is not flat, but even a spectral sequence relating $i^*$ and $\res^A_B$ seems not to exist for $A$ not a torus for reasons related to \cref{rem:counterexample}. 

Before we prove the proposition, we need a lemma, in which we will use the Hopf algebroid $(\underline{\Phi}^BL_A, \underline{\Phi}^BS_A)$. Here $\underline{\Phi}^BL_A$ is obtained from $L_A$ by inverting $e_V$ for all $V\notin \im((A/B)^*\to A^*)$, and the ring $\underline{\Phi}^BS_A$ is defined as $\underline{\Phi}^BL_A \tensor_{L_A} S_A \tensor_{L_A} \underline{\Phi}^BL_A$. This classifies strict isomorphism between equivariant formal group laws where the relevant Euler classes are invertible on source \emph{and} target. Since the invertibility of Euler classes only depends on the underlying equivariant formal group and not on the choice of coordinate, this simplifies to $S_A \tensor_{L_A} \underline{\Phi}^BL_A$.
\begin{lemma}\label{lem:LandweberPhi}
    For $Y\in \Sp_B$ and $q\colon A \to A/B$ the projection, there are natural isomorphisms 
    \[\underline{\Phi}^BL_A \tensor_{L_A} (MU_A)_*(q^*Y) \xrightarrow{\cong} \underline{\Phi}^BL_A \tensor_{L_A} (MU_A)_*(P^*_{A/B}Y) \xrightarrow{\cong} (\underline{\Phi}^BMU_A)_*(Y) \] 
    and 
     \[\underline{\Phi}^BL_A \tensor_{L_{A/B}} (MU_{A/B})_*(Y) \xrightarrow{\cong} (\underline{\Phi}^BMU_A)_*(Y) \] 
     of graded $(\underline{\Phi}^BL_A, \underline{\Phi}^BS_A)$-comodules, where the map $L_{A/B} \to \underline{\Phi}^BL_A$ is defined as the composite $L_{A/B}\xrightarrow{q^*} L_A \to \underline{\Phi}^BL_A$. 
\end{lemma}
\begin{proof}
    A model for $\widetilde{E}\FF[B]$ is given by $S^{\infty W}$ for $W$ the sum of all characters $V\in \VV$, for $\VV$ the set of characters $V$  not restricting to $1$ in $B$ or, equivalently, $V\notin \im((A/B)^*\to A^*)$. Indeed, $W^C = 0$ for $C\subseteq A$ if and only if $B\subseteq C$, as this is equivalent to none of the $V\in \VV$ restricting to $1$ in $C$. 

    In other words, smashing with $\widetilde{E}\FF[B]$ is the same as inverting all the maps $S^0 \to S^V$ for $V\in \VV$. For an $MU_A$-module, this is equivalent to inverting $e_V$ for $V\in \VV$. 

    We have
    \begin{align*}
        \underline{\Phi}^BL_A \tensor_{L_A} (MU_A)_*(q^*Y) &\cong (MU_A)_*(q^*Y)[e_V^{-1}: V\in \VV] \\
        &\cong \pi_*(MU_A \wedge q^*Y \wedge \widetilde{E}\FF[B])^A \\
        &\cong \pi_*(\underline{\Phi}^B(MU_A \wedge P^*_{A/B}Y))^{A/B} \\
        &\cong \pi_*(\underline{\Phi}^BMU_A \wedge Y)^{A/B} \\
        &\cong (\underline{\Phi}^BMU_A)_*(Y).
    \end{align*}
Replacing $q^*Y$ by $P^*_{A/B}Y$ in the chain of isomorphisms above yields a similar chain of isomorphisms. All the isomorphisms are isomorphisms of comodules since the isomorphisms are natural in the $MU_A$-variable and we can plug into this variable the left and right unit $MU_A \to MU_A \wedge MU_A$. 

    To construct the second isomorphism, note that as part of the global structure of equivariant $MU$, there is a ring map $q^*MU_{A/B} \to MU_A$ (cf.\ \cite{LinskensNardinPol}). Applying $\underline{\Phi}^B = (- \wedge \widetilde{E}\FF[H])^B$ yields a ring map $MU_{A/B} \to \underline{\Phi}^B MU_A$. This induces a morphism 
       \[\underline{\Phi}^BL_A \tensor_{L_{A/B}} (MU_{A/B})^*(Y) \to (\underline{\Phi}^BMU_A)^*(Y). \]
It is enough to show that this is an isomorphism for finite $Y$ and hence for $Y = (A/B)/(B'/B)_+ \cong A/B'_+$ for $B\subseteq B'\subseteq A$. In this case, the map becomes 
\begin{equation}\label{eq:PhiMorphism} \underline{\Phi}^BL_A\tensor_{L_{A/B}} L_{B'/B} \to (\underline{\Phi}^BMU_A)^*(A/B'_+).\end{equation}
The natural map 
\[L_A \tensor_{L_{A/B}}L_{B'/B} \to L_{B'}\]
is an isomorphism since $L_A \to L_{B'}$ is a surjection with kernel generated by the Euler classes $e_V$ for those $V\in A^*$ restricting trivially to $B'$; these are exactly the images of the Euler classes $e_W$ for those $W\in (A/B)^*$ restricting trivially to $B'/B$. 

Thus, \eqref{eq:PhiMorphism} becomes 
\[\underline{\Phi}^BL_A\tensor_{L_A}MU_A^*(A/B'_+)\cong \underline{\Phi}^BL_A\tensor_{L_A}L_{B'} \to (\underline{\Phi}^BMU_A)^*(A/B'_+),\]
which is an isomorphism by the first part. Similar to the first part, all isomorphisms are isomorphisms of comodules again.
\end{proof}

\begin{proof}[Proof of \cref{prop:FACompatibility}: ]
   We establish first the isomorphism $j^*\FF^A(X) \cong \FF^{A/B}(\underline{\Phi}^BX)$ for $A$-spectra $X$. Consider the commutative diagram
\[
\xymatrix{
\Spec L_A  \ar[d]&\Spec \underline{\Phi}^BL_A\ar[rr]\ar[dr]^{\varphi}\ar[l] &&\Spec L_{A/B} \ar[dl] \\
\MM_{FG}^A\ar@{}[ur]|{\text{\pigpenfont L}}&&\MM_{FG}^{A/B}\ar[ll]^j,
}
\]
where the down-right arrow comes from applying $q_*$ to an $A$-equivariant formal group classified by a morphism to $\Spec \underline{\Phi}^BL_A$, and the right-pointing horizontal morphisms come from the composition $L_{A/B} \xrightarrow{q^*} L_A \to \underline{\Phi}^BL_A$. The square is a pullback square by \cref{prop:OpenClosedSubstacks}. Thus $\varphi$ is faithfully flat and hence $\varphi^*$ induces an equivalence of $\QCoh(\MM_{FG}^{A/B})$ to graded $(\underline{\Phi}^BL_A, \underline{\Phi}^BS_A)$-comodules. 

The comodule corresponding to $j^*\FF_X^A$ is $\underline{\Phi}^BL_A\tensor_{L_A}(MU_A)_{2*}X$. As in (the proof of) the first isomorphism in \cref{lem:LandweberPhi}, we observe that this is isomorphic to 
\begin{align*}\underline{\Phi}^BL_A\tensor_{L_A}(MU_A)_{2*}(X \wedge \widetilde{E}\FF[B]) &\cong \underline{\Phi}^BL_A\tensor_{L_A}(MU_A)_{2*}(P^*_{A/B}\underline{\Phi}^BX)\\ &\cong  (\underline{\Phi}^BMU_A)_{2*}(\underline{\Phi}^BX).\end{align*}
By the second isomorphism in \cref{lem:LandweberPhi}, this is isomorphic to the comodule corresponding to  $\FF_{\underline{\Phi}^BX}^{A/B}$, i.e.\ to  
$\underline{\Phi}^BL_A \tensor_{L_{A/B}} (MU_{A/B})_{2*}(\underline{\Phi}^BX)$. This establishes the first claimed isomorphism of sheaves. 

Since  the counit $\underline{\Phi}^BP^*_{A/B}\to \id_{\Sp_{A/B}}$ is an equivalence, we can assume for the proof of $j_*\FF^{A/B}(Y) \cong \FF^A(P^*_{A/B}Y)$ for $Y\in \Sp_{A/B}$ that $Y = \underline{\Phi}^BX$ for some $X\in \Sp_A$ and we obtain from the first isomorphism in this case a natural isomorphism 
\[j^*j_*\FF^{A/B}(Y)\cong j^*j_*j^*\FF^A(X) \cong j^*\FF^A(X) \cong \FF^{A/B}(Y) \cong j^*\FF^A(P^*_{A/B}Y).\] 
The Euler classes for characters $V \notin \im((A/B)^* \to A^*)$ act invertibly on $\FF^A(P^*_{A/B}Y)$, and \cref{prop:OpenClosedSubstacks} implies that $\FF^A(P^*_{A/B}Y)$ is in the image of $j_*$. Since the counit \[j^*j_*\to \id_{\QCoh(\MM_{FG}^{A/B})}\] is an isomorphism, the claimed isomorphism follows. 

Lastly, we show $i_*\FF^B(Z) \cong \FF^A(\coind_B^AZ)$ for $Z\in \Sp_B$. This follows from the chain 
\begin{align*}(MU_A)_{2*}(\coind_B^A Z) &\cong \pi_{2*}^A(MU_A \tensor \coind_B^AZ) \\
&\cong \pi_{2*}^A (\coind_B^A (MU_B\tensor Z))\\
&\cong \pi_{2*}^B (MU_B\tensor Z) \\
&\cong MU_{2*}^B(Z).\tag*{\qedhere}
\end{align*}
of isomorphisms of $(L_A, S_A)$-comodules.
\end{proof}

We end this section by connecting the stacky point of view with support theories. For a compact $A$-spectrum $X$, we may consider the support $\supp \FF^A_i(X) \subseteq |\MM_{FG}^A|$, corresponding to those $x\colon \Spec k \to \MM_{FG}^A$ such that $x^*\FF^A_i(X)$ is non-trivial. As remarked above, in general the sheaves $\FF_i^A(X)$ do not play well with restrictions. Thus, for an $A$-spectrum $X$, one should consider all $\FF_i^B(\res^A_B X) \in \QCoh(\MM_{FG}^B)$ for $B\subseteq A$ in conjunction. Using the closed embeddings $\MM_{FG}^B \hookrightarrow \MM_{FG}^A$, the correct notion of support of $\FF_*^A(X)$ becomes thus $\supp \underline{\FF}_*(X) = \bigcup_{B\subseteq A} \supp \FF_*^B(\res_B^A(X))$. 

\begin{prop}\label{prop:universalsupportMFG}
    For any compact $A$-spectrum $X$, the subsets $\supp \underline{\FF}_*(X) \subseteq |\MM_{FG}^A|$ and  $\supp((\underline{MU}_A)_* X)\subseteq \Spec^{\inv}(L_A)$ correspond to each other under the bijection from \cref{thm:ClassificationOfInvariantPrimes}. Thus, $\supp \underline{\FF}_*(-)$ is a universal support theory as well. 
\end{prop}
\begin{proof}
 Let a point $[x\colon \Spec k \to \MM_{FG}^A] \in |\MM_{FG}^A|$ correspond to a pair $(B,n)$. Then $x$ factors as $\overline{x}\colon \Spec k \to \MM_{FG} = \MM_{FG}^{B/B}$ composed with $\MM_{FG}^{B/B} \xhookrightarrow{j} \MM_{FG}^B \xhookrightarrow{i}\MM_{FG}^A$ by \cref{prop:OpenClosedSubstacks} and \cref{rem:IdentifyPointsAlongOpenImmersion}.  By \cref{prop:FACompatibility}, the pullback functor $j^*$ corresponds to $\underline{\Phi}^B$. Thus, $x\in \supp \FF_*^B(\res^A_BX)$ if and only if $\overline{x}^*\FF_*^{\{1\}}(\Phi^BX)\neq 0$, i.e.\ if $\Phi^BX$ has chromatic type at least $n$. We have seen in \cref{prop:supportgeom}, this happens if and only if $I^{\lennart{A}}_{B,n} \in \supp((\underline{MU}_A)_* X)$. 

 Now suppose $I^{\lennart{A}}_{B,n} \notin \supp((\underline{MU}_A)_* X)$ and let $B \subseteq \widetilde{B} \subseteq A$. Thus $(MU_{\widetilde{B}})_*(\res_{\widetilde{B}}^A(X))_{I^{\widetilde{B}}_{B,n}}$ vanishes. 
 Since $[x] \in |\MM_{FG}^{\widetilde{B}}|$ is the image of the prime ideal $I^{\widetilde{B}}_{B,n}\in \Spec L_{\widetilde{B}} \cong \Spec (MU_{\widetilde{B}})_*$, we see by \cref{lem:EFGEFGL} that $x\colon \Spec k \to \MM_{FG}^{\widetilde{B}}$ factors over $\widetilde{x}\colon \Spec k \to \Spec (L_{\widetilde{B}})_{I^{\widetilde{B}}_{B,n}}$. 
 Thus, we can identify $x^*\FF^{\widetilde{B}}_*(X)$ with $\widetilde{x}^*(MU_{\widetilde{B}})_*(\res_{\widetilde{B}}^A(X))_{I^{\widetilde{B}}_{B,n}} = 0$, and $x\notin \supp \underline{\FF}_*(X)$. 
\end{proof}

\bibliographystyle{alpha}
\bibliography{bibliography}
\end{document}